\documentclass[11pt]{article}

\usepackage[english]{babel}
\usepackage{latexsym,dsfont}
\usepackage{fullpage}
\usepackage[T1]{fontenc}
\usepackage{amsmath}
\usepackage{amsfonts}
\usepackage{amssymb}
\usepackage{mathrsfs}
\usepackage{amsthm}
\usepackage{latexsym}
\usepackage{array}
\usepackage{mathrsfs}
\usepackage{amsxtra}
\usepackage{amscd}
\usepackage{mathtools}
\usepackage{verbatim}
\usepackage{stmaryrd}
\usepackage{enumerate}
\usepackage{frcursive}
\usepackage{hyperref}
\hypersetup{
    colorlinks=true,
    linkcolor=blue,
	linktocpage=true
}

\usepackage[toc,page]{appendix}

\usepackage[usenames]{color}
\definecolor{red}{rgb}{1.0,0.0,0.0}

\definecolor{blu}{rgb}{0.0,0.0,1.0}

\definecolor{gre}{rgb}{0.03,0.50,0.03}

\definecolor{amethyst}{rgb}{0.6, 0.4, 0.8}

\definecolor{blue-violet}{rgb}{0.54, 0.17, 0.89}

\definecolor{darkviolet}{rgb}{0.58, 0.0, 0.83}

\definecolor{mrc}{rgb}{0.8, 0, 0.4}

\def\Swiech
{{\accent1 S}wie{\hbox{\kern -0.21em\lower
0.77ex\hbox{$\textfont1=\scriptfont1
\lhook$}}}ch}

\def\SWIECH
{{\accent1 S}WIE{\hbox{\kern -0.21em\lower
0.77ex\hbox{$\textfont1=\scriptfont1
\lhook$}}}CH}

\numberwithin{equation}{section}

\theoremstyle{plain}
\newtheorem{Theorem}{Theorem}[section]
\newtheorem{Definition}[Theorem]{Definition}
\newtheorem{Proposition}[Theorem]{Proposition}
\newtheorem{Lemma}[Theorem]{Lemma}

\newtheorem{Corollary}[Theorem]{Corollary}
\newtheorem{Remark}[Theorem]{Remark}
\newtheorem{Example}[Theorem]{Example}

\newtheorem{innerAssumption}{Assumption}
\newenvironment{Assumption}[1]
{\renewcommand\theinnerAssumption{#1}\innerAssumption}
{\endinnerAssumption}
\renewcommand{\epsilon}{\varepsilon}

\newtheorem{innerStandingAssumption}{Standing Assumption}
\newenvironment{StandingAssumption}[1]
{\renewcommand\theinnerStandingAssumption{#1}\innerStandingAssumption}
{\endinnerStandingAssumption}

\newtheorem{innerNotation}{Notation}
\newenvironment{Notation}[1]
{\renewcommand\theinnerNotation{#1}\innerNotation}
{\endinnerNotation}

\newcommand{\conv}[1]{\stackrel{}{\star_{#1}}}
\newcommand{\sconv}[1]{\stackrel{\scriptscriptstyle dB}{\star_{#1}}}
\newcommand{\wt}{\widetilde}

\def\<{\left\langle }
\def\>{\right\rangle }

\def \S{\mathbf{S}}

\def \E{\mathbb{E}}
\def \F{\mathbb{F}}
\def \N{\mathbb{N}}
\def \P{\mathbb{P}}
\def \R{\mathbb{R}}

\def \Bc{{\cal B}}
\def \Dc{{\cal D}}
\def \Fc{{\cal F}}
\def \Gc{{\cal G}}
\def \Hc{{\cal H}}
\def \Lc{{\cal L}}
\def \Mc{{\cal M}}
\def \Nc{{\cal N}}
\def \Pc{{\cal P}}
\def \Uc{{\cal U}}
\def \Wc{{\cal W}}

\def \eps{\varepsilon}

\def \ni{\noindent}

\def \beqs{\begin{eqnarray*}}
\def \enqs{\end{eqnarray*}}
\def \beq{\begin{eqnarray}}
\def \enq{\end{eqnarray}}

\def \Tr{{\rm Tr}}
\def \Ur{{\rm U}}

\DeclareMathOperator*{\esssup}{ess\,sup}

\begin{document}

\title{Optimal control of path-dependent \\
McKean-Vlasov SDEs in infinite dimension}

\author{
Andrea COSSO\footnote{University of Bologna, Italy; andrea.cosso@unibo.it} \quad
Fausto GOZZI\footnote{Luiss University, Roma, Italy; fgozzi@luiss.it} \quad
Idris KHARROUBI\footnote{LPSM, UMR CNRS 8001, Sorbonne University and Universit\'e de Paris; idris.kharroubi@upmc.fr} \quad \\
Huy\^en PHAM\footnote{LPSM, UMR CNRS 8001, Sorbonne University and Universit\'e de Paris; pham@lpsm.paris} \quad
Mauro ROSESTOLATO \footnote{Dipartimento di Matematica e Fisica ``Ennio De Giorgi'', Universit\`a del Salento, 73100 Lecce, Italy; mauro.rosestolato@unisalento.it.}
}

\date{To appear in {\it Annals of Applied Probability}}

\maketitle

\begin{abstract}
We study the optimal control of  path-dependent McKean-Vlasov equations valued in Hilbert spaces motivated by non Markovian mean-field models driven by stochastic PDEs.
We first establish the well-posedness of the state equation,  and then we prove the dynamic programming principle (DPP) in such a general framework.  The crucial law invariance property of the value function $V$  is rigorously obtained,  which  means that $V$ can be viewed as a function on the Wasserstein space of probability measures on the set of continuous functions valued in Hilbert space.
We then define a notion of pathwise measure derivative,  which extends  the Wasserstein derivative due to  Lions \cite{lio12}, and  prove a related functional It\^o formula  in the spirit of
Dupire \cite{dup09} and Wu and Zhang \cite{WuZhang18PPDE}.  The Master Bellman equation is derived from the DPP by means of a suitable notion of viscosity solution. We provide different formulations and simplifications of such a Bellman equation notably in the special case when there is no dependence on the law of the control.
\end{abstract}



\newpage

\tableofcontents

\newpage
\section{Introduction}

Given two real separable Hilbert spaces $H$ and $K$, let us consider  the nonlinear stochastic differential equation (SDE)  on $H$ in the form:
\begin{align} \label{SDEintro}
dX_t =  A X_t dt +  b_t(X,\P_{X_{\cdot\wedge t}},\alpha_t,\P_{\alpha_t}) dt + \sigma_t(X,\P_{X_{\cdot\wedge t}},\alpha_t,\P_{\alpha_t}) dB_t,
\end{align}
over a finite interval $[0,T]$.  Here,  $A\colon \mathcal{D}(A)\subset H\rightarrow H$ is
the generator of a $C_0$-semigroup of contractions in $H$,
and $B=(B_t)_{t\geq 0}$
is a $K$-valued cylindrical Brownian motion on a  complete probability space $(\Omega,\Fc,\P)$ with $\F^B$ its completed natural filtration.
The coefficients  $b$ and $\sigma$,  valued respectively in $H$ and $\Lc_2(K;H)$ (the space of Hilbert-Schmidt operators from $K$ to $H$), depend on time, on the whole path of the state process $X$, on an input
control process $\alpha$, that is an $\F^B$-adapted process valued in  some Borel space $U$,  and furthermore on the distribution of the state/control  process.

Equation \eqref{SDEintro} is referred to as controlled McKean-Vlasov SDE in Hilbert spaces, and we are interested in the optimal control for \eqref{SDEintro} by minimizing, over control processes $\alpha$, a functional in the form
\begin{equation*}
J(X_0,\alpha) =  \E \Big[ \int_0^T f_t(X,\P_{X_{\cdot\wedge t}},\alpha_t,\P_{\alpha_t}) dt + g(X,\P_{X}) \Big],
\end{equation*}
given running cost and terminal cost functions $f$ and $g$.

When the coefficients $b,\sigma,f,g$ do not depend on the law of the state process, the control of equation \eqref{SDEintro} is
motivated by various kinds of stochastic partial differential equations (SPDEs) like
stochastic heat equations (see e.g. \cite{CannarsaDaPrato91,Gozzi95,
Gozzi96,Masiero05,FuhrmanTessitore02-ann})
stochastic reaction-diffusion equations (see e.g.\cite{Cerrai01,Cerrai01-40}), stochastic porous media equations (see e.g. \cite{barroc19}), singular stochastic dissipative equations (see e.g. \cite{daproc02}),
stochastic Burgers and Navier-Stokes equations (see e.g. \cite{DaPratoDebussche99,dapdeb00,GozziSritharanSwiech05}),
Zakai equation in filtering (see e.g. \cite{GozziSwiech00}),
stochastic first-order equations (see e.g. \cite{GoldysGozzi06}),
stochastic delay equations (see e.g. \cite{FuhrmanMasieroTessitore11}\cite{RosestolatoSwiech15},
\cite{GozziMasiero17,GozziMasiero17bis,GozziMasieroErrata21}
\cite{BiffisGozziProsdocimi20}).  We refer also to the lecture notes and monographs \cite{DPZ}, \cite{bardaproc16}, \cite{FabbriGozziSwiech}, for an account on this topic.

The main novelty of this paper is to consider a mean-field dependence on the coefficients of the infinite-dimensional stochastic differential equation \eqref{SDEintro}, and to study the corresponding control problem. Mean-field diffusion processes, also called McKean-Vlasov equations,  in finite dimension have a long history with  the pioneering works \cite{kac}, \cite{mackean}, and later on with the seminal paper \cite{S89} in the framework of propagation of chaos. The control of such equations has attracted an increasing interest since the emergence of mean-field game theory initiated independently  in \cite{laslio07} and \cite{huacaimal06}, aiming at describing control of large systems or population of interacting particles, and  has generated over the last few years numerous contributions, see e.g.  \cite{phawei17}, \cite{CossoPham19}, \cite{DjetePossamaiTan19}, and the reference monographs  \cite{benetal13} and \cite{cardelbook}.

In addition to the infinite-dimensional feature of the McKean-Vlasov equation \eqref{SDEintro}, we emphasize the  path-dependency (in a nonanticipative way)
of the coefficients $b,\sigma,f,g$, on the state process as well as on its distribution.
This general setting, which is motivated by various applications, see e.g. Example \ref{ex:SE}, seems to be considered for the first time in the present paper.
It is worth noting that, also in the mean-field game theory, some research papers
(see \cite{CarmonaEtAl18,FouqueZhang18}) started to look at the cases when the state of the agents follows a delay equation. As far as we know, only cases with explicit solution are studied up to now, and we think it would be worth to build a general theory in such cases, on the line of what we do here for McKean-Vlasov control problems.

Our basic objective here is to extend to our infinite-dimensional path-dependent setting the tools required in the dynamic programming approach for McKean-Vlasov control problems.

%

In the finite-dimensional case, i.e. $H=\R^d$, and in the Markovian case, i.e. without path-dependency of the coefficients, the Wasserstein derivative in the lifted sense of Lions \cite{lio12}, turns out to be a convenient notion of measure derivative when combined with It\^o's formula along the flow of probability measures (see \cite{bucetal17}) in order to  define
the Master equation in mean-field game/control. These concepts have been recently extended to the path-dependent case in  \cite{WuZhang18PPDE} with a functional It\^o formula in the McKean-Vlasov setting.

\vspace{2mm}

\noindent {\bf Our contributions.}
Our first main result is to prove  the crucial law invariance property of the value function to the control problem (see Theorem \ref{T:id-law}), which implies that the value function can be considered  as a function on the Wasserstein space of probability measures on  $C([0,T];H)$. We also state and provide a direct proof of the dynamic programming principle  in this context (see Theorem \ref{T:DPP} and Corollary \ref{C:DPP}).  Next, we introduce a notion of pathwise derivative in Wasserstein space and a related functional It\^o formula in our infinite-dimensional McKean-Vlasov context that extend the concepts in \cite{WuZhang18PPDE}.  Equipped with these tools, we can then derive from the dynamic programming principle the associated Master HJB equation, which is a PDE
where the  state variable is a probability measure on $C([0,T];H)$.  For such PDE, we provide equivalent formulations and simplifications,
notably in the special case when there is no dependence of the coefficients on the law of the control. We  define an intrinsic notion of viscosity solution in $\Pc_2(C([0,T];H))$, the space of square-integrable probability measures on $C([0,T];H)$,
together with the viscosity property of the value function. Comparison principle for Master Bellman equation is postponed to further investigation, as it is already a challenging issue in the finite-dimensional case where only partial results exist in the literature, see \cite{WuZhang18PPDE} and \cite{buretal19}.

We also point out that our results clarify and improve in particular some statements from  the finite-dimensional case, like the law invariance property (see Remark \ref{R:Aliprantis})
and the dynamic programming principle.
%

The outline of the paper is organized as follows.  In Section \ref{sec-descr}, we present the notations and formulate the McKean-Vlasov state equation valued in Hilbert space: due to the generality of the setting basic results on
well-posedness and approximation of this equation are not known and are carefully proved.
Section \ref{S:ControlPb} is devoted to the formulation of the optimal control problem, the dynamic programming principle and law invariance property of the associated value function.
We introduce in Section \ref{S:Ito}  the notion of pathwise derivative in Wasserstein space and the related functional It\^o formula.
Section \ref{S:HJB}  is concerned with the derivation of the Master Bellman equation and the viscosity property of the value function. Finally, Appendices \ref{App:StateEquation}, \ref{App:LawInvariance}, \ref{App:PathDeriv}, \ref{ItoProof},  \ref{App:Consistency}, \ref{App:HJB} collect some technical results used throughout the paper.

\section{Controlled path-dependent McKean-Vlasov SDEs
in\\ Hilbert spaces}\label{sec-descr}

\subsection{Notations and assumptions}
\label{Subs:Notations}

\paragraph{State space and functional analytic setting.}
We fix two real separable Hilbert spaces $H$ and $K$, with inner products
$\<\cdot,\cdot\>_H$,$\<\cdot,\cdot\>_K$  and induced norms $|\cdot|_H$,$|\cdot|_K$, respectively, omitting the subscripts $H$ or $K$ when clear from the context. We denote by $\Lc(K;H)$ (resp.\ $\Lc(H)$) the space of bounded linear operators from $K$ to $H$ (resp.\ $H$ to $H$). We endow $\Lc(K;H)$ with the operator norm $\|\cdot\|_{\Lc(K;H)}$ defined by
\[
\|{F}\|_{\Lc(K;H)} = \sup_{k\in K,\;k\neq 0} {|{F}k|_H\over|k|_K}\,, \qquad {F}\in \Lc(K;H).
\]
Similarly, we endow $\Lc(H)$ with the corresponding norm $\|\cdot\|_{\Lc(H)}$. We also denote by $\Lc_2(K;H)$ the space of Hilbert-Schmidt operators from $K$ to $H$, that is the set of all
${F}\in \Lc(K;H)$ such that
\[
\sum_{n\in\N}|{F}e_n|^2_H < +\infty
\]
for some orthonormal basis $\{e_n\}_{n\in\N}$ of $K$. We endow $\Lc_2(K;H)$ with the  norm $\|\cdot\|_{\Lc_2(K;H)}$ defined by
\[
\|{F}\|_{\Lc_2(K;H)}  =  \sqrt{ \sum_{n\in\N}|{F}e_n|^2_H}\,,  \qquad {F}\in\Lc_2(K;H).
\]
We recall that the definitions of $\Lc_2(K;H)$ and $\|\cdot\|_{\Lc_2(K;H)}$ do not depend on the choice of the orthonormal basis $\{e_n\}_{n\in\N}$ of $K$.

We now fix a finite time horizon $T > 0$ and consider the \emph{state space} of our optimal control problem which is given by the set $C([0,T];H)$ of continuous $H$-valued functions on $[0,T]$. Given $x\in C([0,T];H)$ and $t\in[0,T]$, we denote by $x_t$ the value of $x$ at time $t$ and we set $x_{\cdot\wedge t}\coloneqq (x_{s\wedge t})_{s\in[0,T]}$. Notice that $x_t\in H$, while $x_{\cdot\wedge t}\in C([0,T];H)$. We endow $C([0,T];H)$ with the uniform norm $\|\cdot\|_T$ defined as
\[
\|x\|_T = \sup_{s\in[0,T]} |x_s|_H\,, \qquad x\in C([0,T];H).
\]
Notice that $(C([0,T];H),\|\cdot\|_T)$ is a Banach space. We denote by $\mathscr B$ the Borel $\sigma$-algebra of $C([0,T];H)$. Finally, for every $t\in[0,T]$ we introduce the seminorm $\|\cdot\|_t$ defined as
\[
\|x\|_t = \| x_{\cdot\wedge t}\|_T\,, \qquad x\in C([0,T];H).
\]

\paragraph{Spaces of probability measures and Wasserstein distance.} Given a metric space $M$, if $\mathscr{M}$ denotes its Borel $\sigma$-algebra,
 we denote by $\Pc(M)$ the set of all probability measures on $(M,\mathscr M)$. We endow $\Pc(M)$ with the topology of weak convergence. When $M$ is a
Polish space $\rm S$, with metric $d_{\rm S}$, we also define, for $q\geq 1$,
\[
\Pc_q({\rm S}) \coloneqq  \left\{\mu \in \Pc({\rm S})\colon \int_{\rm S} d_{\rm S}(x_0,x)^q \mu(dx)<+\infty\right\} ,
\]
where $x_0\in\rm S$ is arbitrary. This set is endowed with the $q$-Wasserstein distance defined as
\begin{align*}
\Wc_q(\mu,\mu') &\coloneqq  \inf\bigg\{\int_{{\rm S}\times{\rm S}}d_{\rm S}(x,y)^q\, \pi(dx,dy)\colon
  \pi \in \Pc({\rm S}\times{\rm S})\\
 & \hspace{2cm} \text{ such that } \pi(\cdot\times{\rm S})= \mu \mbox{ and }\pi({\rm S}\times\cdot)=\mu'\bigg\}^{1\over q}\,,\qquad q\geq 1,
\end{align*}
for every $\mu,\mu'\in \Pc_q({\rm S})$. The space $\big(\Pc_q({\rm S}),\Wc_q)$ turns out to be a Polish space (see for instance \cite[Theorem 6.18]{Vi09}).

\paragraph{Probabilistic setting.}

We fix a complete probability space $(\Omega,\Fc,\P)$ on which a $K$-valued cylindrical Brownian motion $B=(B_t)_{t\geq0}$ is defined (see e.g.
\cite[Section 4.1]{DPZ} and \cite[Remark 1.89]{FabbriGozziSwiech} on the definition of cylindrical Brownian motion).
We denote by $\F^B=(\Fc_t^B)_{t\geq0}$ the $\P$-completion of the filtration generated by $B$\footnote{Notice that it may be not obvious to define the natural filtration for cylindrical Brownian motion as, in principle, it may depend on the choice of the reference system where such process is considered, which, in general, is not unique. However, as noted in \cite[Remark 1.89]{FabbriGozziSwiech}) this will not affect the class of integrable processes and, consequently, the filtration.}. Notice that $\F^B$ is also right-continuous (see~\cite[Lemma 1.94]{FabbriGozziSwiech}), so, in particular, it satisfies the usual conditions. We assume that there exists a sub-$\sigma$-algebra $\Gc$ of $\Fc$ satisfying the following standing assumptions.
\begin{StandingAssumption}{\bf(A$_\Gc$)}\label{A_G}\quad
{\begin{enumerate}[\upshape i)]
\item $\Gc$ and $\Fc_\infty^B$ are independent;
\item $\Gc$ is ``rich enough'' in the sense that the following property holds:
\begin{align*}
\Pc_2\big(C([0,T];H)\big) &= \big\{\P_\xi\text{ with }\xi\colon[0,T]\times\Omega\rightarrow H\text{ continuous and} \\
&\qquad \hspace{1mm} \text{$\Bc([0,T])\otimes\Gc$-measurable process satisfying }\E\big[\|\xi\|_T^2\big] < \infty\big\},
\end{align*}
i.e. for every $\mu\in\Pc_2(C([0,T];H))$ there exists a continuous and $\Bc([0,T])\otimes\Gc$-measurable process $\xi\colon[0,T]\times\Omega\rightarrow H$, satisfying $\E\|\xi\|_T^2<\infty$, such that $\xi$ has law equal to $\mu$.
\end{enumerate}}
\end{StandingAssumption}

\noindent As stated in the following lemma (take $\Hc=\Gc$ in Lemma \ref{L:U_G}), property \textup{\ref{A_G}}-ii) holds if and only if there exists a $\Gc$-measurable random variable $U_\Gc\colon\Omega\rightarrow\R$ having uniform distribution on $[0,1]$ (see also Remark \ref{R:Atomless}).

\begin{Lemma}\label{L:U_G}
On the probability space $(\Omega,\Fc,\P)$ consider a sub-$\sigma$-algebra $\Hc\subset\Fc$. The following statements are equivalent.
\begin{enumerate}[\upshape 1)]
\item There exists a $\Hc$-measurable random variable $U_\Hc\colon\Omega\rightarrow\R$ having uniform distribution on $[0,1]$.
\item $\Hc$ is ``rich enough'' in the sense that the following property holds:
\begin{align*}
\Pc_2\big(C([0,T];H)\big) &= \big\{\P_\xi\text{ with }\xi\colon[0,T]\times\Omega\rightarrow H\text{ continuous and} \\
&\qquad \hspace{1mm} \text{$\Bc([0,T])\otimes\Hc$-measurable process satisfying }\E\big[\|\xi\|_T^2\big] < \infty\big\}.
\end{align*}
\end{enumerate}
\end{Lemma}
\begin{Remark}\label{R:Atomless}
Using the same notations as in Lemma \ref{L:U_G}, if the probability space $(\Omega,\Hc,\P)$ is \emph{atomless} (namely, for any $E\in\Hc$ such that $\P(E)>0$ there exists $F\in\Hc$, $F\subset E$, such that $0<\P(F)<\P(E)$) then property 1), or equivalently 2), of Lemma \ref{L:U_G} holds (see for instance \cite[vol.\ I, p.\ 352]{cardelbook}).
\end{Remark}
\begin{Remark}\label{R:G}
The additional randomness (other than $B$) coming from the $\sigma$-algebra $\Gc$ is used for the initial condition $\xi$ of the state equation \eqref{cont-MKV-differential} (notice that it is necessary to consider \emph{random} initial conditions in order to state and prove the dynamic programming principle, Theorem \ref{T:DPP}, where for instance the initial condition at time $s$ is given by the random variable $X^{t,\xi,\alpha}$). However, we remark that whenever $t>0$ the $\sigma$-algebra $\Gc$ can be replaced by $\Fc_t^B$; in other words, the initial condition $\xi$ can be taken only $\Fc_t^B$-measurable.  As a matter of fact, for every $t>0$, it holds that:
\begin{enumerate}[i)]
\item $\Fc_t^B$ and $\sigma(B_s-B_t,\ s\geq t)$ are independent (in item \textup{i}) of \textup{\ref{A_G}} we have imposed the stronger condition that $\Gc$ and $\Fc_\infty^B$ have to be independent; however, if we consider only the control problem with initial time $t$, then the assumption imposed here is enough.
\item $\Fc_t^B$ satisfies the property of being ``rich enough'' or, equivalently, it satisfies property 1) of Lemma \ref{L:U_G}.
\end{enumerate}
Therefore, the $\sigma$-algebra $\Gc$ is really necessary only for the control problem with initial time $t=0$. In fact, this allows to define the value function $v$ (see \eqref{v}) for every pair $(t,\mu)$ in $[0,T]\times\Pc_2(C([0,T];H))$, with $\mu$ being the law of $\xi$. On the other hand, if we do not use $\Gc$, $v$ is defined on $(0,T]\times\Pc_2(C([0,T];H))$ and for $t=0$ is defined only at the Dirac measures.
\end{Remark}
\begin{proof}[\textbf{Proof of Lemma \ref{L:U_G}.}]
\noindent 1) $\Longrightarrow$ 2). Fix $\mu\in\Pc_2(C([0,T];H))$. Our aim is to find a process $\xi\colon[0,T]\times\Omega\rightarrow H$ continuous and $\Bc([0,T])\otimes\Hc$-measurable with law equal to $\mu$. To this end, consider the probability space $([0,1],\Bc([0,1]),\lambda)$, where $\lambda$ denotes the Lebesgue measure on the unit interval. Given such a $\mu$, it follows from Theorem 3.19 in \cite{Kallenberg} that there exists a measurable function $\Xi\colon[0,1]\rightarrow C([0,T];H)$ such that the image (or push forward)  measure of $\lambda$ by $\Xi$ is equal to $\mu$. Now, denote
\[
\xi_t(\omega) \coloneqq  \Xi(U_\Hc(\omega))_t, \qquad \forall\ (t,\omega)\in[0,T]\times\Omega,
\]
where the subscript in $\Xi(U_\Hc(\omega))_t$ denotes the valuation at time $t$ of the continuous function $\Xi(U_\Hc(\omega))$.

Notice that $\xi$ is a continuous process with law equal to $\mu$. Moreover, for every fixed $t\in[0,T]$, $\xi_t$ is $\Hc$-measurable.
Since $\xi$ has continuous paths, it follows that $\xi$ is also $\Bc([0,T])\otimes\Hc$-measurable (see for instance \cite{DM}, Chapter IV, Theorem 15). This concludes the proof of the implication 1) $\Longrightarrow$ 2).

\vspace{2mm}

\noindent 2) $\Longrightarrow$ 1). The claim follows from Lemma \ref{L:xi_U_indep}.
\end{proof}

\noindent We denote by $\F=(\Fc_t)_{t\geq0}$ the filtration defined as
\[
\Fc_t = \Gc\vee\Fc_t^B\,, \qquad t\geq0 .
\]
Notice that $\F$ satisfies the usual conditions of completeness and right-continuity. We then denote by $\S_2(\F)$ (resp.\ $\S_2(\Gc)$) the set of $H$-valued continuous $\F$-progressively measurable (resp.\ $\Bc([0,T])\otimes\Gc$-measurable) processes $\xi$ such that
\[
\|\xi\|_{\S_2} \coloneqq  \E\big[\|\xi\|_T^{2}\big]^{\frac{1}{2}} < \infty .
\]

\paragraph{Control processes.} {The \emph{space of control actions}, denoted by $\Ur$, satisfies the following standing assumption.}

\begin{StandingAssumption}{\bf(A$_\Ur$)}\label{A_U}\quad
{$\Ur$ is a Borel space (see for instance Definition 7.7 in \cite{BertsekasShreve}), namely a Borel subset of some Polish space $E$.
$\mathscr U$ denotes its Borel $\sigma$-algebra.}
\end{StandingAssumption}

\begin{Remark}\label{R:Polish}
Our Assumption \textup{\ref{A_U}} is quite general. Indeed in most applications it is enough that $\Ur$ is a Polish space or even a Hilbert space, as in the examples of  Example \ref{ex:SE}.
\end{Remark}

Finally, we denote by $\Uc$ the \emph{space of control processes}, namely the family of all $\F$-progressively measurable processes $\alpha\colon[0,T]\times\Omega\rightarrow\Ur$.

\paragraph{Assumptions on the coefficients of the state equation.}
We consider a linear, possibly unbounded, operator $A:\Dc(A)\subset H \to H$ and two functions
\[
b,\;\sigma :~[0,T]\times C([0,T];H)\times \Pc_2\big(C([0,T];H)\big)\times\Ur\times\Pc(\Ur) \longrightarrow
\ H,\;\Lc_2(K;H),
\]
where we recall that $\Pc(\Ur)$ is endowed with the topology of weak convergence. We impose the following assumptions on $A$, $b$, $\sigma$.

\begin{Assumption}{\bf(A$_{A,b,\sigma}$)}\label{A_A,b,sigma}\quad
\begin{enumerate}[(i)]
\item $A$ generates a $C_0$-semigroup of  pseudo-contractions  $\{e^{tA}, \; t \ge 0\}$ in $H$. Hence, there exists and $\eta\in \R$ such that
\begin{equation}\label{eq:Meta}
\|e^{tA}\|_{\Lc(H)} \leq e^{\eta t}.
\end{equation}
\item The functions $b$ and $\sigma$ are measurable.
\item\label{2020-06-20:01} There exists a constant $L$ such that
\begin{equation*}
  \begin{split}
    |b_t(x,\mu,u,\nu)-b_t(x',\mu',u,\nu)|_H &\leq L \big(\|x-x'\|_t + \Wc_2(\mu,\mu')\big), \\
\|\sigma_t(x,\mu,u,\nu)-\sigma_t(x',\mu',u,\nu)\|_{\Lc_2(K;H)} &\leq L \big(\|x-x'\|_t + \Wc_2(\mu,\mu')\big), \\
|b_t(0,\delta_0,u,\nu)|_H + \|\sigma_t(0,\delta_0,u,\nu)\|_{\Lc_2(K;H)} &\leq L,
\end{split}
\end{equation*}
for all $(t,u,\nu)\in[0,T]\times\Ur\times\Pc(\Ur)$, $(x,\mu),(x',\mu')\in C([0,T];H)\times\Pc_2(C([0,T];H))$, with $\delta_0$ being the Dirac measure at $0$, namely the probability measure on $C([0,T];H)$ putting mass equal to $1$ to the constant path $0$.
\end{enumerate}
\end{Assumption}

\begin{Remark}
Notice that, from the Lipschitz property of $b$ and $\sigma$ with respect to the variable $x\in C([0,T];H)$, it follows that $b$ and $\sigma$ satisfies the following \emph{non-anticipativity property:}
\[
b_t(x,\mu,u,\nu) = b_t(x_{\cdot\wedge t},\mu,u,\nu), \qquad\qquad \sigma_t(x,\mu,u,\nu) = \sigma_t(x_{\cdot\wedge t},\mu,u,\nu),
\]
for every $(t,x,\mu,u,\nu)\in[0,T]\times C([0,T];H)\times\Pc_2(C([0,T];H))\times\Ur\times\Pc(\Ur)$.
\end{Remark}

\begin{Remark}
The optimal control problem of McKean-Vlasov SDEs is sometimes called \emph{extended} or \emph{generalized} (see \cite{Acciaio19,CossoPham19}) when the coefficients also depend on the law of the control process $\P_{\alpha_s}$ (as in the present framework, see equation \eqref{cont-MKV-differential} below) or, more generally, on the joint law $\P_{(X_{\cdot\wedge s},\alpha_s)}$. Notice however that, under the assumptions below, the latter case is only apparently more general than our framework. As a matter of fact, consider for simplicity only the drift coefficient $b$ and suppose that it is replaced by a function $\bar b$ from $[0,T]\times C([0,T];H)\times\Ur\times\Pc(C([0,T];H)\times\Ur)$ to $H$. Recall from \textup{\ref{A_A,b,sigma}} that the Lipschitz continuity of $b$ with respect to the law of the path reads as
\begin{equation}\label{Lipschitz_b}
|b_t(x,\mu,u,\nu)-b_t(x,\mu',u,\nu)|_H \leq L \Wc_2(\mu,\mu'),
\end{equation}
for all $(t,x,u,\nu)\in[0,T]\times C([0,T];H)\times\Ur\times\Pc(\Ur)$, $\mu,\mu'\in\Pc_2(C([0,T];H))$. In other words, the Lipschitz continuity is imposed on the law of the state variable, not on the law of the control variable. If we do the same for the coefficient $\bar b$, we get the following $($$\pi(\cdot\times\Ur)$ is the law of the state variable, while $\pi(C([0,T];H)\times\cdot)$ is the law of the control variable$):$
\begin{equation}\label{Lipschitz_b_2}
|\bar b_t(x,u,\pi)-\bar b_t(x,u,\pi')|_H \leq L \Wc_2\big(\pi(\cdot\times\Ur),\pi'(\cdot\times\Ur)\big),
\end{equation}
for every $(t,x,u)\in[0,T]\times C([0,T];H)\times\Ur$, $\pi,\pi'\in\Pc(C([0,T];H)\times\Ur)$ with $\pi(\cdot\times\Ur),\pi'(\cdot\times\Ur)\in\Pc_2(C([0,T];H))$ and $\pi(C([0,T];H)\times\cdot)=\pi'(C([0,T];H)\times\cdot)$. We require that $\pi(C([0,T];H)\times\cdot)=\pi'(C([0,T];H)\times\cdot)$ since, as in \eqref{Lipschitz_b}, the law of the control variable is fixed. More precisely, using the notation of \eqref{Lipschitz_b},
\[
\mu \ = \ \pi(\cdot\times\Ur), \qquad \mu' \ = \ \pi'(\cdot\times\Ur), \qquad \nu \ = \ \pi(C([0,T];H)\times\cdot) \ = \ \pi'(C([0,T];H)\times\cdot).
\]
It follows directly from assumption \eqref{Lipschitz_b_2} that $\bar b_t(x,u,\pi)=\bar b_t(x,u,\pi')$ whenever the marginals of $\pi$ and $\pi'$ coincide: $\pi(\cdot\times\Ur)=\pi'(\cdot\times\Ur)$ and $\pi(C([0,T];H)\times\cdot)=\pi'(C([0,T];H)\times\cdot)$. This shows that $\bar b$ depends on $\pi$ only through its marginals.
\end{Remark}

\subsection{State equation}

Given an initial time $t\in[0,T]$, an initial path $\xi\in\S_2(\F)$, a control process $\alpha\in\Uc$, the state process evolves according to the following controlled path-dependent McKean-Vlasov stochastic differential equation:
\begin{equation}\label{cont-MKV-differential}
\begin{cases}
dX_s = AX_s+ b_s\big(X,\P_X,\alpha_s,\P_{\alpha_s}\big)ds
+\sigma_s\big(X,\P_X,\alpha_s,\P_{\alpha_s}\big)dB_s &s>t \\
X_s = \xi_s & s\leq t.
\end{cases}
\end{equation}

\begin{Definition}
Fix $t\in[0,T]$, $\xi\in\S_2(\F)$, $\alpha\in\Uc$. A \textbf{mild solution} of \eqref{cont-MKV-differential} is a process $X=(X_s)_{s\in[0,T]}$ in $\S_2(\F)$ satisfying
  \begin{multline*}
X_s = e^{((s-t)\vee 0)A}\,\xi_{s\wedge t}+\int_t^{s\vee t}
e^{(s-r)A}\,b_r\big(X,\P_{X_{\cdot\wedge r}},\alpha_r,\P_{\alpha_r}\big)\,dr\nonumber \\
 + \int_t^{s\vee t}
e^{(s-r)A}\,\sigma_r\big(X,\P_{X_{\cdot\wedge r}},\alpha_r,\P_{\alpha_r}\big)\,dB_r\qquad \forall s\in[0,T],\,\P\text{-a.s.}
\end{multline*}
\end{Definition}

\begin{Proposition}\label{Prop-diff-MKVPD}
Fix $t\in[0,T]$, $\xi\in\S_2(\F)$, $\alpha\in\Uc$. Under \textup{\ref{A_A,b,sigma}}, equation \eqref{cont-MKV-differential} admits a unique mild solution
$X^{t,\xi,\alpha}\in\S_2(\F)$.
The map
\[
    [0,T]\times\S_2(\F) \rightarrow \S_2(\F),
(t,\xi) \mapsto  X^{t,\xi,\alpha}
\]
is jointly continuous in $(t,\xi)$, uniformly with respect to $\alpha\in\Uc$, and Lipschitz continuous in $\xi$, uniformly in $t$ and $\alpha$.
 Moreover, $X^{t,\xi,\alpha}=X^{t,\xi_{\cdot\wedge t},\alpha}$ and there exists a constant $C$, independent of $t,\xi,\alpha$, such that
\begin{equation}\label{EstimateX}
\big\|X^{t,\xi,\alpha}\big\|_{\S_2} \leq C\,\big(1 + \|\xi_{\cdot\wedge t}\|_{\S_2}\big).
\end{equation}
\end{Proposition}
\begin{proof}[\textbf{Proof.}]
See Appendix \ref{App:StateEquation}.
\end{proof}

\begin{Remark}\label{rm:convSEinitialdatum}
From Proposition~\ref{Prop-diff-MKVPD} we have, in particular, for $r,t\in[0,T], \xi\in \mathbf{S}_2(\mathbb{F})$,
\begin{equation}
  \label{2020-10-11:03}
  \begin{split}
      \lim_{r\rightarrow t} \sup_{\alpha\in \mathcal{U}}  \|X^{\alpha,t,\xi}_{r\wedge \cdot}-\xi_{t\wedge \cdot}\|_{\mathbf{S}_2}
  \ &\leq \
  \lim_{r\rightarrow t}
   \left(
\sup_{\alpha\in \mathcal{U}}  \|X^{\alpha,t,\xi}_{r\wedge \cdot}-
  X^{\alpha,r,\xi}_{r\wedge \cdot}\|_{\mathbf{S}_2}
+\|\xi_{r\wedge \cdot}
-\xi_{t\wedge \cdot}\|_{\mathbf{S}_2} \right)\\
&\leq \
\lim_{r\rightarrow t}
\left(
  \sup_{\alpha\in \mathcal{U}}
  \|X^{\alpha,t,\xi}-
  X^{\alpha,r,\xi}\|_{\mathbf{S}_2}
+\|\xi_{r\wedge \cdot}
  -\xi_{t\wedge \cdot}\|_{\mathbf{S}_2} \right) \ = \ 0,
\end{split}
\end{equation}
where we have used the fact that
$X^{\alpha,r,\xi}_{r\wedge \cdot}=\xi_{r\wedge \cdot}$.
\end{Remark}

Let $\{A_n\}_{n\in\mathbb{N}}$ be the Yosida approximation of $A$, i.e. $A_n=nA(n-A)^{-1}$, for $n\in \mathbb{N},n> \eta$, with $\eta$ as in \eqref{eq:Meta}. Denote by $S^n$ the uniformly continuous semigroup generated by $A_n$.
Notice that
$S^n$ is a pseudo-contraction semigroup for all $n\in \mathbb{N}$, $n> \eta$, and that, for some $\tilde \eta>0$, $\|S^n_t\|_{\mathcal{L}(H)}\leq e^{\tilde{\eta}t}$
uniformly for $n\in \mathbb{N}, t\geq 0$.
In particular,
we can apply Proposition~\ref{Prop-diff-MKVPD} to obtain existence of a unique mild solution
 $X^{n,t,\xi,\alpha}$
to the following equation:
\begin{equation}\label{cont-MKV-differential-An}
\begin{cases}
dX^n_s = A_nX^n_s+ b_s\big(X^n,\P_{X^n},\alpha_s,\P_{\alpha_s}\big)ds
+\sigma_s\big(X^n,\P_{X^n},\alpha_s,\P_{\alpha_s}\big)dB_s, &\qquad s>t, \\
X^n_s = \xi_s, &\qquad s\leq  t.
\end{cases}
\end{equation}

\begin{Proposition}\label{Prop-diff-MKVPD-An}
  There exists a constant $C>0$ such that
  \begin{equation}
    \label{2020-10-11:07}
    \sup_{\substack{\alpha\in \mathcal{U}\\n\in \mathbb{N}\\t\in[0,T]}}
    \|X^{n,t,\xi,\alpha}-X^{n,t,\xi',\alpha}\|_{\mathbf{S}_2} \ \leq \ C\|\xi-\xi'\|_{\mathbf{S}_2},\qquad \forall\,\xi,\xi' \in \mathbf{S}_2(\mathbb{F}).
  \end{equation}
  Moreover,
  \begin{equation}
    \label{2020-10-11:08}
    \lim_{\substack{t\rightarrow t'\\n\rightarrow \infty}}
    \|X^{n,t,\xi,\alpha}-X^{t',\xi,\alpha}\|_{\mathbf{S}_2} \ = \ 0,\qquad \forall\,\alpha\in \mathcal{U},\,\xi\in \mathbf{S}_2(\mathbb{F}),\,t'\in [0,T].
  \end{equation}
\end{Proposition}

\begin{proof}[\textbf{Proof.}]
See Appendix \ref{App:StateEquation}.
\end{proof}

\begin{Remark}\label{rm:unboundedcontrol}
In the third inequality of Assumption \ref{A_A,b,sigma}-(iv) we assume, for simplicity, the boundedness of $b$ and $\sigma$ with respect to the controls.
In many applications (e.g. in linear quadratic control cases see Example \ref{ex:SE} below) the state equation contains unbounded control terms. It is therefore interesting to understand what happens in such cases.
Assume that $\Ur$ is a closed subset of a Hilbert space and that the right-hand side of the third inequality of Assumption \ref{A_A,b,sigma}-(iv) is replaced by
$L(1+|u|_{\Ur})$. In this case Proposition \ref{Prop-diff-MKVPD} still holds, assuming that $\alpha_{\cdot} \in L^2([0,T]\times \Omega;\Ur)$, with the estimate \eqref{EstimateX} replaced by
\begin{equation}\label{EstimateXnew}
\big\|X^{t,\xi,\alpha}\big\|_{\S_2} \leq C\,\left[1 + \|\xi_{\cdot\wedge t}\|_{\S_2} +\ \left(\E\int_t^T |\alpha_{s}|^2_{\Ur} ds\right)\right].
\end{equation}
The only difference is that the joint continuity in $(t,\xi)$ is uniform only with respect to $\alpha$ belonging to the bounded sets of $L^2([0,T]\times \Omega;\Ur)$. Consequently also the limit in \eqref{2020-10-11:03}
is uniform only in the same sense.
Similarly, also  Proposition \ref{Prop-diff-MKVPD-An} still holds but with the supremum in $\alpha$ belonging to the bounded sets of $L^2([0,T]\times \Omega;\Ur)$.
\end{Remark}

\begin{Example}\label{ex:SE}
Examples of problems where the state equation has the above structure.
\begin{itemize}
  \item[(i)] Lifecycle optimal portfolio problems.
This family of problems introduces, together with the standard equation for the wealth $x(\cdot)$ used in Merton model, another state variable $y$ which is the labor income of the agent.
The equation for the labor income (which is one of the state equations\footnote{The equation for $y$ is not controlled, however $y$ is part of the state as it appears in the wealth dynamics, which is controlled.}
of the optimal portfolio problem)
is, in the simplest case, the one of a geometric Brownian motion (see e.g.  \cite{DYBVIG_LIU_JET_2010}).
It is however natural, for a more realistic description of such dynamics, to introduce two extensions in the equation for $y$:
\begin{itemize}
  \item
first, as proposed in the concluding remarks of \cite{DYBVIG_LIU_JET_2010} and done (in different cases) in \cite{BiffisGozziProsdocimi20,BCGZ21,BGZ21},
to add a path-dependent term in the drift and/or in the diffusion;
  \item
second, as proposed in the introduction of \cite{DjeicheGozziZancoZanella20}),
to add, in the drift, a mean-field term depending on the distribution $\P_{y(t)}$ of $y$ itself at time $t$.\footnote{As written at pages 2-3 of \cite{DjeicheGozziZancoZanella20},
``the labor income $y_i$ of an agent $i$ is benchmarked
against the labor incomes of a population
$y^N := (y_1, y_2, \dots , y_N)$ of $N$ agents with comparable tasks or ranks among the profession such as the level of full professor, associate professor, actuary, trader,
risk manager etc., where one usually uses some wage level
$b(y^N)$ as a reference to declare whether
that agent has a superior, fair or inferior labor income compared with her peers''.}
\end{itemize}
Consequently, it makes sense to model the dynamics of the labor income ``$y$'' using a one-dimensional stochastic delay ODE of McKean-Vlasov type as follows (here $\phi\in L^2(-d,0)$ is a given datum providing the weight of the past income into the the current trend, and $Z$ is a one-dimensional Brownian motion).
      $$
      dy(t) \ = \ \left[b_0\big(\P_{y(t)}\big) +
      \int_{-d}^{0}y(t+\xi)\,\phi(\xi)\,d\xi\right]dt + \sigma\,y(t)\,dZ(t).
      $$
      Such equations can be rephrased as SDEs in the Hilbert space $\R\times L^2(-d,0)$ and the resulting dynamics falls into the class treated in the present section.
      In \cite{DjeicheGozziZancoZanella20} the authors study only the case where $b_0$ is a linear function\footnote{This is needed there to get a simplified HJB equation and to find explicit solutions of it.} of the expectation $\E[y(t)]$ but, as explained there,
      other choices are possible, such as ``the median
wage or the truncated average above a certain level, within the company or even within the profession''.

The setting of the present paper allows to cover not only the case when $b_0$ above is nonlinear but also more general (and still interesting, like e.g. the average income of the last years) cases where it depends on the law of the past of $y$.

Clearly, the problem becomes much more difficult than the one treated in \cite{DjeicheGozziZancoZanella20} and we should consider the results of this paper as a first step to study such types of models.

\item[(ii)] Optimal investment with vintage capital.
These are typical partial equilibrium models arising in Economics
(see e.g. \cite{BarucciGozzi98,BarucciGozzi01,FaggianGozzi10,Feichtingeretal06JET})
where the state variable ``$x$'' is the capital stock and the control variable is the investment ``$u$'', both depending on time $t\geq0$ and vintage $s\in [0,\bar{s}]$ (here $\bar s>0$ is the maximum possible vintage): capital goods indexed with small $s$ embody newer technologies. In the above papers $x$, in the simplest cases, is required to satisfy a first-order PDE of the following type:
$$
\frac{\partial x(t,s)}{\partial t}+
\frac{\partial x(t,s)}{\partial s}=-\delta x(t,s)+ u(t,s),
$$
where $\delta$ is a depreciation rate of the capital goods, which is kept constant for simplicity.
This PDE can be easily rewritten as an ODE in the space $H:=L^2(0,\bar s)$.
If one wants to take into account stochastic disturbances (similarly to what is done in \cite{Fousekis} in a case without vintage), the state equation,
written in $L^2(0,\bar s)$, becomes the following infinite dimensional SDE
(here $B$ is a cylindrical Wiener process):
$$
dx(t) \ = \ \left[Ax(t) - \delta x(t) +Cu(t)\right]dt + \sigma(x(t))\,dB(t),
$$
for suitable linear operators $A,C,\delta$ and
$\sigma\colon H\to\mathcal{L} (H)$. Here $x(t),u(t)$ stand for $x(t,\cdot)$ and $u(t,\cdot)$, respectively.
\\
Given such a controlled infinite dimensional SDE, the objective to be maximized depends, beyond the control $u$, also on the production $Q(t)$ provided by the capital stock.
The expression of $Q(t)$, as given for example in \cite{Feichtingeretal06JET}, is
$$
Q(t):=\int_{0}^{\bar s}f(t-s)v(s)x(t,s)ds,
$$
where $f$ takes into account  the technological progress and $v$ embodies learning and spillover effects. As observed e.g in \cite{Garcia06,LucasMoll12}, effects of this type can be modelled taking $f$ and $v$ depending on the distribution of $x(t)$ and $u(t)$.
Moreover, effects like the so-called time to build (see e.g. \cite{KydlandPrescott,BambiJEDC}) call for path-dependency for such coefficients
(see, for a deterministic infinite dimensional modelling of time-to-build, \cite{BambiJEDC,BambiFabbriGozzi,BambiDiGirolami,FGG1,FGG2}).

\item[(iii)] Optimal consumption in spatial growth models in Economics.
A way to model capital accumulation in spatial growth models
(see e.g. \cite{BoucekkineCamachoFabbri13}) is to assume that the capital $x(t,\xi)$ at time $t\ge 0$ in the position $\xi$ (here we take, for simplicity, $\xi\in S^1$, the one dimensional sphere)
satisfies a second order PDE like
$$
\frac{\partial x(t,\xi)}{\partial t}=
\frac{\partial^2 x(t,\xi)}{\partial \xi^2}+a(t,\xi) x(t,\xi) -c(t,\xi), \qquad \xi \in S^1,\,t\ge 0,
$$
where $a$ is a productivity coefficient and $c$ the consumption.
If we modify such PDE to take into account the time delay $d$ due to time-to-build, it takes the following form:
$$
\frac{\partial x(t,\xi)}{\partial t}=
\frac{\partial^2 x(t,\xi)}{\partial \xi^2}+a(t,\xi) x(t-d,\xi) -c(t,\xi), \qquad \xi \in S^1,\,t\ge 0.
$$
By considering  stochastic disturbances (as done, e.g., in \cite{Bismut75} for the case without space variable and in \cite{GozziLeocata21} in the spatial growth framework)
and of the mean-field dependence on the productivity (as argued in the previous example) we get the following path-dependent SPDE of McKean-Vlasov type in the space $H:=L^2(S^1)$
$$
dx(t)= \left[Ax(t) + h\left(x(t-d),\P_{x(t)}\right)-c(t)\right] dt
+\sigma(x(t))dB(t),
$$
where $B$ is a cylindrical Wiener process,
$A$ is a suitable linear second order differential operator, $h\colon H\times \mathcal{P}(H)\to H$  and
$\sigma\colon H\to\mathcal{L} (H)$. Here $x(t),c(t)$ stand for $x(t,\cdot)$ and $c(t,\cdot)$, respectively.
Again this problem falls into the class we treat in the present paper.
\end{itemize}
\end{Example}

\section{The optimal control problem}
\label{S:ControlPb}

\subsection{Reward functional and lifted value function}

We are given two functions
\begin{align*}
f\colon[0,T]\times C([0,T];H)\times \Pc_2\big(C([0,T];H)\big)\times\Ur\times\Pc(\Ur) &\longrightarrow \R \\
g\colon C([0,T];H)\times \Pc_2\big(C([0,T];H)\big) &\longrightarrow \R
\end{align*}
on which we impose the following assumptions.

\begin{Assumption}{\bf(A$_{f,g}$)}\label{A_f,g}\quad
\begin{enumerate}[(i)]
\item The functions $f$ and $g$ are measurable.
\item The function $f$ satisfies the non-anticipativity property:
\[
f_t(x,\mu,u,\nu) = f_t(x_{\cdot\wedge t},\mu,u,\nu),
\]
for all $(t,x,\mu,u,\nu)\in[0,T]\times C([0,T];H)\times\Pc_2(C([0,T];H))\times\Ur\times\Pc(\Ur)$.
\item There exists a locally bounded function $h\colon[0,\infty)\rightarrow[0,\infty)$ such that
\[
|f_t(x,\mu,u,\nu)| \leq h\big(\Wc_2(\mu,\delta_0)\big)\big(1 + \|x\|^2_t\big), \qquad |g(x,\mu)| \leq h\big(\Wc_2(\mu,\delta_0)\big)\big(1 + \|x\|^2_T\big),
\]
for all $(t,x,\mu,u,\nu)\in[0,T]\times C([0,T];H)\times\Pc_2(C([0,T];H))\times\Ur\times\Pc(\Ur)$.
\end{enumerate}
\end{Assumption}

\ni We will also need the following continuity assumption on $f$ and $g$.

\begin{Assumption}{{\bf(A}$_{f,g}${\bf)}$_\text{\textnormal{cont}}$}\label{A_f,g_cont}
The function $f$ is locally uniformly continuous in $(x,\mu)$ uniformly with respect to $(t,u,\nu)$. Similarly, $g$ is locally uniformly continuous. More precisely, it holds that: for every $\eps>0$ and $n\in\N$ there exists $\delta=\delta(\eps,n)>0$ such that, for every $(t,u,\nu)\in[0,T]\times\Ur\times\Pc(\Ur)$, $(x,\mu),(x',\mu')\in C([0,T];H)\times\Pc_2(C([0,T];H))$, with $\|x\|_T+\Wc_2(\mu,\delta_0)\leq n$ and $\|x'\|_T+\Wc_2(\mu',\delta_0)\leq n$,
\begin{align*}
&\|x - x'\|_t + \Wc_2(\mu,\mu') \leq \delta \\
&\hspace{2cm}\Longrightarrow \quad |f(t,x,\mu,u,\nu) - f(t,x',\mu',u,\nu)| \leq \eps \; \text{ and } \; |g(x,\mu) - g(x',\mu')| \leq \eps.
\end{align*}
\end{Assumption}

Under Assumptions \textup{\ref{A_A,b,sigma}} and \textup{\ref{A_f,g}}, from Proposition \ref{Prop-diff-MKVPD} we get that the \emph{reward functional} $J$, given by
\[
J(t,\xi,\alpha) = \E\bigg[\int_t^Tf_s\big(X^{t,\xi,\alpha},\P_{X^{t,\xi,\alpha}_{\cdot\wedge s}},\alpha_s,\P_{\alpha_s}\big)\,ds + g\big(X^{t,\xi,\alpha},\P_{X^{t,\xi,\alpha}}\big)\bigg],
\]
is well-defined for any $(t,\xi,\alpha)\in [0,T]\times\S_2(\F)\times\Uc$. We then consider the function $V\colon[0,T]\times\S_2(\F)\longrightarrow\R$, to which we refer as the \emph{lifted value function}, defined as
\begin{equation}\label{V_aux}
V(t,\xi) = \sup_{\alpha\in \Uc} J(t,\xi,\alpha)\,, \qquad \forall\,(t,\xi)\in[0,T]\times\S_2(\F)\,.
\end{equation}

\begin{Remark}\label{R:V_non_ant}
Recall from Proposition \ref{Prop-diff-MKVPD} that $X^{t,\xi,\alpha}$ only involves the values of $\xi$ up to time $t$, namely it holds that $X^{t,\xi,\alpha}=X^{t,\xi_{\cdot\wedge t},\alpha}$. As a consequence, both $J$ and $V$ satisfy the \emph{non-anticipativity property:}
\[
J(t,\xi,\alpha) = J(t,\xi_{\cdot\wedge t},\alpha), \qquad\qquad V(t,\xi) = V(t,\xi_{\cdot\wedge t}),
\]
for every $(t,\xi)\in[0,T]\times\S_2(\F)$, $\alpha\in\Uc$.
\end{Remark}

\begin{Remark}\label{rm:quadraticreward}
In Remark \ref{rm:unboundedcontrol} we looked at the case when $\Ur$ is a Hilbert space and the coefficients of the state equation have linear growth in the controls. In such a case, in order to have a well-defined reward functional $J$, we need some compensating term in the current reward $f$. A typical assumption which guarantees that $J$ is well-defined is that the first inequality of Assumption \ref{A_f,g}-(iii) is replaced by
\[
f_t(x,\mu,u,\nu) \ \leq \ h\big(\mathcal{W}_2(\mu,\delta_0)\big)\,\big(1+\|x\|_t^2\big) - C |u|_{\Ur}^\theta,
\]
for some $\theta>1$ and $C$ $>$ $0$.  This would include some typical linear-quadratic control cases. For example in the case of optimal investment problems mentioned in Example \ref{ex:SE}-(ii)
a typical form of $f$ would be (recall that here $H=L^2(0,\bar s)$)
\begin{equation}\label{eq:costinvest}
f_t(x,\mu,u,\nu) \ = \ f_t(x,u) \ = \ e^{-rt}[R(Q(t))-\<a_2,u(t)\>_H-
\<Mu(t),u(t)\>_H],
\end{equation}
where $r$ is the interest rate, $R$ is suitable one variable function (possibly linear or quadratic) $a_2\in L^2(0,\bar s)$,
and $M$ is a suitable multiplication operator
in $L^2(0,\bar s)$ (see e.g. \cite{FaggianGozzi10,Feichtingeretal06JET}).
\end{Remark}

\begin{Proposition}\label{P:Cont}
Suppose that \textup{\ref{A_A,b,sigma}} and \textup{\ref{A_f,g}} hold. The function $V$ satisfies a quadratic growth condition: there exists a constant $C$ such that
\begin{equation}\label{EstimateV}
|V(t,\xi)| \leq C\,\big(1 + \|\xi_{\cdot\wedge t}\|_{\S_2}^2\big),
\end{equation}
for every $(t,\xi)\in[0,T]\times\S_2(\F)$. Moreover, if in addition \textup{\ref{A_f,g_cont}} holds, the map $V\colon[0,T]\times\S_2(\F)\rightarrow\R$ is jointly continuous.
\end{Proposition}
\begin{proof}[\textbf{Proof.}]
We split the proof into two steps.

\vspace{1mm}

\noindent\emph{\textbf{Step 1}. Proof of estimate \eqref{EstimateV}.} From the definition of $J$, we have
\[
|J(t,\xi,\alpha)| \leq \E\bigg[\int_t^T\big|f_s\big(X^{t,\xi,\alpha},\P_{X^{t,\xi,\alpha}_{\cdot\wedge s}},\alpha_s,\P_{\alpha_s}\big)\big|\,ds\bigg] + \E\big[\big|g\big(X^{t,\xi,\alpha},\P_{X^{t,\xi,\alpha}}\big)\big|\big].
\]
By the quadratic growth of $f$ and $g$, together with estimate \eqref{EstimateX}, we see that there exists a constant $C$ such that
\begin{equation}\label{EstimateJ}
|J(t,\xi,\alpha)| \leq C\,\big(1 + \|\xi_{\cdot\wedge t}\|_{\S_2}^2\big),
\end{equation}
for every $(t,\xi,\alpha)\in[0,T]\times\S_2(\F)\times\Uc$. Then, estimate \eqref{EstimateV} follows directly from the definition of $V$ and the fact that \eqref{EstimateJ} holds uniformly with respect to $\alpha\in\Uc$.

\vspace{1mm}

\noindent\emph{\textbf{Step 2}. Continuity of $V$.} We begin noticing that, for every $(t,\xi),(s,\eta)\in[0,T]\times\S_2(\F)$,
\[
|V(t,\xi) - V(s,\eta)| \leq \sup_{\alpha\in\Uc} |J(t,\xi,\alpha) - J(s,\eta,\alpha)|.
\]
Then, the continuity of $V$ follows once we prove that $J$ is continuous in $(t,\xi)$ uniformly with respect to $\alpha$, namely that the following property holds: for every $\eps>0$ and every $(t,\xi)\in[0,T]\times\S_2(\F)$, there exists $\delta=\delta(\eps,t,\xi)>0$ such that, for every $(s,\eta,\alpha)\in[0,T]\times\S_2(\F)\times\Uc$,
\[
|t - s|\,\leq\,\delta \quad \text{ and } \quad \|\xi - \eta\|_{\S_2} \leq \delta \qquad \Longrightarrow \qquad |J(t,\xi,\alpha) - J(s,\eta,\alpha)| \leq \eps.
\]
Such a property is a straightforward consequence of the last statement of Proposition \ref{Prop-diff-MKVPD} and of assumption
\textup{\ref{A_f,g_cont}}.
\end{proof}

\subsection{Dynamic programming principle for $V$}

In this section we prove the dynamic programming principle for the lifted value function $V$ defined in \eqref{V_aux}.

\begin{Theorem}\label{T:DPP}
Suppose that \textup{\ref{A_A,b,sigma}} and \textup{\ref{A_f,g}} hold. The lifted value function $V$ satisfies the \textbf{dynamic programming principle}: for every $t,s\in[0,T]$, with $t\leq s$, and every $\xi\in\S_2(\F)$ it holds that
\[
V(t,\xi) = \sup_{\alpha\in\Uc}\bigg\{\E\bigg[\int_t^s f_r\big(X^{t,\xi,\alpha},\P_{X^{t,\xi,\alpha}_{\cdot\wedge r}},\alpha_r,\P_{\alpha_r}\big)\,dr\bigg] + V\big(s, X^{t,\xi,\alpha}\big)\bigg\}.
\]
\end{Theorem}

\begin{proof}[\textbf{Proof.}]
Set
\[
\Lambda(t,\xi) \coloneqq  \sup_{\alpha\in\Uc}\bigg\{\E\bigg[\int_t^s f_r\big(X^{t,\xi,\alpha},\P_{X^{t,\xi,\alpha}_{\cdot\wedge r}},\alpha_r,\P_{\alpha_r}\big)\,dr\bigg] + V\big(s, X^{t,\xi,\alpha}\big)\bigg\}.
\]
\emph{\textbf{Step 1}. Proof of the inequality $\Lambda(t,\xi)\geq V(t,\xi)$.} For every fixed $\alpha\in\Uc$, the lifted value function at $(s,X^{t,\xi,\alpha})$ is given by
\[
V(s,X^{t,\xi,\alpha}) = \sup_{\beta\in\Uc}\E\bigg[\int_s^T f_r\Big(X^{s,X^{t,\xi,\alpha},\beta},\P_{X^{s,X^{t,\xi,\alpha},\beta}_{\cdot\wedge r}},\beta_r,\P_{\beta_r}\Big)\,dr + g\Big(X^{s,X^{t,\xi,\alpha},\beta},\P_{X^{s,X^{t,\xi,\alpha},\beta}}\Big)\bigg].
\]
Choosing $\beta=\alpha$, we find
\[
V(s,X^{t,\xi,\alpha}) \geq \E\bigg[\int_s^T f_r\Big(X^{s,X^{t,\xi,\alpha},\alpha},\P_{X^{s,X^{t,\xi,\alpha},\alpha}_{\cdot\wedge r}},\alpha_r,\P_{\alpha_r}\Big)\,dr + g\Big(X^{s,X^{t,\xi,\alpha},\alpha},\P_{X^{s,X^{t,\xi,\alpha},\alpha}}\Big)\bigg].
\]
By the uniqueness property for equation \eqref{cont-MKV-differential} stated in Proposition \ref{Prop-diff-MKVPD}, we obtain the flow property
\[
X^{t,\xi,\alpha} = X^{s,X^{t,\xi,\alpha},\alpha}.
\]
Hence
\[
V(s,X^{t,\xi,\alpha}) \geq \E\bigg[\int_s^T f_r\big(X^{t,\xi,\alpha},\P_{X^{t,\xi,\alpha}_{\cdot\wedge r}},\alpha_r,\P_{\alpha_r}\big)\,dr + g\big(X^{t,\xi,\alpha},\P_{X^{t,\xi,\alpha}}\big)\bigg].
\]
Adding to both sides the quantity $\E\int_t^s f_r(X^{t,\xi,\alpha},\P_{X^{t,\xi,\alpha}_{\cdot\wedge r}},\alpha_r,\P_{\alpha_r})\,dr$, we get
\[
\Lambda(t,\xi) \geq \E\bigg[\int_t^T f_r\big(X^{t,\xi,\alpha},\P_{X^{t,\xi,\alpha}_{\cdot\wedge r}},\alpha_r,\P_{\alpha_r}\big)\,dr + g\big(X^{t,\xi,\alpha},\P_{X^{t,\xi,\alpha}}\big)\bigg].
\]
As the latter inequality holds true for every $\alpha\in\Uc$, we conclude that $\Lambda(t,\xi)\geq V(t,\xi)$.

\vspace{1mm}

\noindent\emph{\textbf{Step 2}. Proof of the inequality $\Lambda(t,\xi)\leq V(t,\xi)$.} For every $\eps>0$, let $\alpha^\eps\in\Uc$ be such that
\begin{equation}\label{DPP_proof1}
\Lambda(t,\xi) \leq  \E\bigg[\int_t^s f_r\big(X^{t,\xi,\alpha^\eps},\P_{X^{t,\xi,\alpha^\eps}_{\cdot\wedge r}},\alpha_r^\eps,\P_{\alpha_r^\eps}\big)\,dr\bigg] + V(s,X^{t,\xi,\alpha^\eps}) + \eps.
\end{equation}
From the definition of $V(s,X^{t,\xi,\alpha^\eps})$, it follows that there exists $\beta^\eps\in\Uc$ such that
\begin{align}\label{DPP_proof2}
V(s,X^{t,\xi,\alpha^\eps}) &\leq \E\bigg[\int_s^T f_r\Big(X^{s,X^{t,\xi,\alpha^\eps},\beta^\eps},\P_{X^{s,X^{t,\xi,\alpha^\eps},\beta^\eps}_{\cdot\wedge r}},\beta_r^\eps,\P_{\beta_r^\eps}\Big)\,dr \\
&\quad + \, g\Big(X^{s,X^{t,\xi,\alpha^\eps},\beta^\eps},\P_{X^{s,X^{t,\xi,\alpha^\eps},\beta^\eps}}\Big)\bigg] + \eps. \notag
\end{align}
Set
\[
\gamma^\eps = \alpha^\eps\,\mathds{1}_{[0,s]} + \beta^\eps\,\mathds{1}_{(s,T]}.
\]
Notice that $\gamma^\eps\in\Uc$. Using again the uniqueness property for equation \eqref{cont-MKV-differential}, we get
\[
X^{s,X^{t,\xi,\alpha^\eps},\beta^\eps} = X^{t,\xi,\gamma^\eps}.
\]
Hence, \eqref{DPP_proof2} becomes
\[
V(s,X^{t,\xi,\alpha^\eps}) \leq \E\bigg[\int_s^T f_r\big(X^{t,\xi,\gamma^\eps},\P_{X^{t,\xi,\gamma^\eps}_{\cdot\wedge r}},\gamma_r^\eps,\P_{\gamma_r^\eps}\big)\,dr + g\big(X^{t,\xi,\gamma^\eps},\P_{X^{t,\xi,\gamma^\eps}}\big)\bigg] + \eps.
\]
Then, by \eqref{DPP_proof1} it follows that
\begin{align}\label{DPP_proof3}
\Lambda(t,\xi) \leq \E\bigg[\int_t^s f_r\big(X^{t,\xi,\alpha^\eps},\P_{X^{t,\xi,\alpha^\eps}_{\cdot\wedge r}},\alpha_r^\eps,\P_{\alpha_r^\eps}\big)\,dr &+ \int_s^T f_r\big(X^{t,\xi,\gamma^\eps},\P_{X^{t,\xi,\gamma^\eps}_{\cdot\wedge r}},\gamma_r^\eps,\P_{\gamma_r^\eps}\big)\,dr \notag \\
&+ g\big(X^{t,\xi,\gamma^\eps},\P_{X^{t,\xi,\gamma^\eps}}\big)\bigg] + 2\eps.
\end{align}
From the definition of $\gamma^\eps$, we see that
\[
X^{t,\xi,\alpha^\eps}_{\cdot\wedge s} = X^{t,\xi,\gamma^\eps}_{\cdot\wedge s}.
\]
As a consequence, we can rewrite \eqref{DPP_proof3} in terms of the only process $X^{t,\xi,\gamma^\eps}$ as
\[
\Lambda(t,\xi) \leq \E\bigg[\int_t^T \!\! f_r\big(X^{t,\xi,\gamma^\eps},\P_{X^{t,\xi,\gamma^\eps}_{\cdot\wedge r}},\gamma_r^\eps,\P_{\gamma_r^\eps}\big)dr + g\big(X^{t,\xi,\gamma^\eps},\P_{X^{t,\xi,\gamma^\eps}}\big)\bigg] + 2\eps \leq V(t,\xi) + 2\eps.
\]
The claim follows from the arbitrariness of $\eps$.
\end{proof}

  \begin{Remark}
It is worth noticing that, despite the stochastic setting, there is no issue of measurability in the proof of the dynamic programming principle. This is a consequence of the fact that the function $V$ depends on the whole random variable $\xi$, so that the proof of the dynamic programming principle can be done proceeding along the same lines as in the case of deterministic optimal control.
\end{Remark}

\subsection{Law invariance property of the lifted value function $V$}

In the present section we introduce the \emph{value function} of the optimal control problem, which is a real-valued map defined on $[0,T]\times\Pc_2(C([0,T];H))$ (see \eqref{v}). In order to define such a value function, it is necessary to prove that the map $V$ satisfies the following \emph{law invariance property:} for every $t\in[0,T]$ and every $\xi,\eta\in\S_2(\F)$ it holds that
\[
V(t,\xi) = V(t,\eta).
\]
This is the subject of the next theorem.

\begin{Theorem}\label{T:id-law}
Suppose that \textup{\ref{A_A,b,sigma}} and \textup{\ref{A_f,g}} hold. Fix $t\in[0,T]$ and $\xi,\eta\in\S_2(\F)$, with $\P_\xi=\P_\eta$. Suppose that there exist two random variables $U_\xi$ and $U_\eta$ having uniform distribution on $[0,1]$, being $\Fc_t$-measurable and such that $\xi$ and $U_\xi$ (resp.\ $\eta$ and $U_\eta$) are independent. Then, it holds that
\[
V(t,\xi) = V(t,\eta).
\]
If in addition \textup{\ref{A_f,g_cont}} holds, the map $V$ satisfies the \textbf{law invariance property}: for every $t\in[0,T]$ and every $\xi,\eta\in\S_2(\F)$, with $\P_\xi=\P_\eta$, it holds that
\[
V(t,\xi) = V(t,\eta).
\]
\end{Theorem}
\begin{proof}[\textbf{Proof.}]
We split the proof into two steps.

\vspace{1mm}

\noindent\textbf{\emph{Step 1.}} \emph{Only \textup{\ref{A_A,b,sigma}} and \textup{\ref{A_f,g}} hold.} Fix $t\in[0,T]$, $\xi,\eta\in\S_2(\F)$, with $\P_\xi=\P_\eta$, and let $U_\xi$, $U_\eta$ be $\Fc_t$-measurable random variables with uniform distribution on $[0,1]$, such that $\xi$ and $U_\xi$ (resp.\ $\eta$ and $U_\eta$) are independent.

By Remark \ref{R:V_non_ant} we can assume that $\xi=\xi_{\cdot\wedge t}$ and $\eta=\eta_{\cdot\wedge t}$, so, in particular, both $\xi$ and $\eta$ are $\Bc([0,T])\otimes\Fc_t$-measurable (this is needed in order to apply Lemma \ref{L:alpha=a} of Appendix \ref{App:LawInvariance}). Now, given $\alpha\in\Uc$ consider the function $\mathrm a\colon[0,T]\times\Omega\times C([0,T];H)\times[0,1]\rightarrow\Ur$ introduced in Lemma \ref{L:alpha=a}. By \eqref{EqLaw_bis} we have
\[
\Big((\xi_s)_{s\in[0,T]},(\mathrm a_s(\xi,U_\xi))_{s\in[t,T]},(B_s-B_t)_{s\in[t,T]}\Big) \overset{\mathscr L}{=} \Big((\xi_s)_{s\in[0,T]},(\alpha_s)_{s\in[t,T]},(B_s-B_t)_{s\in[t,T]}\Big),
\]
where $\overset{\mathscr L}{=}$ stands for equality in law (between random objects defined on $(\Omega,\Fc,\P)$). Then, notice that (here we use again that $\xi$ and $\eta$ are $\Bc([0,T])\otimes\Fc_t$-measurable, so, in particular, they are independent of $(B_s-B_t)_{s\in[t,T]}$)
\begin{equation}\label{Equality_in_law}
\Big((\xi_s)_{s\in[0,T]},(\alpha_s)_{s\in[t,T]},(B_s-B_t)_{s\in[t,T]}\Big) \overset{\mathscr L}{=} \Big((\eta_s)_{s\in[0,T]},(\beta_s)_{s\in[t,T]},(B_s-B_t)_{s\in[t,T]}\Big),
\end{equation}
where
\[
\beta \coloneqq  \big(\mathrm a_s(\eta,U_\eta)\big)_{s\in[0,T]}.
\]
Observe that $\beta\in\Uc$ and, by \eqref{Equality_in_law},
\[
\big((X_s^{t,\xi,\alpha})_{s\in[t,T]},(\alpha_s)_{s\in[t,T]}\big) \overset{\mathscr L}{=} \big((X_s^{t,\eta,\beta})_{s\in[t,T]},(\beta_s)_{s\in[t,T]}\big),
\]
where the above equality in law can be deduced from \eqref{Equality_in_law} proceeding along the same lines as in the proof of Proposition 1.137 in \cite{FabbriGozziSwiech}. As a consequence, it holds that
\[
J(t,\xi,\alpha) = J(t,\eta,\beta).
\]
Hence $J(t,\xi,\alpha)\leq V(t,\eta)$. From the arbitrariness of $\alpha$, we deduce that $V(t,\xi)\leq V(t,\eta)$. Changing the roles of $\xi$ and $\eta$ we get the opposite inequality, from which we deduce that $V(t,\xi)=V(t,\eta)$.

\vspace{1mm}

\noindent\textbf{\emph{Step 2.}} \emph{Assumptions \textup{\ref{A_A,b,sigma}}, \textup{\ref{A_f,g}}, \textup{\ref{A_f,g_cont}} hold.} Fix $t\in[0,T]$ and $\xi,\eta\in\S_2(\F)$, with $\P_\xi=\P_\eta$. As in the previous step, we exploit Remark \ref{R:V_non_ant} and take $\xi=\xi_{\cdot\wedge t}$ , $\eta=\eta_{\cdot\wedge t}$ (so, in particular, both $\xi$ and $\eta$ are $\Bc([0,T])\otimes\Fc_t$-measurable; this is needed in order to apply Lemma \ref{L:ExistUnif}).

\vspace{1mm}

\noindent\textbf{\emph{Substep 2.1.}} \emph{The discrete case.} Suppose that
\[
\P_{\xi} = \sum_{i=1}^m p_i\,\delta_{x_i},
\]
for some $\{x_1,\ldots,x_m\}\subset C([0,T];H)$, with $x_i\neq x_j$ if $i\neq j$, where $\delta_{x_i}$ is the Dirac measure at $x_i$ and $p_i>0$, with $\sum_{i=1}^m p_i=1$. Then, by Lemma \ref{L:ExistUnif} there exist two $\Fc_t$-measurable random variables $U_\xi$ and $U_\eta$, with uniform distribution on $[0,1]$, such that $\xi$ and $U_\xi$ (resp.\ $\eta$ and $U_\eta$) are independent. The claim then follows from \textbf{\emph{Step 1}}\!.

\vspace{1mm}

\noindent\textbf{\emph{Substep 2.2.}} \emph{The general case.} In the general case, we rely on the continuity of the map $\xi\mapsto V(t,\xi)$, defined from $\S_2(\F)$ into $\R$, which follows from Proposition \ref{P:Cont}. More precisely, we proceed by approximating $\xi$ and $\eta$. For $n\in \mathbb{N}$, let $\{C^n_i\}_{i\in \mathbb{N}}$ be a partition of $C([0,T];H)$ of Borel sets such that $\operatorname{diam}(C^n_i)<2^{-n}$. For each $i\in \mathbb{N}$, choose $x^n_i\in C^n_i$. Then, define
\[
\tilde\xi_n \coloneqq  \sum_{i=1}^\infty x^n_{i}\,\mathds{1}_{C^n_i}(\xi), \qquad\qquad \tilde\eta_n \coloneqq  \sum_{i=0}^\infty x^n_{i}\, \mathds{1}_{C^n_i}(\eta).
\]
Notice that $\tilde\xi_n,\tilde\eta_n\in\S_2(\F)$ and $\tilde\xi_n\rightarrow \xi$, $\tilde\eta_n\rightarrow\eta$ uniformly with respect to $\omega\in\Omega$. Moreover, $\tilde\xi_n$ and $\tilde\eta_n$ have the same law. By a diagonal argument, we can choose $N_n\in \mathbb{N}$ such that the sequences
$\{\xi_n\}_{n\in \mathbb{N}}$ and $\{\eta_n\}_{n\in \mathbb{N}}$, defined by
\begin{equation}\label{xi_n}
\xi_n \coloneqq  \sum_{i=1}^{N_n} x^n_i\, \mathds{1}_{C^n_i}(\xi), \qquad\qquad \eta_n \coloneqq  \sum_{i=1}^{N_n} x^n_i\, \mathds{1}_{C^n_i}(\eta),
\end{equation}
converge respectively to $\xi$ and $\eta$, both $\mathbb{P}$-a.s. and in $L^2(\Omega;C([0,T];H))$.
From \textbf{\emph{Substep 2.1}} we have
\[
V(t,\xi_n) = V(t,\eta_n), \qquad \forall\,n\in\N.
\]
Then, using the continuity of $V$, we can pass to the limit as $n\rightarrow\infty$ and conclude that $V(t,\xi)=V(t,\eta)$.
\end{proof}

\begin{Remark}\label{R:GinsteadF}
Suppose that \textup{\ref{A_A,b,sigma}}, \textup{\ref{A_f,g}}, \textup{\ref{A_f,g_cont}} hold. Thanks to the law invariance property stated in Theorem \ref{T:id-law}, in the definition of $V$ we can consider only $\xi\in\S_2(\Gc)$ rather than $\xi\in\S_2(\F)$ (recall that it was necessary to take $\xi\in\S_2(\F)$ in order to state and prove the dynamic programming principle, Theorem \ref{T:DPP}, where for instance the initial condition at time $s$ is $X^{t,\xi,\alpha}$ and $X^{t,\xi,\alpha}\in\S_2(\F)$ but, in general, $X^{t,\xi,\alpha}\notin\S_2(\Gc)$).
\end{Remark}

\begin{Remark}\label{R:Aliprantis}
In the finite-dimensional and non-path-dependent case, the law invariance property was already addressed in \cite{CossoPham19}, Proposition 3.1, under only \textup{\ref{A_A,b,sigma}} and \textup{\ref{A_f,g}}. Notice however that the proof of such a proposition is based on the measurable selection theorem stated in \cite{AliprantisBorder}, Corollary 18.23, which is unfortunately not true. For this reason, Theorem \ref{T:id-law} is also relevant in the finite-dimensional and non-path-dependent setting. Moreover, we emphasize that assuming only \textup{\ref{A_A,b,sigma}} and \textup{\ref{A_f,g}} is not enough for the validity of the law invariance property. To this regard, we give the following example.\\
\textbf{Example.} Let $T=1$, $H=\R^3$, $K=\R$, $\Ur=[0,1]$. We consider a non-path-dependent setting. The coefficients
\[
b,\,\sigma,\,f\colon[0,T]\times\R^3\times\Pc_2(\R^3)\times[0,1]\times\Pc([0,1]) \longrightarrow \R^3,\,\R^3,\,\R
\]
and $g\colon\R^3\times\Pc_2(\R^3)\rightarrow\R$ are given by
\[
b_t(x,\mu,u,\nu) \coloneqq
\left(\begin{array}{c}
0 \\
u \\
0
\end{array}\right), \quad \sigma_t(x,\mu,u,\nu) \coloneqq
\left(\begin{array}{c}
0 \\
0 \\
1
\end{array}\right), \quad
f_t(x,\mu,u,\nu) \coloneqq  0, \quad
g(x,\mu) \coloneqq  \mathds{1}_{\{\mu=\mu_0\}}
\]
for every $(t,x,\mu,u,\nu)\in[0,T]\times\R^3\times\Pc_2(\R^3)\times[0,1]\times\Pc([0,1])$, where $\mu_0\in\Pc_2(\R^3)$ is the probability measure defined as
\[
\mu_0 \coloneqq  \textup{Unif}\,(0,1)\otimes\textup{Bern}(1/2)\otimes\Nc(0,1),
\]
with $\textup{Unif}\,(0,1)$ being the uniform distribution on $[0,1]$, $\textup{Bern}(1/2)$ the Bernoulli distribution with parameter $1/2$, $\Nc(0,1)$ the standard Gaussian distribution.

We now fix the probabilistic setting. Consider the probability spaces $(\Omega^\circ,\Fc^\circ,\P^\circ)$ and  $(\Omega^\text{\tiny\rm 1},\Fc^\text{\tiny\rm 1},\P^\text{\tiny\rm 1})$ where $\Omega^\circ=[0,1]$, $\Fc^\circ$ its Borel $\sigma$-algebra and  $\P^\circ$ is the Lebesgue measure on the unit interval, while $\Omega^\text{\tiny\rm 1}=\{\omega^\text{\tiny\rm 1}\in C([0,1];\R)\colon\omega^\text{\tiny\rm 1}_0=0\}$, $\Fc^\text{\tiny\rm 1}$ its Borel $\sigma$-algebra and $\P^\text{\tiny\rm 1}$ is the Wiener measure on $(\Omega^\text{\tiny\rm 1},\Fc^\text{\tiny\rm 1})$. We then define $\Omega\coloneqq \Omega^\circ\times\Omega^{\text{\tiny\rm 1}}$, $\Fc$ the completion of $\Fc^\circ\otimes\Fc^{\text{\tiny\rm 1}}$ with respect to $\P^\circ\otimes\P^{\text{\tiny\rm 1}}$ and by $\P$ the extension of $\P^\circ\otimes\P^{\text{\tiny\rm 1}}$ to $\Fc$. We also denote by $\Gc\coloneqq \Fc^\circ\otimes\{\emptyset,\Omega^\text{\tiny\rm 1}\}$ the canonical extension of $\Fc^\circ$ to the product space $\Omega$. Finally, we denote by $B=(B_t)_{t\in[0,1]}$ the canonical process $B_t(\omega^\circ,\omega^\text{\tiny\rm 1})\coloneqq \omega^\text{\tiny\rm 1}_t$, $\forall\,t\in[0,1]$. Notice that, under the probability measure $\P$, the process $B$ is a real-valued Brownian motion. Then, in the present context the lifted value function is given by
\[
V(t,\xi) = \sup_{\alpha\in\Uc} \E\Big[\mathds{1}_{\big\{\P_{X_1^{t,\xi,\alpha}}=\mu_0\big\}}\Big], \qquad \forall\,(t,\xi)\in[0,1]\times L^2(\Omega,\Fc_t,\P;\R^3),
\]
where
\[
\xi = \
\left(\begin{array}{c}
\xi^1 \\
\xi^2 \\
\xi^3
\end{array}\right)
\qquad
\text{ and }
\qquad
X_1^{t,\xi,\alpha} = \
\left(\begin{array}{c}
\vspace{1mm}\xi^1 \\
\vspace{1mm}\displaystyle\xi^2 + \int_t^1 \alpha_s\,ds \\
\xi^3 + B_1 - B_t
\end{array}\right).
\]
Now, let $\xi\colon\Omega\rightarrow\R^3$ be given by
\[
\xi(\omega^\circ,\omega^\text{\tiny\rm 1}) \coloneqq  \
\left(\begin{array}{c}
\,\,\,\omega^\circ \\
0 \\
0
\end{array}\right),
\qquad \forall\,(\omega^\circ,\omega^\text{\tiny\rm 1})\in\Omega.
\]
Notice that $\xi^1$ has distribution $\textup{Unif}\,(0,1)$. Moreover, $\xi$ is $\Gc$-measurable and generates the $\sigma$-algebra $\Gc$ itself, namely $\Gc=\sigma(\xi)$. Define $\eta\colon\Omega\rightarrow\R^3$ and $Z\colon\Omega\rightarrow[0,1]$ by
\[
\eta \coloneqq  \
\left(\begin{array}{c}
2\xi\,\mathds{1}_{\{\xi\leq1/2\}} + (2\xi-1)\,\mathds{1}_{\{\xi>1/2\}} \\
0 \\
0
\end{array}\right),
\qquad\qquad
Z \coloneqq  \mathds{1}_{\{\xi\leq1/2\}}.
\]
Notice that $\eta$ and $Z$ are $\Gc$-measurable and independent. Moreover, the first component of $\eta$, namely $\eta^1$, has distribution $\textup{Unif}\,(0,1)$, while $Z$ has distribution $\textup{Bern}(1/2)$.

Let us prove that $V(0,\xi)\neq V(0,\eta)$ and, in particular, $V(0,\xi)=0$ while $V(0,\eta)=1$. If the initial condition at time $t=0$ is $\eta$, then taking the control process $\alpha_s^*=Z$, $\forall\,s\in[0,T]$, we get
\[
X_1^{0,\eta,\alpha^*} = \
\left(\begin{array}{c}
\eta^1 \\
Z \\
B_1
\end{array}\right).
\]
Notice that $\P_{X_1^{0,\eta,\alpha^
*}}=\mu_0$, which proves that $V(0,\eta)=1$. On the other hand, when the initial condition is $\xi$, for every $\alpha\in\Uc$ we have
\[
X_1^{0,\xi,\alpha}(\omega^\circ,\omega^\text{\tiny\rm 1}) = \
\left(\begin{array}{c}
\vspace{1mm}\omega^\circ \\
\vspace{1mm}\displaystyle\int_0^1 \alpha_s(\omega^\circ,\omega^\text{\tiny\rm 1})\,ds \\
\omega_1^\text{\tiny\rm 1}
\end{array}\right),
\qquad \forall\,(\omega^\circ,\omega^\text{\tiny\rm 1})\in\Omega.
\]
  Let us prove that $V(0,\xi)=0$. Suppose on the contrary that $V(0,\xi)=1$. This implies that $X_1^{0,\xi,\alpha}$ has distribution $\mu_0$, so in particular $\int_0^1\alpha_s\,ds$ has Bernoulli distribution with parameter $1/2$. Recalling that the control process takes values in $[0,1]$, there exists some set $E\in\Fc$ such that $\alpha_s(\omega)=1_E(\omega)$, $\forall\,(s,\omega)\in[0,1]\times\Omega$. Since $\alpha$ is progressively measurable, $E\in\Gc$; in other words, $\alpha$ depends only on $\omega^\circ$. It follows that the random variable $\int_0^1\alpha_s\,ds$ cannot be independent of $\xi$ $($and, a fortiori, of $(\xi,B_1)$$)$, unless it is a constant. This contradicts the fact that $\int_0^1\alpha_s\,ds$ has Bernoulli distribution of parameter $1/2$ and proves that $V(0,\xi)=0$.
\end{Remark}

In conclusion, under assumptions \textup{\ref{A_A,b,sigma}}, \textup{\ref{A_f,g}}, \textup{\ref{A_f,g_cont}}, we can define the \emph{value function} $v\colon[0,T]\times\Pc_2(C([0,T];H))\rightarrow\R$ as
\begin{equation}\label{v}
v(t,\mu) = V(t,\xi), \qquad \forall\,(t,\mu)\in[0,T]\times\Pc_2(C([0,T];H)),
\end{equation}
for any $\xi\in\S_2(\F)$ with $\P_\xi=\mu$. By Theorem \ref{T:DPP} we immediately deduce the dynamic programming principle for $v$.

\begin{Corollary}\label{C:DPP}
Suppose that \textup{\ref{A_A,b,sigma}}, \textup{\ref{A_f,g}}, \textup{\ref{A_f,g_cont}} hold. The value function $v$ satisfies the \textbf{dynamic programming principle}: for every $t,s\in[0,T]$, with $t\leq s$, and every $\mu\in\Pc_2(C([0,T];H))$ it holds that
\[
v(t,\mu) = \sup_{\alpha\in\Uc}\bigg\{\E\bigg[\int_t^s f_r\big(X^{t,\xi,\alpha},\P_{X^{t,\xi,\alpha}_{\cdot\wedge r}},\alpha_r,\P_{\alpha_r}\big)\,dr\bigg] + v\big(s,\P_{X^{t,\xi,\alpha}}\big)\bigg\},
\]
for any $\xi\in\S_2(\F)$ with $\P_\xi=\mu$.
\end{Corollary}

\begin{Remark}
Recall from Remark \ref{R:V_non_ant} that $V$ is non-anticipative, namely $V(t,\xi)=V(t,\xi_{\cdot\wedge t})$, for every $(t,\xi)\in[0,T]\times\S_2(\F)$. As a consequence, the value function $v$ satisfies the following \emph{non-anticipativity property:}
\[
v(t,\mu) = v(t,\mu_{[0,t]}),
\]
for every $(t,\mu)\in[0,T]\times C([0,T];H)$, where we denote by $\mu_{[0,t]}$ the measure $\mu\circ\big( (x_s)_{s\in[0,T]}\mapsto (x_{s\wedge t})_{s\in[0,T]} \big)^{-1}$.
\end{Remark}

\section{Pathwise derivatives in the Wasserstein space and It\^o's formula}
\label{S:Ito}

This section is devoted to the proof of It\^o's formula for a real-valued function $\varphi$ defined on $[0,T]\times\Pc_2(C([0,T];H))$. Such a formula involves the so-called pathwise derivatives in the Wasserstein space that we now define. In the present section we substantially follow \cite[Section 2]{WuZhang18PPDE} (see also \cite{WuZhang18}), extending their framework to our more general setting with $H$ being a real separable Hilbert space (not necessarily a Euclidean space as in \cite{WuZhang18PPDE}).

\subsection{Notations}

In order to define the pathwise derivatives, we need to extend the canonical space $C([0,T];H)$ to the space of c\`adl\`ag paths $D([0,T];H)$, which we endow with the Skorokhod topology (in what follows, we denote paths in $D([0,T];H)$ using $\hat\cdot$ in order to distinguish them from paths in $C([0,T];H)$; we do the same with other mathematical objects).
In order to introduce the Skorohod topology, we follow \cite{Billingsley} and introduce the following metric on $D([0,T];H)$ (which corresponds to metric $d^\circ$ in \cite[formula (12.16)]{Billingsley}):
\[
    d_{\textup{Sk}}(\hat x,\hat y) \coloneqq  \inf_{\boldsymbol\lambda\in\mathbf\Lambda}\bigg\{\max\bigg(\sup_{\substack{s<t\\s,t\in[0,T]}}\bigg|\log\frac{\boldsymbol\lambda(t)-\boldsymbol\lambda(s)}{t-s}\bigg|,\sup_{t\in[0,T]}\big|\hat x_t - \hat y_{\boldsymbol\lambda(t)}\big|_H\bigg)\bigg\},
\]
where $\mathbf\Lambda$ denotes the set of strictly increasing and continuous maps $\boldsymbol\lambda\colon[0,T]\rightarrow[0,T]$, satisfying $\boldsymbol\lambda(0)=0$ and $\boldsymbol\lambda(T)=T$. We recall 
  from \cite[Theorem 12.2]{Billingsley} (in \cite{Billingsley} the case $H=\R$ is considered, however the same proof applies to the general case of a real separable Hilbert space $H$) that $(D([0,T];H),d_{\textup{Sk}})$ is a Polish space.
Moreover,
$d_{\rm Sk}$ induces on $C([0,T];H)$ the topology of the uniform convergence (see \cite{Billingsley}, p.\ 124), and $C([0,T];H)$ is a Borel subset of $\big(D([0,T];H),d_{\rm Sk}\big)$ (apply e.g.\ result number (2) at p.\ 67 in \cite{Federer1969}, with $C=X=\big(C([0,T];H),\|\cdot\|_T\big)$, $Y=\big(D([0,T],H),d_{\rm Sk}\big)$, and as $f$ take the canononical embedding).
We define the spaces
\[
\mathscr{H} \coloneqq  [0,T]\times \Pc_2\big(C([0,T];H)\big), \qquad\qquad \mathscr{\hat H} \coloneqq  [0,T]\times \Pc_2\big(D([0,T];H)\big).
\]
For every $(t,\mu)\in \mathscr{H}$, we denote by $\mu_{[0,t]}$ the measure $\mu\circ\big( (x_s)_{s\in[0,T]}\mapsto (x_{s\wedge t})_{s\in[0,T]} \big)^{-1}$. We define similarly $\hat\mu_{[0,t]}$, for every $(t,\hat\mu)\in\mathscr{\hat H}$. We then equip $\mathscr{H}$ and $\mathscr{\hat H}$ with the following pseudo-distances, respectively:
\begin{align*}
d_{\mathscr{H}}\big((t,\mu),(t',\mu')\big) &\coloneqq  \Big(|t-t'|^2+\Wc_2\big(\mu_{[0,t]},\mu_{[0,t']}'\big)^2\Big)^{\frac{1}{2}}\,, \qquad (t,\mu),(t',\mu')\in \mathscr{H}, \\
d_{\mathscr{\hat H}}\big((t,\hat\mu),(t',\hat\mu')\big) &\coloneqq  \Big(|t-t'|^2+\widehat\Wc_2\big(\hat\mu_{[0,t]},\hat\mu_{[0,t']}'\big)^2\Big)^{\frac{1}{2}}\,, \qquad (t,\hat\mu),(t',\hat\mu')\in \mathscr{\hat H},
\end{align*}
where $\widehat\Wc_2$ is defined as (denoting ${\rm D}:=D([0,T];H)$)
\begin{align*}
\widehat\Wc_2(\hat\mu,\hat\mu') &\coloneqq  \inf\bigg\{\int_{{\rm D}\times{\rm D}}d_{\rm Sk}(\hat x,\hat y)^2\,\hat\pi(d\hat x,d\hat y)\colon
  \hat\pi \in \Pc({\rm D}\times{\rm D})\\
 & \hspace{2cm} \text{ such that } \hat\pi(\cdot\times{\rm D})= \hat\mu \mbox{ and }\hat\pi({\rm D}\times\cdot)=\hat\mu'\bigg\}^{1\over 2}\,,
\end{align*}
for every $\hat\mu,\hat\mu'\in \Pc_2(D([0,T];H))$.

\begin{Remark}
  Notice that a function $\varphi\colon\mathscr{H}\rightarrow\R$ (resp.\ $\hat\varphi\colon\mathscr{\hat H}\rightarrow\R$) that is
  measurable with respect to $d_{\mathscr{H}}$ (resp.\ $d_{\mathscr{\hat H}}$) satisfies the \textup{non-anticipativity property:}
\[
\varphi(t,\mu) = \varphi(t,\mu_{[0,t]}), \quad \forall\,(t,\mu)\in\mathscr{H} \qquad \big(\text{\textup{resp.\ $\hat\varphi(t,\hat\mu)=\hat\varphi(t,\hat\mu_{[0,t]})$, \quad $\forall\,(t,\hat\mu)\in\mathscr{\hat H}$}}\big).
\]
\end{Remark}

\begin{Remark}
Notice that there is a natural injection
$$
i_0:\Pc_2\big(C([0,T];H)\big)\to
\Pc_2\big(D([0,T];H)\big),
\qquad \mu \mapsto i_0(\mu)
$$
where $i_0(\mu)(E)=\mu (E\cap C([0,T];H))$ for all $E$ Borel subset of $D([0,T];H)$.
This induces an injection $i=\mathscr{H} \to \mathscr{\hat H}$
given by $i(t,\mu)=(t,i_0(\mu))$.
We claim that
the restriction of $d_{\mathscr{\hat H}}$ to $i(\mathscr H)\times
i(\mathscr H)$ gives rise to the same topology on $i(\mathscr H)$ induced by $d_{\mathscr H}$ through the injection $i$.
Indeed, this is a consequence of the following property. Let $\{\mu_n\}_{n\in\N}\subset\Pc_2(C([0,T];H))$ and $\mu\in\Pc_2(C([0,T];H))$. Then, denoting ${\rm C}:=C([0,T];H)$, it holds that
\[
\Wc_2(\mu_n,\mu)\coloneqq\inf\bigg\{\int_{{\rm C}\times{\rm C}}\|x-y\|_T^2\, \pi(dx,dy)\colon
  \pi \in \Pc({\rm C}\times{\rm C}),\;\pi(\cdot\times{\rm C})= \mu_n,\;\pi({\rm C}\times\cdot)=\mu\bigg\}^{\frac{1}{2}}  \overset{n\rightarrow\infty}{\longrightarrow}  0
\]
if and only if $($denoting ${\rm D}:=D([0,T];H)$$)$
\begin{align*}
&\widehat\Wc_2(i_0(\mu_n),i_0(\mu))\\
&\coloneqq\inf\bigg\{\int_{{\rm D}\times{\rm D}}d_{\textup{Sk}}(\hat x,\hat y)^2\,\hat\pi(d\hat x,d\hat y)\colon
  \hat\pi \in \Pc({\rm D}\times{\rm D}),\;\hat\pi(\cdot\times{\rm D})= i_0(\mu_n),\;\hat\pi({\rm D}\times\cdot)=i_0(\mu)\bigg\}^{\frac{1}{2}}  \overset{n\rightarrow\infty}{\longrightarrow} 0.
\end{align*}
To prove the latter equivalence we first observe that the above probability measure $\hat\pi$ with marginals $i_0(\mu_n)$ and $i_0(\mu)$ satisfies $\hat\pi({\rm C}\times{\rm C})=1$. Then, the following equality holds:
\[
\widehat\Wc_2(i_0(\mu_n),i_0(\mu))=\inf\bigg\{\int_{{\rm C}\times{\rm C}}d_{\textup{Sk}}(x,y)^2\,\pi(dx,dy)\colon
  \pi \in \Pc({\rm C}\times{\rm C}),\;\pi(\cdot\times{\rm C})= \mu_n,\;\pi({\rm C}\times\cdot)=\mu\bigg\}^{\frac{1}{2}}.
\]
Now, we recall from \cite[Definition 6.8 and Theorem 6.9]{Vi09}) that $\lim_{n\rightarrow\infty}\Wc_2(\mu_n,\mu)=0$ if and only if for all continuous functions $\varphi\colon(C([0,T];H),\|\cdot\|_T)\rightarrow\R$ with $|\varphi(x)|\leq c(1+\|x\|_T^2)$, $c\in\R$, one has
\begin{equation}\label{Uniform}
\int_{\rm C} \varphi(x)\,\mu_n(dx) \ \overset{n\rightarrow\infty}{\longrightarrow} \ \int_{\rm C}\varphi(x)\,\mu(dx).
\end{equation}
Similarly, since $(C([0,T];H),d_{\textup{Sk}})$ is a Radon separable metric space\footnote{This comes from the fact that such a space is topologically equivalent to the Polish space $\big(C([0,T];H),\|\cdot\|_T\big)$.} (see \cite[Definition 5.1.4]{AGS08}), from the proof of Proposition 7.1.5. in \cite{AGS08} we deduce that $\lim_{n\rightarrow\infty}\widehat\Wc_2(i_0(\mu_n),i_0(\mu))=0$ if and only if for all continuous functions $\varphi\colon(C([0,T];H),d_{\textup{Sk}})\rightarrow\R$ with $|\varphi(x)|\leq c(1+d_{\textup{Sk}}(x,0)^2)$, $c\in\R$, one has
\begin{equation}\label{Sk}
\int_{\rm C} \varphi(x)\,\mu_n(dx) \ \overset{n\rightarrow\infty}{\longrightarrow} \ \int_{\rm C}\varphi(x)\,\mu(dx).
\end{equation}
Hence, recalling that the Skorohod topology relativized to $C([0,T];H)$ coincides with the uniform topology, 
 it follows that the set of real-valued continuous functions on $(C([0,T];H),\|\cdot\|_T)$ coincides with the set of real-valued continuous on $(C([0,T];H),d_{\textup{Sk}})$. Finally, concerning the  subquadratic growth, it holds that $d_{\textup{Sk}}(x,0)=\|x\|_T$, which shows that the class of functions involved in \eqref{Uniform} and \eqref{Sk} are the same.
\end{Remark}

\noindent We also introduce the lifted spaces
\[
\mathfrak{H} \coloneqq  [0,T]\times L^2\big(\Omega;C([0,T];H)\big), \qquad\qquad \mathfrak{\hat H} \coloneqq  [0,T]\times L^2\big(\Omega;D([0,T];H)\big).
\]
\begin{Remark}\label{R:U_tilde}
To alleviate notation, in the present Section \ref{S:Ito} we work on the same probability space $(\Omega,\Fc,\P)$ adopted in the rest of the paper. Notice however that, for the definition of pathwise derivatives, $(\Omega,\Fc,\P)$ can be replaced by any other probability space which supports a random variable having uniform distribution on $[0,1]$. See also Remark \ref{R:Atomless}.
\end{Remark}

\noindent Finally, we introduce the following notation.

\begin{Notation}{\bf(Ntn$_{\hat{\mathbb P}}$)}\label{Ntn_P}
For every $\xi\in L^2(\Omega;C([0,T];H))$, we denote by $\hat\P_\xi$ the law of $\xi$ on $D([0,T];H)$, while we recall that $\P_\xi$ denotes the law of $\xi$ on $C([0,T];H)$. So, in particular, $\hat\P_\xi\in\Pc_2(D([0,T];H))$, while $\P_\xi\in\Pc_2(C([0,T];H))$. Clearly, it holds that $\hat\P_\xi(\mathfrak B)=\P_\xi(\mathfrak B)$, for every Borel subset $\mathfrak B$ of $C([0,T];H)$. Notice that in \cite{WuZhang18PPDE} the probability $\hat\P_\xi$ is denoted simply by $\P_\xi$ (see the beginning of Section 2.4 in \cite{WuZhang18PPDE}).
\end{Notation}

\subsection{Pathwise derivatives for a map $\hat\varphi\colon\mathscr{\hat H}\rightarrow\R$ and It\^o's formula}

We start with the definition of pathwise time derivative for a map $\hat\varphi\colon\mathscr{\hat H}\rightarrow\R$.

\begin{Definition}\label{Def:HorizontalDer}
Let $\hat\varphi\colon\mathscr{\hat H}\rightarrow\R$ be a non-anticipative function. We say that $\hat\varphi$ is pathwise differentiable in time at $(t,\hat\mu)\in\mathscr{\hat H}$, with $t<T$, if the following limit exists and is finite:
\[
\partial_t \hat\varphi(t,\hat\mu) \coloneqq  \lim_{\delta\rightarrow 0^+}\frac{\hat\varphi(t+\delta,\hat\mu_{[0,t]})-\hat\varphi(t,\hat \mu)}{\delta}.
\]
At time $t=T$, we define
\[
\partial_t \hat\varphi(t,\hat\mu) \coloneqq  \lim_{t\rightarrow T^-}\partial_t \hat\varphi(t,\hat\mu),
\]
when the limit exists and is finite. We refer to $\partial_t \hat\varphi$ as the \textbf{pathwise time derivative} (or \textbf{horizontal derivative}) \textbf{of $\hat\varphi$ at $(t,\hat\mu)$}.
If $ \partial _t\hat{\varphi}$ exists everywhere as a function $\mathscr{\hat{H}}\rightarrow \mathbb{R}$, we refer to it as the \textbf{pathwise time derivative of $\hat{\varphi}$}.
\end{Definition}

\begin{Remark}\label{R:NA-time_deriv}
Notice that $\partial_t\hat\varphi$ is a non-anticipative function, namely $\partial_t\hat\varphi(t,\hat\mu)=\partial_t\hat\varphi(t,\hat\mu_{[0,t]})$, for every $(t,\hat\mu)\in\mathscr{\hat H}$.
\end{Remark}

In order to define the pathwise measure derivative, we need to consider the lifting of $\hat\varphi$.

\begin{Definition}
Given $\hat\varphi\colon\mathscr{\hat H}\to \R$, we say that $\hat\Phi\colon\mathfrak{\hat H}\to\R$ is a \textbf{lifting} of $\hat\varphi$ if
\[
\hat\Phi(t,\hat\xi) = \hat\varphi(t,\P_{\hat\xi}), \qquad \forall\,(t,\hat\xi)\in\mathfrak{\hat H},
\]
where we recall that $\P_{\hat\xi}$ stands for the law of the random variable $\hat\xi\in L^2(\Omega;D([0,T];H))$.
\end{Definition}

\begin{Definition}\label{D:PathSpaceDeriv}
Let $\hat\Phi\colon\mathfrak{\hat H}\rightarrow\R$ be a non-anticipative function, namely $\hat\Phi(t,\hat\xi)=\hat\Phi(t,\hat\xi_{\cdot\wedge t})$, for every $(t,\hat\xi)\in\mathfrak{\hat H}$. We say that $\hat\Phi$ is \textbf{pathwise differentiable in space} at $(t,\hat\xi)\in\mathfrak{\hat H}$ if there exists $D\hat\Phi(t,\hat\xi)\in L^2(\Omega;H)$ such that
\[
  \lim_{Y\rightarrow 0}
  \frac{\left| \hat\Phi(t,\hat\xi+Y\mathds{1}_{[t,T]})-\hat\Phi(t,\hat\xi)
  -\E\big[\langle D\hat\Phi(t,\hat\xi),Y\rangle_H\big]\right|}{|Y|_{L^2(\Omega;H)}}=0.
\]
We refer to $D\hat\Phi(t,\hat\xi)$ as the \textbf{pathwise space derivative} (or \textbf{vertical derivative}) \textbf{of $\hat\Phi$ at $(t,\hat\xi)$}.
If $ D\hat{\Phi}$ exists everywhere as a function $\mathfrak{\hat{H}}\rightarrow L^2(\Omega;H)$, we refer to it as the \textbf{pathwise space derivative of $\hat{\varphi}$}.
\end{Definition}

\begin{Remark}\label{R:NA-DPhi}
Notice that, if $\hat{\Phi}$ is pathwise differentiable in space at $(t,\hat\xi)$, then it is pathwise differentiable in space at $(t,\hat{\xi}')$ for every $\hat{\xi}'\in L^2(\Omega; D([0,T];H))$ such that
$\hat{\xi}_{t\wedge \cdot}=\hat{\xi}'_{t\wedge \cdot}$ $\mathbb{P}$-a.s., and, in such a case,
$D\hat\Phi(t,\hat\xi)=D\hat\Phi(t,\hat\xi')$ $\P$-a.s.
\end{Remark}

We can then give, similarly to \cite[Definition 5.22]{CD14},
the following definition.

\begin{Definition}\label{D:MeasureDerivative}
Let $\hat\varphi\colon\mathscr{\hat H}\rightarrow\R$ be a non-anticipative function and $(t,\hat\mu)\in\mathscr{\hat H}$.
We say that $\hat\varphi$ is \textbf{pathwise differentiable in measure} at $(t,\hat\mu)$ if its lifting $\hat\Phi$ is pathwise differentiable in space at some $(t,\hat\xi)\in\mathfrak{\hat H}$ such that $\P_{\hat\xi}=\hat\mu$.

Moreover, we say that $\hat\varphi$ admits
\textbf{pathwise measure derivative}
at $(t,\hat\mu)$
if its lifting $\hat \Phi$ is pathwise differentiable in space
at every $(t,\hat{\xi})\in \mathfrak{\hat H}$ such that $\mathbb{P}_{\hat{\xi}}=\hat{ \mu}$
and if there exists a masurable function
$\hat{g}\colon D([0,T];H)\rightarrow H$
such that,
for all $(t,\hat{\xi})\in \mathfrak{\hat H}$ with $\mathbb{P}_{\hat{\xi}}=\hat{ \mu}$,
\begin{equation}\label{DPhi=partial_mu}
D\hat\Phi(t,\hat\xi) = \hat g(\hat\xi)
  \qquad \P\text{-a.s.}
\end{equation}
 The map $\hat g$ (which is $\hat{\mu}$-a.s.\ uniquely determined)
 is called \textbf{pathwise measure derivative of $\hat\varphi$ at $(t,\hat\mu)$}\footnote{Notice that, if a function is pathwise differentiable in measure at some point and the related pathwise space derivative is continuous (at least in a neighborhood), then it admits the pathwise measure derivative at that point, see e.g.\ \cite[Theorem 6.5]{carda12} or \cite[Proposition 5.25]{cardelbook} in finite dimension and our Lemma \ref{L:General}. Without the continuity assumption of the pathwise space derivative such a result is not obvious. See also \cite{GangboTudor}.}.

 Finally, if $\hat \varphi$ admits pathwise measure derivative at every $(t,\hat{\mu})\in \mathfrak{\hat{H}}$, then the function
 \begin{equation*}
   \partial _\mu \hat\varphi\colon \mathscr{\hat H}\times D([0,T];H) \longrightarrow H, \qquad\qquad (t,\hat\mu, \hat x) \longmapsto   \partial _\mu\hat{\varphi} (t,\hat\mu,\hat x)
 \end{equation*}
 such that, for every $(t,\hat{\xi}) \in \mathfrak{\hat H}$,
 $\partial _\mu\hat{\varphi}(t,\mathbb{P}_{\hat{\xi}},\cdot)$ is measurable and
 $D\hat\Phi(t,\hat\xi) = \partial _\mu\hat{\varphi}(t,\mathbb{P}_{\hat{\xi}},\hat\xi)$
 $\mathbb{P}$-a.s., is called the \textbf{pathwise measure derivative of $\hat{\varphi}$}\footnote{We stress the fact that $ \partial _\mu\hat{\varphi}(t,\mathbb{P}_{\hat{\xi}},\cdot)$ is \emph{a-priori} uniquely determined only $\mathbb{P}_{\hat{\xi}}$-a.s.}.
\end{Definition}



  In the following lemma we state, under general conditions, 
  the existence of the pathwise measure derivative.
In order to do it, we proceed similarly to what is done in \cite[Section 2.3]{WuZhang18PPDE} (see also \cite[Section 5.3]{CD14}) in the finite-dimensional case.
The proof of the lemma is postponed in Appendix \ref{App:PathDeriv}.

\begin{Lemma}\label{L:General}
Fix $(t,\hat\xi)\in\mathfrak{\hat H}$. Let $\hat\varphi\colon \mathscr{\hat H}\to\R$ be such that its lifting
$\hat\Phi$ admits a continuous pathwise space derivative $D\hat\Phi$ on the set $\{t\}\times \mathcal O_{\hat\xi}$, where $\mathcal O_{\hat\xi}$ is a neighborhood of $\hat\xi$ in $L^2(\Omega;D([0,T];H))$. Then, there exists a measurable function
$\hat g\colon D([0,T];H)\rightarrow H$
and
\begin{equation}
  \label{2018-11-06:10}
  D\hat\Phi(t,\hat\xi) = \hat g(\hat\xi),
  \qquad \P\text{-a.s.}
\end{equation}
Let $\hat\xi'$ be such that $\P_{\hat\xi'}=\P_{\hat\xi}$. If in addition $\hat\Phi$ admits a continuous pathwise space derivative on the set $\{t\}\times \mathcal O_{\hat\xi'}$, with $\mathcal O_{\hat\xi'}$ being a neighborhood of $\hat\xi'$, then \eqref{2018-11-06:10} holds true with $\hat\xi$ replaced by $\hat\xi'$.

Hence,
if the pathwise space derivative $D\hat\Phi(t,\hat \xi)$ at $(t,\hat\xi)$ exists
for every $(t,\hat \xi)\in\mathfrak{\hat H}$ and if $D\hat{\Phi}$ is continuous,
then there exists
the pathwise measure derivative $ \partial _\mu\hat{\varphi}$ of $\hat{\varphi}$.

If in addition
the map $ \partial _\mu \hat\varphi(t,\cdot,\cdot)\colon\Pc_2(D([0,T];H))\times D([0,T];H)\rightarrow\R$ is continuous
for every $t\in[0,T]$, then $ \partial _\mu\hat \varphi$ is uniquely defined.

Finally,assume that 
the pathwise space derivative $D\hat\Phi$ exists everywhere and is uniformly continuous. Then
$ \partial _\mu\hat{\varphi}$ is measurable.
\end{Lemma}


\begin{Remark}[\textbf{Non-anticipativity property of $\partial_\mu\hat\varphi$.}]\label{R:NA-measure-deriv}
Let $\hat\varphi\colon\mathscr{\hat H}\rightarrow\R$ be a non-anticipative function. Suppose that
  there exists
  the pathwise measure derivative
of $\hat{\varphi}$.
Then, thanks to Remark \ref{R:NA-DPhi} and equality \eqref{DPhi=partial_mu}, $\partial_\mu\hat\varphi$ is a \textup{non-anticipative function} in the sense that,
for every $(t,\hat{\xi})\in  \mathfrak{\hat H}$,
\begin{equation*}
  \partial_\mu \hat\varphi(t,\mathbb{P}_{\hat{\xi}})(\hat\xi)=
\partial_\mu \hat\varphi(t,(\mathbb{P}_{\hat{\xi}})_{[0,t]})(\hat\xi_{\cdot\wedge t})\qquad\P\textrm{-a.s.}
\end{equation*}
or, equivalently,
\[
\partial_\mu\hat\varphi(t,\hat\mu)(\hat x) = \partial_\mu \hat\varphi(t,\hat\mu_{[0,t]})(\hat{x}_{\cdot\wedge t}),
\qquad\hat\mu\text{-a.e.}
\]
for every
$(t,\hat\mu,\hat x)\in\mathscr{\hat H}\times D([0,T];H)$.
\end{Remark}

Finally, we define the pathwise derivative of second-order $\partial_x\partial_{\mu}\hat\varphi$.

\begin{Definition}\label{D:2nd-measure-deriv}
   Let $\hat\varphi\colon\mathscr{\hat H}\rightarrow\R$ be a non-anticipative function and $(t,\hat\mu)\in\mathscr{\hat H}$.
  Suppose that:
\begin{enumerate}[1)]
\item there exists the pathwise measure derivative $\partial_\mu\hat\varphi$;
\item for every $t\in[0,T]$, the map $\partial_\mu\hat\varphi(t,\cdot)(\cdot)\colon\Pc_2(D([0,T];H))\times D([0,T];H)\rightarrow H$ is continuous (hence, by Lemma~\ref{L:General}, $ \partial _\mu\hat{\varphi}$ is uniquely determined).
\end{enumerate}
Given $\hat x\in D([0,T];H)$, we say that $\hat\varphi$ is \textbf{pathwise differentiable in measure and space at $(t,\hat\mu,\hat x)$} if there exists an operator $\partial_x\partial_\mu\hat\varphi(t,\hat\mu)(\hat x)\in\Lc(H)$ such that
\begin{equation*}
  \lim_{h\rightarrow 0}
  \frac{\big|\partial_\mu\hat\varphi(t,\hat\mu)(\hat x + h\,\mathds{1}_{[t,T]}) - \partial_\mu \hat\varphi(t,\hat\mu)(\hat x) - \partial_x\partial_\mu \hat\varphi(t,\hat\mu)(\hat x)h\big|_H }{|h|_H}= 0.
\end{equation*}
We refer to $\partial_x\partial_\mu\hat\varphi(t,\hat\mu)(\hat x)$ as the \textbf{second-order pathwise derivative in measure and space of $\hat\varphi$ at $(t,\hat\mu,\hat x)$}.

If $ \partial _x \partial _\mu\hat{\varphi}$ exists everywhere as function
$\mathscr{\hat H}\times D([0,T];H)\rightarrow\Lc(H)$,
we refer to it as the \textbf{pathwise derivative in measure and space of $\hat{\varphi}$}.
\end{Definition}


\begin{Remark}\label{R:NA-2nd-measure-deriv}
Recalling Remark \ref{R:NA-measure-deriv}, we see that $\partial_x\partial_\mu\hat\varphi$ is a non-anticipative function, namely it holds that $\partial_x\partial_\mu\hat\varphi(t,\hat\mu)(\hat\xi)=\partial_x\partial_\mu\hat\varphi(t,\hat\mu_{[0,t]})(\hat\xi_{\cdot\wedge t})$, $\P$-a.s., for every $(t,\hat\xi)\in\mathfrak{\hat H}$, with $\P_{\hat\xi}=\hat\mu$, or, equivalently,
\[
\partial_x \partial_\mu\hat\varphi(t,\hat\mu)(\hat x) = \partial_x \partial_\mu \hat\varphi(t,\hat\mu_{[0,t]})(\hat{x}_{\cdot\wedge t}),
\qquad\hat\mu\text{-a.e.}
\]
for every
$(t,\hat\mu,\hat x)\in\mathscr{\hat H}\times D([0,T];H)$.

\end{Remark}

\begin{Definition}\label{D:C1,2}
We denote by $\boldsymbol C^{1,2}(\mathscr{\hat H})$ the set of non-anticipative functions $\hat\varphi\colon\mathscr{\hat H}\rightarrow\R$ such that:
\begin{enumerate}[\upshape 1)]
\item
  the lifting $\hat\Phi$ of $\hat\varphi$ admits a continuous pathwise space derivative $D\hat\Phi$ on $\mathfrak{\hat H}$ (hence, by Lemma~\ref{L:General}, there exists the pathwise measure derivative $ \partial _\mu \hat{\varphi}$);
\item $\hat{\varphi}, \partial _\mu\hat{\varphi}$ are continuous;
\item
  there exist
the pathwise time derivative $ \partial _t\hat{\varphi}$, the second-order pathwise derivative in measure and space $ \partial _x \partial _\mu\hat{\varphi}$, and $ \partial _t\hat{\varphi}, \partial _x \partial _\mu\hat{\varphi}$ are continuous.
\end{enumerate}
\end{Definition}

\begin{Definition}
We denote by $\boldsymbol C_b^{1,2}(\mathscr{\hat H})$ the set of $\hat\varphi\in\boldsymbol C^{1,2}(\mathscr{\hat H})$ such that
$\hat\varphi$, $\partial_t\hat\varphi$, $\partial_\mu\hat\varphi$, $\partial_x\partial_\mu\hat\varphi$ are bounded.
\end{Definition}

We end this section with the It\^o formula.
The proof is postponed in Appendix~\ref{ItoProof}.

\begin{Theorem}\label{T:ItoD}
Fix $t\in[0,T]$ and let $\xi\in\S_2(\F)$.
Let also  $F\colon[0,T]\times\Omega\rightarrow H$, $G\colon[0,T]\times\Omega\rightarrow\mathcal{L}_2(K;H)$ be square-integrable and $\F$-progressively measurable processes, so, in particular,
\[
\int_0^T \E[|F_s|^2_H]\,ds < \infty, \qquad\qquad \int_0^T \E\big[\textup{Tr}(G_sG_s^*)\big]\,ds < \infty.
\]
Consider the process $X=(X_s)_{s\in[0,T]}$ given by\footnote{In what follows, we will always implicitly
refer to   It\^o processes only by  continuous versions.}
\begin{equation}\label{2020-12-04:02}
X_s = \xi_{s\wedge t} + \int_t^{s\vee t} F_r\,dr + \int_t^{s\vee t} G_r\,dB_r, \qquad \forall\,s\in[0,T].
\end{equation}
If $\hat\varphi\colon\mathscr{\hat H}\rightarrow\mathbb{R}$ is in $\boldsymbol C_b^{1,2}(\mathscr{\hat H})$, then the following \textbf{It\^o formula} holds:
\begin{align}\label{ItoD}
\hat\varphi(s,\hat\P_{X_{\cdot\wedge s}}) &= \hat\varphi(t,\hat\P_{\xi_{\cdot\wedge t}}) + \int_t^s \partial_t \hat\varphi(r,\hat\P_{X_{\cdot\wedge r}})\,dr + \int_t^s
\mathbb{E}
 \left[   \langle F_r,\partial_\mu \hat\varphi(r,\hat\P_{X_{\cdot\wedge r}})(X_{\cdot\wedge r})\rangle_H
 \right]dr \notag \\
&\quad + \frac{1}{2}\int_t^s\mathbb{E}
 \left[
   \Tr \left(
        G_r G_r^* \partial _x\partial_\mu
       \hat\varphi(r,\hat\P_{X_{\cdot\wedge r}})(X_{\cdot\wedge r})
   \right)
 \right]dr,
\end{align}
for every $s\in[t,T]$ (for the definition of $\hat\P_{X_{\cdot\wedge r}}$ see \ref{Ntn_P}).
\end{Theorem}
\begin{Remark}\label{rm:ITOunbounded1}
  Proceeding along the same lines as in the proof of Theorem \ref{T:ItoD}, it is possible to prove It\^o's formula for a larger class of functions than $\boldsymbol C_b^{1,2}(\mathscr{\hat H})$,
  for example
  weakening the boundedness assumption of $\partial_\mu\hat\varphi$ by assuming linear growth with respect to $\hat\xi$.
\end{Remark}


\begin{Remark}\label{R:Sym}
Following \cite{WuZhang18PPDE}, we notice that in the last term of It\^o's formula \eqref{ItoD} we can replace $\partial_x\partial_\mu\hat\varphi(r,\hat\P_{X_{\cdot\wedge r}})(X_{\cdot\wedge r})$ by its {symmetrization} $\partial_x^{\textup{sym}}\partial_\mu\hat\varphi(r,\hat\P_{X_{\cdot\wedge r}})(X_{\cdot\wedge r})$, where $\partial_x^{\textup{sym}}\partial_\mu\hat\varphi\colon\mathscr{\hat H}\times D([0,T];H)\rightarrow\Lc(H)$ is defined as
\begin{equation}\label{Sym}
\partial_x^{\textup{sym}}\partial_\mu\hat\varphi(t,\hat\mu)(\hat x) \coloneqq  \frac{1}{2}\Big(\partial_x\partial_\mu\hat\varphi(t,\hat\mu)(\hat x) + \big(\partial_x\partial_\mu\hat\varphi(t,\hat\mu)(\hat x)\big)^*\Big),
\end{equation}
for every $(t,\hat\mu,\hat x)\in\mathscr{\hat H}\times D([0,T];H)$.\\
In the finite-dimensional case, it is proved in \cite[Remark 5.98]{cardelbook} that $\partial_x\partial_\mu\hat\varphi$ is already symmetric. Notice however that such a proof is quite involved and it is still an open problem to show that it remains valid in the present infinite-dimensional framework.
\end{Remark}

\subsection{Pathwise derivatives for a map $\varphi\colon\mathscr H\rightarrow\R$ and It\^o's formula}

In the present section we use several times the notation \ref{Ntn_P}, namely we use the superscript $\hat\cdot$ to denote the natural extension to $D([0,T];H)$ of a probability measure on $C([0,T];H)$.

\begin{Definition}\label{D:Consistent}
Given $\varphi\colon\mathscr H\rightarrow\R$ and a non-anticipative map $\hat\varphi\colon\mathscr{\hat H}\rightarrow\R$, we say that $\hat\varphi$ is \textbf{consistent} with $\varphi$ if (for the definition of $\hat\P_\xi$ see \ref{Ntn_P})
\begin{equation}\label{Consistency}
\varphi(t,\P_\xi) = \hat\varphi(t,\hat\P_\xi),
\end{equation}
for every $(t,\xi)\in\mathfrak H$, with $\xi\in\S_2(\F)$ (namely, $\xi\in L^2(\Omega;C([0,T];H))$ and it is $\F$-progressively measurable).
\end{Definition}
\begin{Remark}
Notice that we can replace $\S_2(\F)$ with $\S_2(\Gc)$ in Definition \ref{D:Consistent}, as a matter of fact the sets $\{\P_\xi\colon\xi\in\S_2(\F)\}$ and $\{\P_\xi\colon\xi\in\S_2(\Gc)\}$ are equal and coincide with $\Pc_2(C([0,T];H))$, as it follows from property  \textup{\ref{A_G}}-ii). This also shows that equality \eqref{Consistency} characterizes $\varphi$ in terms of $\hat\varphi$ for every pair $(t,\mu)\in\mathscr H$.
\end{Remark}

Next result is crucial in order to define pathwise derivatives for a map $\varphi\colon\mathscr H\rightarrow\R$, as it states a consistency property for the pathwise derivatives themselves.

\begin{Lemma}\label{L:Consistency}
Let $\hat\varphi_1,\hat\varphi_2\in\boldsymbol C_b^{1,2}(\mathscr{\hat H})$. If (for the definition of $\hat\P_\xi$ see \ref{Ntn_P})
\[
\hat\varphi_1(t,\hat\P_\xi) = \hat\varphi_2(t,\hat\P_\xi), \qquad \forall\,(t,\xi)\in\mathfrak H,\;\text{with }\xi\in\S_2(\F),
\]
then
\begin{align}
\partial_t\hat\varphi_1(t,\hat\P_\xi) &= \partial_t\hat\varphi_2(t,\hat\P_\xi), \label{Cons-time-deriv} \\
\partial_\mu\hat\varphi_1(t,\hat\P_\xi)(\xi) &= \partial_\mu\hat\varphi_2(t,\hat\P_\xi)(\xi), \qquad\qquad\! \P\text{\textup{-a.s.}} \label{Cons-measure-deriv} \\
\partial_x^{{\textup{sym}}}\partial_\mu\hat\varphi_1(t,\hat\P_\xi)(\xi) &= \partial_x^{{\textup{sym}}}\partial_\mu\hat\varphi_2(t,\hat\P_\xi)(\xi), \qquad \P\text{\textup{-a.s.}} \label{Cons-2nd-measure-deriv}
\end{align}
for every $(t,\xi)\in\mathfrak H$, with $\xi\in\S_2(\F)${, where $\partial_x^{\textup{sym}}\partial_\mu\hat\varphi$ is defined by \eqref{Sym}.}
\end{Lemma}
\begin{proof}[\textbf{Proof.}]
See Appendix \ref{App:Consistency}.
\end{proof}
\begin{Remark}\label{R:2nd-measure-deriv}
{We do not address here the consistency property of $\partial_x\partial_\mu\hat\varphi$ (for hints on its proof we refer to \cite{WuZhang18PPDE}, see the paragraph just after Theorem 2.9), as It\^o's formula (and hence the Hamilton-Jacobi-Bellman equation) only depends on $\partial_x^{\textup{sym}}\partial_\mu\hat\varphi$, see Remark \ref{R:Sym}.}
\end{Remark}
\begin{Remark}\label{R:NA-Consistency}
By the non-anticipativity property of the pathwise derivatives (see Remarks \ref{R:NA-time_deriv}, \ref{R:NA-measure-deriv}, \ref{R:NA-2nd-measure-deriv}), it follows that equalities \eqref{Cons-time-deriv}-\eqref{Cons-measure-deriv}-\eqref{Cons-2nd-measure-deriv} hold if and only if
\begin{align*}
\partial_t\hat\varphi_1(t,\hat\P_{\xi_{\cdot\wedge t}}) &= \partial_t\hat\varphi_2(t,\hat\P_{\xi_{\cdot\wedge t}}), \\
\partial_\mu\hat\varphi_1(t,\hat\P_{\xi_{\cdot\wedge t}})(\xi_{\cdot\wedge t}) &= \partial_\mu\hat\varphi_2(t,\hat\P_{\xi_{\cdot\wedge t}})(\xi_{\cdot\wedge t}), \qquad\qquad\! \P\text{\textup{-a.s.}} \\
\partial_x^{{\textup{sym}}}\partial_\mu\hat\varphi_1(t,\hat\P_{\xi_{\cdot\wedge t}})(\xi_{\cdot\wedge t}) &= \partial_x^{{\textup{sym}}}\partial_\mu\hat\varphi_2(t,\hat\P_{\xi_{\cdot\wedge t}})(\xi_{\cdot\wedge t}), \qquad \P\text{\textup{-a.s.}}
\end{align*}
for every $(t,\xi)\in\mathfrak H$, with $\xi\in\S_2(\F)$.
\end{Remark}

Using Lemma \ref{L:Consistency}, we can now define the class $\boldsymbol C_b^{1,2}(\mathscr H)$.

\begin{Definition}\label{D:C12}
{We denote by $\boldsymbol C^{1,2}(\mathscr H)$ (respectively $\boldsymbol C_b^{1,2}(\mathscr H)$) the set of maps $\varphi\colon\mathscr H\rightarrow\R$ for which there exists $\hat\varphi\colon\mathscr{\hat H}\rightarrow\R$ such that $\hat\varphi$ is consistent with $\varphi$ and $\hat\varphi\in\boldsymbol C^{1,2}(\mathscr{\hat H})$
(respectively $\boldsymbol C_b^{1,2}(\mathscr{\hat H})$).} Then, we define (for the definition of $\hat\P_\xi$ see \ref{Ntn_P})
\begin{align*}
\partial_t\varphi(t,\P_\xi) &\coloneqq  \partial_t\hat\varphi(t,\hat\P_\xi), \\
\partial_\mu\varphi(t,\P_\xi)(\cdot) &\coloneqq  \partial_\mu\hat\varphi(t,\hat\P_\xi)(\cdot), \\
\partial_x^{{\textup{sym}}}\partial_\mu\varphi(t,\P_\xi)(\cdot) &\coloneqq  \partial_x^{{\textup{sym}}}\partial_\mu\hat\varphi(t,\hat\P_\xi)(\cdot),
\end{align*}
for every $(t,\xi)\in\mathfrak H$, with $\xi\in\S_2(\F)$.
\end{Definition}

We can finally state the It\^o formula.

\begin{Theorem}\label{T:Ito}
  Fix $t\in[0,T]$ and let $\xi\in\S_2(\F)$.
  Let also $F\colon[0,T]\times\Omega\rightarrow H$, $G\colon[0,T]\times\Omega\rightarrow\mathcal{L}_2(K;H)$ be square integrable
   and $\F$-progressively measurable process, so, in particular,
\[
\int_0^T \E[|F_s|^2_H]\,ds < \infty, \qquad\qquad \int_0^T \E\big[\textup{Tr}(G_sG_s^*)\big]\,ds < \infty.
\]
Consider the process $X=(X_s)_{s\in[0,T]}$ given by
\begin{equation*}
X_s = \xi_{s\wedge t} + \int_t^{s\vee t} F_r\,dr + \int_t^{s\vee t} G_r\,dB_r, \qquad \forall\,s\in[0,T].
\end{equation*}
If $\varphi\colon\mathscr H\rightarrow\mathbb{R}$ is in $\boldsymbol C_b^{1,2}(\mathscr H)$, then the following \textbf{It\^o formula} holds:
\begin{align}\label{Ito}
\varphi(s,\P_{X_{\cdot\wedge s}}) &= \varphi(t,\P_{\xi_{\cdot\wedge t}}) + \int_t^s \partial_t \varphi(r,\P_{X_{\cdot\wedge r}})\,dr + \int_t^s
\mathbb{E}
 \left[   \langle F_r,\partial_\mu \varphi(r,\P_{X_{\cdot\wedge r}})(X_{\cdot\wedge r})\rangle_H
 \right]dr \notag \\
&\quad + \frac{1}{2}\int_t^s\mathbb{E}
 \left[
   \Tr \left(
        G_r G_r^* \partial _x\partial_\mu
       \varphi(r,\P_{X_{\cdot\wedge r}})(X_{\cdot\wedge r})
   \right)
 \right]dr,
\end{align}
for every $s\in[t,T]$.
\end{Theorem}
\begin{proof}[\textbf{Proof.}]
By Definition \ref{D:C12} there exists $\hat\varphi\colon\mathscr{\hat H}\rightarrow\R$ such that $\hat\varphi$ is consistent with $\varphi$ and $\hat\varphi\in\boldsymbol C_b^{1,2}(\mathscr{\hat H})$. As a consequence, by Theorem \ref{T:ItoD} we have the It\^o formula
\begin{align*}
\hat\varphi(s,\hat\P_{X_{\cdot\wedge s}}) &= \hat\varphi(t,\hat\P_{\xi_{\cdot\wedge t}}) + \int_t^s \partial_t \hat\varphi(r,\hat\P_{X_{\cdot\wedge r}})\,dr + \int_t^s
\mathbb{E}
 \left[   \langle F_r,\partial_\mu \hat\varphi(r,\hat\P_{X_{\cdot\wedge r}})(X_{\cdot\wedge r})\rangle_H
 \right]dr \notag \\
&\quad + \frac{1}{2}\int_t^s\mathbb{E}
 \left[
   \Tr \left(
        G_r G_r^* \partial _x\partial_\mu
       \hat\varphi(r,\hat\P_{X_{\cdot\wedge r}})(X_{\cdot\wedge r})
   \right)
 \right]dr,
\end{align*}
for every $s\in[t,T]$. Using the fact that $\hat\varphi$ is consistent with $\varphi$ and recalling the definition of pathwise derivatives of $\varphi$ (see Definition \ref{D:C12}), we obtain the claimed It\^o formula for $\varphi$.
\end{proof}

{In order to apply It\^o's formula to our case we need the following variant of
Theorem \ref{T:Ito} (for a similar result, see Proposition 1.165 in \cite{FabbriGozziSwiech}).}

\begin{Theorem}\label{T:ItoForTest}
{Fix $t\in[0,T]$ and let $\xi\in\S_2(\F)$ (namely, $\xi\in L^2(\Omega;C([0,T];H))$ and it is $\F$-progressively measurable). Let $A,b,\sigma$ be as in Assumption \ref{A_A,b,sigma}.
Let $X=X^{t,\xi,\alpha}$ be the unique mild solution of \eqref{cont-MKV-differential}.
Let $\varphi\colon\mathscr H\rightarrow\mathbb{R}$ belong to $\boldsymbol C_b^{1,2}(\mathscr H)$. Assume also that, for all
$(t,\mu,x)\in \mathscr H\times C([0,T];H)$,
$\partial_\mu \varphi(t,\mu)(x)\in D(A^*)$
and that the map
\[
\mathscr H \times C([0,T];H) \longrightarrow H
\qquad\qquad (t,\mu,x) \longmapsto A^*\partial\varphi_\mu(t,\mu)(x)
\]
is continuous and bounded
\footnote{ Indeed, in view of \eqref{EstimateX} here
    we could ask only
linear growth of $ \partial _\mu \varphi$ in $x$}.
Then the following variant of \textbf{It\^o formula} holds:
\begin{align}\label{Itonew}
\varphi(s,\P_{X_{\cdot\wedge s}}) &= \varphi(t,\P_{\xi_{\cdot\wedge t}}) +
\int_t^s \partial_t \varphi(r,\P_{X_{\cdot\wedge r}})\,dr +
\int_t^s\mathbb{E}
 \left[   \langle X_r,
A^* \partial_\mu \varphi(r,\P_{X_{\cdot\wedge r}})
 (X_{\cdot\wedge r})\rangle_H
 \right]dr \notag \\
&+\int_t^s\mathbb{E}\left[\langle
b_r\left(X,\P_{X_{\cdot\wedge r}}, \alpha_r,\P_{\alpha_r}\right),
 \partial_\mu \varphi(r,\P_{X_{\cdot\wedge r}})
 (X_{\cdot\wedge r})\rangle_H
 \right]dr \\
&+ \frac{1}{2}\int_t^s\mathbb{E}
 \left[\Tr \left(
\sigma_r\left(X,\P_{X_{\cdot\wedge r}},\alpha_r,\P_{\alpha_r}\right)
\sigma^*_r\left(X,\P_{X_{\cdot\wedge r}}, \alpha_r,\P_{\alpha_r}\right) \partial _x\partial_\mu
       \varphi(r,\P_{X_{\cdot\wedge r}})(X_{\cdot\wedge r})
   \right)
 \right]dr, \notag
\end{align}
for every $s\in[t,T]$.}
\end{Theorem}
\begin{proof}[\textbf{Proof.}]
  {The proof can be done proceeding along the same lines as in the proof of Proposition 1.165 in \cite{FabbriGozziSwiech}. We provide a sketch of proof.
\\
First of all if $A$ is a bounded operator from Theorem \ref{T:Ito} we immediately get \eqref{Itonew}.
Now take $A$ possibly unbounded and consider its Yosida approximations $A_n$, for $n\in \N$. Call, as in \eqref{cont-MKV-differential-An}, $X^n$ the solution of the state equation when $A$ is replaced by $A_n$. Then, from  \eqref{Itonew} we get
\begin{align*}
\varphi(s,\P_{X^n_{\cdot\wedge s}}) &= \varphi(t,\P_{\xi_{\cdot\wedge t}}) +
\int_t^s \partial_t \varphi(r,\P_{X^n_{\cdot\wedge r}})\,dr +
\int_t^s\mathbb{E}
 \left[   \langle X^n_r,
A_n^* \partial_\mu \varphi(r,\P_{X^n_{\cdot\wedge r}})
 (X^n_{\cdot\wedge r})\rangle_H
 \right]dr \notag \\
&+\int_t^s\mathbb{E}\left[\langle
b_r\left(X^n,\P_{X^n_{\cdot\wedge r}}, \alpha_r,\P_{\alpha_r}\right),
 \partial_\mu \varphi(r,\P_{X^n_{\cdot\wedge r}})
 (X^n_{\cdot\wedge r})\rangle_H
 \right]dr \\
&+ \frac{1}{2}\int_t^s\mathbb{E}
 \left[\Tr \left(
\sigma_r\left(X^n,\P_{X^n_{\cdot\wedge r}},\alpha_r,\P_{\alpha_r}\right)
\sigma^*_r\left(X^n,\P_{X^n_{\cdot\wedge r}}, \alpha_r,\P_{\alpha_r}\right) \partial _x\partial_\mu
       \varphi(r,\P_{X^n_{\cdot\wedge r}})(X^n_{\cdot\wedge r})
   \right)
 \right]dr, \notag
\end{align*}
for every $s\in[t,T]$.
Now the convergence of all terms above follows applying, in a straightforward way,
the result of Proposition \ref{Prop-diff-MKVPD-An}.}
\end{proof}


\section{Hamilton-Jacobi-Bellman equation}
\label{S:HJB}

\subsection{Viscosity properties of the value function}

\label{sub:HJBviscosity}

For every $t\in[0,T]$, we introduce the set
\[
\Mc_t \coloneqq  \big\{\mathfrak a\colon\Omega\rightarrow\Ur\colon\mathfrak a\text{ is $\Fc_t$-measurable}\big\}.
\]
Note that, since the filtration $\Fc_t$ is right-continuous, we have
$\Mc_{t}=\cap_{\eps>0} \Mc_{t+\eps}$.
We now consider the following Hamilton-Jacobi-Bellman (HJB) equation:
\begin{align}\label{HJB}
0 \ &= \ \partial_t w(t,\mu) + \E\langle \xi_t,A^*\partial_\mu w(t,\mu)(\xi)\rangle_H \notag \\
&\quad \ + \sup_{\mathfrak a\in\Mc_t}\bigg\{\E\big[f_t\big(\xi,\mu,\mathfrak a,\P_{\mathfrak a}\big) + \big\langle b_t\big(\xi,\mu,{\mathfrak a},\P_{\mathfrak a}\big),\partial_\mu w(t,\mu)(\xi)\big\rangle_H\big] \\
&\quad \ + \dfrac{1}{2}\,\E\Big[\textup{Tr}\Big(\sigma_t\big(\xi,\mu,{\mathfrak a},\P_{\mathfrak a}\big)\sigma_t^*\big(\xi,\mu,{\mathfrak a},\P_{\mathfrak a}\big)\partial_x\partial_\mu w(t,\mu)(\xi)\Big)\Big]\bigg\},  \notag
\end{align}
for $(t,\mu)\in\mathscr H$, $t<T$, $\xi\in\S_2(\Gc)$ such that $\P_\xi=\mu$,
with terminal condition
\begin{equation}\label{eq:TCHJB}
w(T,\mu) \ = \ \E[g(\xi,\mu)], \qquad
\hbox{for $\mu\in\Pc_2(C([0,T];H))$, $\xi\in\S_2(\Gc)$ such that $\P_\xi=\mu$}.
\end{equation}

\begin{Definition}
We say that a function $w\colon\mathscr H\to \R$ belongs to the space $\boldsymbol C_{b,A^*}^{1,2}(\mathscr H)$
if it satisfies the following regularity assumptions:
\begin{itemize}
  \item[(i)] $w\colon\mathscr H\rightarrow\mathbb{R}$ belongs to $\boldsymbol C_b^{1,2}(\mathscr H)$;
  \item[(ii)] for all
$(t,\mu,\xi)\in \mathscr H\times \S_2(\F)$,
$\partial_\mu \varphi(t,\mu)(\xi)\in L^2(\Omega;D(A^*))$
and the map
\[
\mathscr H \times \S_2(\F) \longrightarrow L^2(\Omega;H), \qquad\qquad (t,\mu,\xi) \longmapsto A^*\varphi(t,\mu)(\xi)
\]
is continuous and bounded.
\end{itemize}
\end{Definition}

\begin{Definition}
We say that a function $w\colon\mathscr H\to \R$
is a classical solution to the HJB equation \eqref{HJB}
with terminal condition \eqref{eq:TCHJB},
if it belongs to the space
$\boldsymbol C_{b,A^*}^{1,2}(\mathscr H)$
and satisfies \eqref{HJB}-\eqref{eq:TCHJB}.
\end{Definition}

Using Theorem \ref{T:ItoForTest} and Corollary \ref{C:DPP} we are able to prove the following result.

\begin{Theorem}\label{T:HJBreg}
Let Assumptions \ref{A_A,b,sigma} and \ref{A_f,g_cont} hold. Assume also that $b,\sigma,f$ are uniformly continuous in $t$, uniformly with respect to the other variables.
Assume that the value function $v$ (see \eqref{v}) belongs to the space
$\boldsymbol C_{b,A^*}^{1,2}(\mathscr H)$.
Then $v$ is a classical solution of \eqref{HJB}-\eqref{eq:TCHJB}.
\end{Theorem}

\begin{proof}[\textbf{Proof.}]
From Corollary \ref{C:DPP} we know that, for every $(t,\mu) \in \mathscr{H}$, for every $\xi \in\S_2(\F)$ such $\P_\xi=\mu$, for every $\alpha \in \Uc$, and for every $h>0$ sufficiently small,
\begin{equation}\label{eq:DPPforHJB}
0 = \sup_{\alpha\in\Uc}\bigg\{\E\bigg[
\frac{1}{h}\int_t^{t+h}
f_r\big(X,\P_{X_{\cdot\wedge r}},
\alpha_r,\P_{\alpha_r}\big)\,dr\bigg] + \frac{1}{h}\left[v\big(t+h,\P_{X}\big)-v(t,\mu)
\right]\bigg\},
\end{equation}
where, for simplicity, we wrote simply $X$ in place of $X^{t,\xi,\alpha}$.
Now we use \eqref{Itonew} getting
\begin{align}\label{Itonewbis}
&\frac{1}{h}\left[v\big(t+h,\P_{X^{t,\xi,\alpha}}\big)-v(t,\mu) \right]
=
\frac{1}{h}\int_t^{t+h} \partial_t v(r,\P_{X_{\cdot\wedge r}})\,dr
+
\frac{1}{h}\int_t^{t+h}\mathbb{E}
 \left[   \langle X_r,
A^* \partial_\mu v(r,\P_{X_{\cdot\wedge r}})
 (X_{\cdot\wedge r})\rangle_H
 \right]dr
 \notag \\
&\qquad +\frac{1}{h}\int_t^{t+h}\mathbb{E}\left[\langle
b_r\left(X,\P_{X_{\cdot\wedge r}}, \alpha_r,\P_{\alpha_r}\right),
 \partial_\mu v(r,\P_{X_{\cdot\wedge r}})
 (X_{\cdot\wedge r})\rangle_H
 \right]dr
  \\
 \notag
&\qquad +\frac{1}{2h}\int_t^{t+h}\mathbb{E}
 \left[\Tr \left(
\sigma_r\left(X,\P_{X_{\cdot\wedge r}},\alpha_r,\P_{\alpha_r}\right)
\sigma^*_r\left(X,\P_{X_{\cdot\wedge r}}, \alpha_r,\P_{\alpha_r}\right) \partial _x\partial_\mu
       v(r,\P_{X_{\cdot\wedge r}})(X_{\cdot\wedge r})
   \right)
 \right]dr,
\end{align}
We now show, as in the typical proof of this result, the two inequalities.
First take any ${\mathfrak a}\in \Mc_t$ and consider the control
$\alpha\in \Uc$ defined as
$$
\alpha_s=0, \quad s\in [0,t), \qquad
\alpha_s={\mathfrak a}, \quad s\in [t,T]
$$
Then, from \eqref{eq:DPPforHJB} and \eqref{Itonewbis} we get
\begin{align}\label{Itonewter}
&0 \ge
\frac{1}{h}\int_t^{t+h}
\E f_r\big(X,\P_{X_{\cdot\wedge r}},
{\mathfrak a},\P_{{\mathfrak a}}\big)\,dr
+
\frac{1}{h}\int_t^{t+h} \partial_t v(r,\P_{X_{\cdot\wedge r}})\,dr
 \notag \\
&\qquad +
\frac{1}{h}\int_t^{t+h}\mathbb{E}
 \left[   \langle X_r,
A^* \partial_\mu v(r,\P_{X_{\cdot\wedge r}})
 (X_{\cdot\wedge r})\rangle_H
 \right]dr
 \notag \\
&\qquad +\frac{1}{h}\int_t^{t+h}\mathbb{E}\left[\langle
b_r\left(X,\P_{X_{\cdot\wedge r}}, {\mathfrak a},\P_{{\mathfrak a}}\right),
 \partial_\mu v(r,\P_{X_{\cdot\wedge r}})
 (X_{\cdot\wedge r})\rangle_H
 \right]dr
 \\ \notag
&\qquad +\frac{1}{2h}\int_t^{t+h}\mathbb{E}
 \left[\Tr \left(
\sigma_r\left(X,\P_{X_{\cdot\wedge r}},{\mathfrak a},\P_{{\mathfrak a}}\right)
\sigma^*_r\left(X,\P_{X_{\cdot\wedge r}}, {\mathfrak a},\P_{{\mathfrak a}}\right) \partial _x\partial_\mu
       v(r,\P_{X_{\cdot\wedge r}})(X_{\cdot\wedge r})
   \right)
 \right]dr,
\end{align}
Using the regularity of $v$ and the continuity properties of $b,\sigma,f$ we get, passing to the limit for $h \to 0^+$
\begin{align}\label{Itonew4}
0 \ge \
&\partial_t v(t,\mu) + \E\langle \xi_t,A^*\partial_\mu v(t,\mu)(\xi)\rangle_H
\E\big[f_t\big(\xi,\mu,\mathfrak a,\P_{\mathfrak a}\big) + \big\langle b_t\big(\xi,\mu,{\mathfrak a},\P_{\mathfrak a}\big),\partial_\mu v(t,\mu)(\xi)\big\rangle_H\big]
\notag
\\
&+\,\dfrac{1}{2}\E\Big[\textup{Tr}\Big(\sigma_t\big(\xi,\mu,{\mathfrak a},\P_{\mathfrak a}\big)\sigma_t^*\big(\xi,\mu,{\mathfrak a},\P_{\mathfrak a}\big)\partial_x\partial_\mu v(t,\mu)(\xi)\Big)\Big]
\end{align}
Then the inequality $\ge$ follows by the arbitrariness of ${\mathfrak a}$.
We now prove the opposite inequality. For every $\eps>0$ we take  $\alpha^\eps\in \Uc$ such that, denoting by $X^\eps$ the corresponding state trajectory,
\begin{equation}\label{eq:DPPforHJBeps}
-\eps \le \E\bigg[
\frac{1}{\eps}\int_t^{t+\eps}
f_r\big(X^\eps,\P_{X^\eps_{\cdot\wedge r}},
\alpha^\eps_r,\P_{\alpha^\eps_r}\big)\,dr\bigg] + \frac{1}{\eps}\left[v\big(t+\eps,\P_{X^\eps}\big)-v(t,\mu)
\right]\bigg\},
\end{equation}
Now we apply Ito's formula \eqref{Itonewbis} above, for $h=\eps$, getting
\begin{align}\label{eq:DPPforHJBepsnew}
&-\eps \le \E\bigg[
\frac{1}{\eps}\int_t^{t+\eps}
f_r\big(X^\eps,\P_{X^\eps_{\cdot\wedge r}},
\alpha^\eps_r,\P_{\alpha^\eps_r}\big)\,dr\bigg] +
\frac{1}{\eps}\int_t^{t+\eps}
\partial_t v(r,\P_{X^\eps_{\cdot\wedge r}})\,dr
\notag\\
&+
\frac{1}{\eps}\int_t^{t+\eps}\mathbb{E}
 \left[   \langle X^\eps_r,
A^* \partial_\mu v(r,\P_{X^\eps_{\cdot\wedge r}})
 (X^\eps_{\cdot\wedge r})\rangle_H
 \right]dr
 \notag \\
&\qquad +\frac{1}{\eps}\int_t^{t+\eps}\mathbb{E}\left[
\langle
b_r\left(X^\eps,\P_{X^\eps_{\cdot\wedge r}}, \alpha_r,\P_{\alpha_r}\right),
 \partial_\mu v(r,\P_{X^\eps_{\cdot\wedge r}})
 (X^\eps_{\cdot\wedge r})\rangle_H
 \right]dr
  \\
 \notag
&\qquad +\frac{1}{2\eps}\int_t^{t+\eps}\mathbb{E}
 \left[\Tr \left(
\sigma_r\left(X^\eps,\P_{X^\eps_{\cdot\wedge r}},\alpha_r,\P_{\alpha_r}\right)
\sigma^*_r\left(X^\eps,\P_{X^\eps_{\cdot\wedge r}}, \alpha_r,\P_{\alpha_r}\right) \partial _x\partial_\mu
       v(r,\P_{X^\eps_{\cdot\wedge r}})(X^\eps_{\cdot\wedge r})
   \right)
 \right]dr,
\end{align}
By Remark \ref{rm:convSEinitialdatum} we obtain that, as $\eps \to 0$,
$X^\eps_{\cdot\wedge r} \to \xi_{\cdot\wedge t}$  and $\P_{X^\eps_{\cdot\wedge r}} \to \P_{\xi_{\cdot\wedge t}}$, hence
the second and third integrals of the above right-hand side converge to
$$
\partial_t v(t,\mu)
+\mathbb{E}\left[\langle \xi_t,
A^* \partial_\mu v(t,\mu)(\xi)\rangle_H\right]
$$
The remaining integrals of the right-hand side of
\eqref{eq:DPPforHJBepsnew} can be rewritten as
\begin{align}\label{eq:DPPforHJBepsnewbis}
\frac{1}{\eps}\int_t^{t+\eps}
\bigg(\mathbb{E}\left[
f_r\big(X^\eps,\P_{X^\eps_{\cdot\wedge r}},
\alpha^\eps_r,\P_{\alpha^\eps_r}\big)
+
\langle
b_r\left(X^\eps,\P_{X^\eps_{\cdot\wedge r}}, \alpha_r,\P_{\alpha_r}\right),
 \partial_\mu v(r,\P_{X^\eps_{\cdot\wedge r}})
 (X^\eps_{\cdot\wedge r})\rangle_H
 \right]
\\
\notag
\mathbb{E}
 \left[\Tr \left(
\sigma_r\left(X^\eps,\P_{X^\eps_{\cdot\wedge r}},\alpha_r,\P_{\alpha_r}\right)
\sigma^*_r\left(X^\eps,\P_{X^\eps_{\cdot\wedge r}}, \alpha_r,\P_{\alpha_r}\right) \partial _x\partial_\mu
       v(r,\P_{X^\eps_{\cdot\wedge r}})(X^\eps_{\cdot\wedge r})
   \right)
 \right]\bigg)dr
\end{align}
Recall now that, by our assumptions, $b$ and $\sigma$ are uniformly continuous in $(t,x,\mu)$ uniformly with respect to the other variables and that $f$ is locally uniformly continuous in $(x,\mu)$ uniformly with respect to the other variables.
Hence, using again Remark \ref{rm:convSEinitialdatum} we obtain that
\eqref{eq:DPPforHJBepsnewbis} can be rewritten as
\begin{align}\label{eq:DPPforHJBepsnewter}
\frac{1}{\eps}\int_t^{t+\eps}
\bigg(\mathbb{E}\left[
f_t\big(\xi,\mu,
\alpha^\eps_r,\P_{\alpha^\eps_r}\big)
+
\langle
b_t\left(\xi,\mu, \alpha_r,\P_{\alpha_r}\right),
 \partial_\mu v(t,\mu)
 (\xi)\rangle_H
 \right]
\\
\notag
\mathbb{E}
 \left[\Tr \left(
\sigma_t\left(\xi,\mu,\alpha_r,\P_{\alpha_r}\right)
\sigma^*_r\left(\xi,\mu, \alpha_r,\P_{\alpha_r}\right) \partial_x\partial_\mu
       v(t,\mu)(\xi)
   \right)
 \right]\bigg)dr + \rho (\eps)
\end{align}
where $\rho (\eps)\to 0$ as $\eps \to 0$.
It follows
\begin{align}\label{HJBeps}
&0\le \eps + \rho (\eps) +\partial_t v(t,\mu) + \E\langle \xi_t,A^*\partial_\mu v(t,\mu)(\xi)\rangle_H \notag \\
&+\,\sup_{\mathfrak a\in\Mc_{t+\eps}}\bigg\{\E\big[f_t\big(\xi,\mu,\mathfrak a,\P_{\mathfrak a}\big) + \big\langle b_t\big(\xi,\mu,{\mathfrak a},\P_{\mathfrak a}\big),\partial_\mu v(t,\mu)(\xi)\big\rangle_H\big] \\
&+\,\dfrac{1}{2}\,\E\Big[\textup{Tr}\Big(\sigma_t\big(\xi,\mu,{\mathfrak a},\P_{\mathfrak a}\big)\sigma_t^*\big(\xi,\mu,{\mathfrak a},\P_{\mathfrak a}\big)\partial_x\partial_\mu v(t,\mu)(\xi)\Big)\Big]\bigg\}  \notag
\end{align}
The conclusion then follows invoking the second part of Lemma
\ref{L:Formula_F}.
\end{proof}

Now, we provide the definition of viscosity solution that we shall  use.

\begin{Definition}
We say that a function $w\colon\mathscr H\to \R$
is a viscosity subsolution (respectively supersolution) to the HJB equation \eqref{HJB}
with terminal condition \eqref{eq:TCHJB},
if:
\begin{itemize}
\item $w(T,\mu)\leq(\text{respectively $\geq$})\,\E[g(\xi,\mu)]$, for $\mu\in\Pc_2(C([0,T];H))$, $\xi\in\S_2(\Gc)$ such that $\P_\xi=\mu$;
\item for $(t,\mu) \in \mathscr H$ and for every test function $\varphi\in \boldsymbol C_{b,A^*}^{1,2}(\mathscr H)$
such that $w-\varphi$ has a maximum at $(t,\mu)$ (with value $0$), one has that \eqref{HJB}-\eqref{eq:TCHJB} is satisfied with the inequality $\le$ $($respectively $\ge$$)$ in place of the equality and with $\varphi$ in place of $w$.
\end{itemize}
Moreover, $w$ is called a viscosity solution of
\eqref{HJB}-\eqref{eq:TCHJB} if it is both a viscosity subsolution and a viscosity supersolution.
\end{Definition}

\begin{Theorem}\label{T:HJBvisc}
Let Assumptions \ref{A_A,b,sigma} and \ref{A_f,g_cont} hold. Assume also that $b,\sigma,f$ are uniformly continuous in $t$, uniformly with respect to the other variables.
Then, the value function $v$ is a viscosity solution of \eqref{HJB}-\eqref{eq:TCHJB}.
\end{Theorem}

\begin{proof}[\textbf{Proof.}]
The proof follows exactly the same lines as in the proof of Theorem \ref{T:HJBreg}, simply replacing $v$ with $\varphi$.
\end{proof}

\color{black}

\subsection{Alternative forms of the HJB equation}
\label{sub:alterHJBviscosity}

We derive alternative forms of the Hamilton-Jacobi-Bellman equation \eqref{HJB} relying on technical results reported in Appendix \ref{App:HJB}. We first need to introduce the following sets:
\begin{itemize}
\item $\Mc_\Gc$ is the set
\[
\Mc_\Gc \coloneqq  \big\{\mathfrak a\colon\Omega\rightarrow\Ur\colon\mathfrak a\text{ is $\Gc$-measurable}\big\};
\]
\item $\check\Mc$ is the set of Borel-measurable maps $\mathrm{\check a}\colon C([0,T];H)\times[0,1]\rightarrow\Ur$;
\item $\Mc$ is the set of Borel-measurable maps $\mathrm a\colon C([0,T];H)\rightarrow\Ur$.
\end{itemize}
The technical results reported in Appendix \ref{App:HJB} provides the following key proposition.

\begin{Proposition}\label{P:AlternativeHJB}
{Suppose that \textup{\ref{A_A,b,sigma}} and \textup{\ref{A_f,g}} hold}. Let $(t,\mu)\in\mathscr H$, $w\in\boldsymbol C_b^{1,2}(\mathscr H)$, and define $F\colon C([0,T];H)\times\Ur\times\Pc(\Ur)\rightarrow\R$ by
\begin{align*}
F(x,u,\nu) &\coloneqq  f_t(x,\mu,u,\nu) + \big\langle b_t(x,\mu,u,\nu),\partial_\mu w(t,\mu)(x)\big\rangle_H \\
&\quad \ + \dfrac{1}{2}\,\textup{Tr}\Big(\sigma_t(x,\mu,u,\nu)\sigma_t^*(x,\mu,u,\nu)\partial_x\partial_\mu w(t,\mu)(x)\Big),
\end{align*}
for every $(x,u,\nu)\in C([0,T];H)\times\Ur\times\Pc(\Ur)$. Let also $\xi\in\S_2(\Gc)$ with $\P_\xi=\mu$.
\begin{enumerate}[\upshape 1)]
\item Suppose that $\xi$ is such that there exists a $\Gc$-measurable random variable $U_\xi$ having uniform distribution on $[0,1]$ and being independent of $\xi$ (by Lemma \ref{L:xi_U_indep} we know that for each $\mu$ there exist at least one $\xi$, with $\P_\xi=\mu$, satisfying this property). Then, it holds that
\begin{equation}\label{Formula_F_HJB}
\sup_{{\mathfrak a}\in\Mc_t} \E\big[F\big(\xi,{\mathfrak a},\P_{\mathfrak a}\big)\big] = \sup_{{\mathfrak a}\in\Mc_\Gc} \E\big[F\big(\xi,{\mathfrak a},\P_{\mathfrak a}\big)\big] = \sup_{\mathrm{\check a}\in\check\Mc} \E\big[F\big(\xi,\mathrm{\check a}(\xi,U_\xi),\P_{\mathrm{\check a}(\xi,U_\xi)}\big)\big].
\end{equation}
\item Suppose that $F$ does not depend on its last argument, namely $F=F(x,u)$. Then, it holds that
\begin{equation}\label{Formula_F_bis_HJB}
\sup_{{\mathfrak a}\in\Mc_t} \E\big[F\big(\xi,{\mathfrak a}\big)\big] = \sup_{\mathrm a\in\Mc} \E\big[F(\xi,\mathrm a(\xi))\big] = \E\Big[\esssup_{u\in\Ur}F(\xi,u)\Big].
\end{equation}
\end{enumerate}
\end{Proposition}
\begin{proof}[\textbf{Proof.}]
Equalities \eqref{Formula_F_HJB} are a direct consequence of Lemma \ref{L:Formula_F}, while \eqref{Formula_F_bis_HJB} follows directly from equalities \eqref{Formula_F_bis} and \eqref{Formula_F_ter} of Lemma \ref{L:Formula_F_bis}.
\end{proof}

\begin{Remark}
{Notice that the requirement that $F$ does not depend on $\nu$ in item 2) of Proposition \ref{P:AlternativeHJB} is necessary for the validity of the first equality in \eqref{Formula_F_bis_HJB} (we do not consider the second equality in \eqref{Formula_F_bis_HJB} in this case, as it is not clear how to write the last quantity in \eqref{Formula_F_bis_HJB} when $F$ depends also on $\nu$). As a matter of fact, consider the following example.\\
\textbf{Example.} Take $\Ur=[0,1]$, endowed with its Borel $\sigma$-algebra, and let $F$ be given by
\[
F(x,u,\nu) = -\,\Wc_2(\nu,\lambda), \qquad \forall\,(x,u,\nu)\in C([0,T];H)\times[0,1]\times\Pc([0,1]),
\]
with $\lambda$ being the Lebesgue measure on the unit interval. Let also $\xi$ be constant and identically equal to some fixed path $\bar x\in C([0,T];H)$. Moreover, denote by $U_\Gc$ a $\Gc$-measurable random variable having distribution $\lambda$, whose existence follows from Lemma \ref{L:U_G}. Then, for every $t\in[0,T]$,
\[
\sup_{{\mathfrak a}\in\Mc_t} \big(- \Wc_2\big(\P_{\mathfrak a},\lambda\big)\big) = 0
\]
and the supremum is attained at ${\mathfrak a}^*$, where ${\mathfrak a}^*\coloneqq U_\Gc$. On the other hand, if $\mathrm a\in\Mc$ then $\mathrm a(\xi)$ is equal to the constant $\mathrm a(\bar x)$, so, in particular, $\P_{\mathrm a(\xi)}=\delta_{\mathrm a(\bar x)}$. This implies that
\begin{align*}
\sup_{\mathrm a\in\Mc} \big(- \Wc_2(\P_{\mathrm a(\xi)},\lambda)\big) = \sup_{c\in[0,1]} \big(- \Wc_2(\delta_c,\lambda)\big) &= - \inf_{c\in[0,1]}\bigg(\int_0^1 |c - r|^2\,dr\bigg)^{\frac{1}{2}} \\
&= - \inf_{c\in[0,1]} \sqrt{c^2 - c + \frac{1}{3}} = - \frac{1}{12}.
\end{align*}}
\end{Remark}

\begin{Remark}\label{R:esssup=sup}
Suppose that $F=F(x,u)$ and define $F^*\colon C([0,T];H)\rightarrow\R\cup\{+\infty\}$ as
\[
F^*(x) \coloneqq  \sup_{u\in\Ur} F(x,u), \qquad \forall\,x\in C([0,T];H).
\]
If $F^*$ is measurable, then $\esssup_{u\in\Ur}F(\xi,u)=F^*(\xi)$, $\P$-a.s., and the \emph{essential supremum} appearing in \eqref{Formula_F_bis_HJB} can be replaced with the \emph{supremum}, so that we obtain
\[
\sup_{{\mathfrak a}\in\Mc_t} \E\big[F(\xi,{\mathfrak a})\big] = \E\Big[\sup_{u\in\Ur}F(\xi,u)\Big].
\]
{Notice that, under assumptions \textup{\ref{A_A,b,sigma}} and \textup{\ref{A_f,g}},} it follows from Proposition 7.47 in \cite{BertsekasShreve} that $F^*$ is lower semi-analytic (see Definition 7.21 in \cite{BertsekasShreve} for the definition of \emph{lower semi-analytic}). However, we cannot in general say that $F^*$ is measurable (see for instance the discussion at the end of Section B.5 in \cite{BertsekasShreve}). Sufficient conditions ensuring the measurability of $F^*$ are given for instance in Proposition 7.32 of \cite{BertsekasShreve} and read as follows:
\begin{enumerate}[\upshape (a)]
\item If $F\colon C([0,T];H)\times\Ur\rightarrow\R$ is lower semi-continuous and $\Ur$ is compact, then $F^*$ is lower semi-continuous.
\item If $F\colon C([0,T];H)\times\Ur\rightarrow\R$ is upper semi-continuous, then $F^*$ is upper semi-continuous.
\end{enumerate}
\end{Remark}

\vspace{3mm}

\noindent It follows from Proposition \ref{P:AlternativeHJB} that, under \textup{\ref{A_A,b,sigma}} and \textup{\ref{A_f,g}}, the Hamilton-Jacobi-Bellman equation \eqref{HJB} can also be written in the following two alternative forms:
\begin{equation}\label{HJB_M_Gc}
\hspace{-4.3mm}\begin{cases}
\partial_t w(t,\mu) + \E\langle \xi_t,A^*\partial_\mu w(t,\mu)(\xi)\rangle_H \\
+\,\sup_{\mathfrak a\in\Mc_\Gc}\bigg\{\E\big[f_t\big(\xi,\mu,\mathfrak a,\P_{\mathfrak a}\big) + \big\langle b_t\big(\xi,\mu,{\mathfrak a},\P_{\mathfrak a}\big),\partial_\mu w(t,\mu)(\xi)\big\rangle_H\big] \\
+\,\dfrac{1}{2}\E\Big[\textup{Tr}\Big(\sigma_t\big(\xi,\mu,{\mathfrak a},\P_{\mathfrak a}\big)\sigma_t^*\big(\xi,\mu,{\mathfrak a},\P_{\mathfrak a}\big)\partial_x\partial_\mu w(t,\mu)(\xi)\Big)\Big]\bigg\} = 0, \hspace{3.35mm} (t,\mu)\in\mathscr H,\,t<T, \\
w(T,\mu) = \E[g(\xi,\mu)], \hspace{6.875cm} \mu\in\mathscr P_2(C([0,T];H)),
\end{cases}
\end{equation}
or, alternatively,
\begin{equation}\label{HJB_checkM}
\hspace{-4mm}\begin{cases}
\vspace{1mm}\partial_t w(t,\mu) + \E\langle \xi_t,A^*\partial_\mu w(t,\mu)(\xi)\rangle_H + \sup_{\mathrm{\check a}\in\check\Mc}\bigg\{\E\big[f_t\big(\xi,\mu,\mathrm{\check a}(\xi,U_\xi),\P_{\mathrm{\check a}(\xi,U_\xi)}\big) \\
\vspace{1mm}+ \big\langle b_t\big(\xi,\mu,\mathrm{\check a}(\xi,U_\xi),\P_{\mathrm{\check a}(\xi,U_\xi)}\big),\partial_\mu w(t,\mu)(\xi)\big\rangle_H\big] \\
\vspace{1mm}+ \dfrac{1}{2}\E\Big[\textup{Tr}\Big(\sigma_t\sigma_t^*\big(\xi,\mu,\mathrm{\check a}(\xi,U_\xi),\P_{\mathrm{\check a}(\xi,U_\xi)}\big)\partial_x\partial_\mu w(t,\mu)(\xi)\Big)\Big]\bigg\} = 0, \hspace{3mm} (t,\mu)\in\mathscr H,\,t<T, \\
w(T,\mu) = \E[g(\xi,\mu)], \hspace{6.825cm} \mu\in\mathscr P_2(C([0,T];H)),
\end{cases}
\end{equation}
where, in both \eqref{HJB_M_Gc} and \eqref{HJB_checkM}, $\xi\in\S_2(\Gc)$, with $\P_\xi=\mu$, is such that there exists a $\Gc$-measurable random variable $U_\xi$ having uniform distribution on $[0,1]$ and being independent of $\xi$ (we recall that, by Lemma \ref{L:xi_U_indep}, for each $\mu$ there exists at least one $\xi$, with $\P_\xi=\mu$, satisfying this latter property).

{Now, suppose that \textup{\ref{A_A,b,sigma}}, \textup{\ref{A_f,g}}} hold and also that the coefficients $b$, $\sigma$, $f$ do not depend on their last argument (namely, $b=b_t(x,\mu,u)$, $\sigma=\sigma_t(x,\mu,u)$, $f=f_t(x,\mu,u)$). Then, by Proposition \ref{P:AlternativeHJB} we deduce that the Hamilton-Jacobi-Bellman equation \eqref{HJB} can also be written in the following two alternative forms:
\begin{equation}\label{HJB_M}
\hspace{-4.6mm}\begin{cases}
\partial_t w(t,\mu) + \E\langle \xi_t,A^*\partial_\mu w(t,\mu)(\xi)\rangle_H \\
+ \sup_{\mathrm a\in\Mc}\bigg\{\E\big[f_t\big(\xi,\mu,\mathrm a(\xi)\big) + \big\langle b_t\big(\xi,\mu,\mathrm a(\xi)\big),\partial_\mu w(t,\mu)(\xi)\big\rangle_H\big] \\
+ \dfrac{1}{2}\E\Big[\textup{Tr}\Big(\sigma_t\big(\xi,\mu,\mathrm a(\xi)\big)\sigma_t^*\big(\xi,\mu,\mathrm a(\xi)\big)\partial_x\partial_\mu w(t,\mu)(\xi)\Big)\Big]\bigg\} = 0, \hspace{6mm} (t,\mu)\in\mathscr H,\,t<T, \\
w(T,\mu) = \E[g(\xi,\mu)], \hspace{6.9cm} \mu\in\mathscr P_2(C([0,T];H)),
\end{cases}
\end{equation}
for every $\xi\in\S_2(\Gc)$ with $\P_\xi=\mu$, or, alternatively:
\begin{equation}\label{HJB_esssup}
\hspace{-4.2mm}\begin{cases}
\partial_t w(t,\mu) + \E\langle \xi_t,A^*\partial_\mu w(t,\mu)(\xi)\rangle_H \\
+ \E\bigg[\esssup_{u\in\Ur}\bigg[f_t(\xi,\mu,u) + \big\langle b_t(\xi,\mu,u),\partial_\mu w(t,\mu)(\xi)\big\rangle_H \\
+ \dfrac{1}{2}\textup{Tr}\Big(\sigma_t(\xi,\mu,u)\sigma_t^*(\xi,\mu,u)\partial_x\partial_\mu w(t,\mu)(\xi)\Big)\bigg]\bigg] = 0, \hspace{2.05cm} (t,\mu)\in\mathscr H,\,t<T, \\
w(T,\mu) = \E[g(\xi,\mu)], \hspace{6.85cm} \mu\in\mathscr P_2(C([0,T];H)),
\end{cases}
\end{equation}
for every $\xi\in\S_2(\Gc)$ with $\P_\xi=\mu$.

\begin{Remark}
  \label{rm:HJBinv}
  In the case of the optimal investment problem outlined in Example \ref{ex:SE}-(ii) and in Remark \ref{rm:quadraticreward}, equation \eqref{eq:costinvest},
the HJB equation can be written as follows (here we explicitly write that the production $Q$ depends on the state process $\xi_t$ and on its distribution $\mu$):
\begin{equation}\label{HJBInvestmentfinal}
\hspace{-4.2mm}\begin{cases}
\partial_t w(t,\mu) + \E\langle \xi_t,A^*\partial_\mu w(t,\mu)(\xi)\rangle_H \\
+ \E\bigg[e^{-rt}R(Q(t,\xi_t,\mu)
+\delta(\xi,\mu)\<\xi_t,\partial_\mu w(t,\mu)(\xi)\>_H
\\
+\esssup_{u\in\Ur}\Big[
\<Cu,\partial_\mu w(t,\mu)(\xi)\>_H -e^{-rt}(\<a_1,u\>_H + \<M u,u\>_H)\Big] \\
+ \dfrac{1}{2}\sigma^2\textup{Tr}
\Big(
\partial_x\partial_\mu w(t,\mu)(\xi)\Big)\bigg]
 = 0,  \hspace{5.9cm} (t,\mu)\in\mathscr H,\,t<T, \\
w(T,\mu) = \E[g(\xi,\mu)], \hspace{7.55cm} \mu\in\mathscr P_2(C([0,T];H)),
\end{cases}
\end{equation}
for every $\xi\in\S_2(\Gc)$ with $\P_\xi=\mu$.
Note that in the above case the $\esssup$ appearing in the Hamiltonian can be explicitly computed.
\end{Remark}

\appendix

\section{State equation: proofs}
\label{App:StateEquation}

We collect in the following lemma some continuity results for contractions in Banach spaces that we will use to obtain the needed continuity properties of the mild solution to the state equation.
\begin{Lemma}\label{2020-10-11:02}
  Let
  $\mathcal{R}$ be a non-empty set,
  $\mathcal{T}$ be a topological space,
  $(M,d)$ be a metric space,
  $Y$ be a Banach space,
  $\gamma\in [0,1)$.
  Let $w$ be a modulus of continuity.
  Let $h\colon \mathcal{R}\times \mathcal{\mathcal{T}}\times M\times Y\rightarrow Y$ be such that:
  \begin{equation*}
     |h(r,x,m,y)-h(r,x,m',y')|\leq w(d(m,m'))+\gamma|y-y'|, \qquad \forall\, r\in \mathcal{R},\, x\in \mathcal{T},\, m,m'\in M,\, y,y'\in Y.
   \end{equation*}
   Let $\mathcal{E}\subset 2^{\mathcal{R}}$ be a set of subsets of $\mathcal{R}$.
   Assume that, if $\{x_\iota\}_{\iota\in \mathcal{I}}\subset \mathcal{T}$ is a net converging to $x\in \mathcal{T}$, then
   \begin{equation*}
     \lim_{\iota}   \sup_{r\in E}|
     h (r,x_\iota,m,y)-\varphi (r,x,m,y)| \ = \ 0, \qquad \forall\,  m\in M,\,  y\in Y, \,E\in \mathcal{E}.
   \end{equation*}
   Denote by $\varphi\colon \mathcal{R}\times \mathcal{T}\times M\rightarrow Y$ the fixed-point map associated with $h$, i.e. $\varphi$ is the unique map satisfying
   \begin{equation*}
     h(r,x,m,\varphi(r,x,m)) \ = \ \varphi(r,x,m), \qquad \forall\,r\in \mathcal{R},\, x\in \mathcal{T},\, m\in M.
   \end{equation*}
   Then, given a net $\{x_\iota\}_{\iota\in \mathcal{I}}\subset \mathcal{T}$  converging to $x$, it holds that
   \begin{equation}
      \label{2020-10-11:00}
  \lim_{\iota}   \sup_{r\in E}|\varphi (r,x_\iota,m)-\varphi (r,x,m')| \ \leq \ \frac{1}{1-\gamma}w(d(m,m')), \qquad \forall\,  m,m'\in M, \,E\in\mathcal{E}.
   \end{equation}
 \end{Lemma}
\begin{proof}[\textbf{Proof.}]
  Write
  \begin{equation*}
    \begin{split}
      |\varphi (r,x_\iota,m)-\varphi (r,x,m')|\ &= \
                |h(r,x_\iota,m,\varphi (r,x_\iota,m))-h(r,x,m',\varphi (r,x,m'))|\\
                &\leq \
                |h(r,x_\iota,m,\varphi(r,x_\iota,m))
                -
                h(r,x_\iota,m,\varphi(r,x,m'))|\\
                &\quad \ +
                |h(r,x_\iota,m,\varphi(r,x,m'))
                -
                h(r,x,m,\varphi(r,x,m'))|\\
                &\quad \ +
                |h(r,x,m,\varphi(r,x,m'))
                -
                h(r,x,m',\varphi(r,x,m'))|\\
                &\leq \
                \gamma|\varphi(r,x_\iota,m)-\varphi(r,x,m')|\\
                &\quad \ +|h(r,x_\iota,m,\varphi(r,x,m'))
                -
                h(r,x,m,\varphi(r,x,m'))|+w(d(m,m')).
 \end{split}
\end{equation*}
Then
\begin{equation*}
  \begin{split}
      \sup_{r\in E}
  |\varphi(r,x_\iota,m)-\varphi(r,x,m')|
  \ &\leq \ \frac{1}{1-\gamma}\sup_{r\in E} |h(r,x_\iota,m,\varphi(r,x,m'))
                -
                h(r,x,m,\varphi(r,x,m'))|\\
                &\quad \ +\frac{1}{1-\gamma}w(d(m,m'))
  \end{split}
\end{equation*}
and the claim follows taking the limit with respect to $\iota$.
\end{proof}

\begin{proof}[\textbf{Proof of Proposition~\ref{Prop-diff-MKVPD}.}]
  The proof is based, as usual, on a contraction argument. We introduce the space $L^2_\mathbb{F}(H)$ (resp.\ $L^2_\mathbb{F}(\mathcal{L}_2(K;H))$) of all square-integrable $\mathbb{F}$-progressively measurable processes on $[0,T]$ taking values in $H$ (resp.\ $\mathcal{L}_2(K;H)$), normed respectively by $\|X\|_{L^2_\mathbb{F}(H)}:= \big( \mathbb{E} \big[  \int _0^T
         |X_s|_H^2ds\big]  \big)^{1/2}$ and $\|\Phi\|_{L^2_\mathbb{F}(H)}:= \big( \mathbb{E} \big[  \int _0^T
    |\Phi_s|_{\mathcal{L}_2(K;H)}^2ds \big]  \big) ^{1/2}$. For the sake of brevity, in what follows we will denote $S_t:=e^{At}$, $b^{\alpha}_t(X,\mathbb{P}_X):=
    b_t(X,\mathbb{P}_X,\alpha_t,\mathbb{P}_{\alpha_t})$, $\sigma^{\alpha}_t(X,\mathbb{P}_X):=
        \sigma_t(X,\mathbb{P}_X,\alpha_t,\mathbb{P}_{\alpha_t})$. For $t\in [0,T]$ and $\alpha\in \mathcal{U}$, let us define
\begin{equation}\label{2020-09-19:00}
  \begin{split}
    \mathrm{id}^S_t\colon& \mathbf{S}_2(\mathbb{F})\rightarrow \mathbf{S}_2(\mathbb{F}), \hspace{2.5cm} \xi \mapsto \xi_\cdot\mathbf{1}_{[0,t]}(\cdot)+S_{\cdot-t}\xi_t\mathbf{1}_{(t,T]}(\cdot)\\[0.5em]
    F_{b^{\alpha}}\colon &\mathbf{S}_2(\mathbb{F})\rightarrow L^2_\mathbb{F}(H), \hspace{2.25cm} X \mapsto b^\alpha_\cdot(X,\mathbb{P}_X)\\[0.5em]
    F_{\sigma^{\alpha}}\colon &\mathbf{S}_2(\mathbb{F})\rightarrow L^2_\mathbb{F}(\mathcal{L}_2(K;H)), \hspace{1cm} X \mapsto \sigma^\alpha_\cdot(X,\mathbb{P}_X)\\[0.5em]
    S\conv{t}\#\colon& L^2_\mathbb{F}(H)\rightarrow \mathbf{S}_2(\mathbb{F}), \hspace{2.25cm} X \mapsto \mathbf{1}_{[t,T]}(\cdot)\int_t^\cdot    S_{\cdot-s}X_sds\\[0.5em]
        S\sconv{t}\#\colon& L^2_\mathbb{F}(\mathcal{L}_2(K;H))\rightarrow \mathbf{S}_2(\mathbb{F}), \hspace{1.05cm} \Phi \mapsto \mathbf{1}_{[t,T]}(\cdot)\int_t^\cdot    S_{\cdot-s}\Phi_sdB_s.
    \end{split}
  \end{equation}
  We briefly explain why the functions above are well-defined.
  Regarding $\mathrm{id}^S_t$, it holds that
  \begin{equation}\label{2020-06-20:00}
    \mathrm{id}^S_t(\xi) \ = \ \xi_{\cdot\wedge t}+\mathbf{1}_{(t,T]}(\cdot)(S_{\cdot-t}-I)\xi_t,
  \end{equation}
  which clearly shows that $\mathrm{id}^S_t(\xi)\in \mathbf{S}_2(\mathbb{F})$.
  Regarding $F_{b^\alpha}$,
  due to the measurability assumptions on $b$,
  we have that $F_{b^\alpha}(X)$ is progressively measurable. Moreover, recalling
  Assumption~\ref{A_A,b,sigma}-\eqref{2020-06-20:01}, we have
  \begin{equation}\label{2020-06-20:07}
      \|F_{b^\alpha}(X)\|_{L^2_\mathbb{F}(H)}^2
      \leq
      3L^2\mathbb{E} \left[ \int_0^T
        \left(
          1+2\|X\|_t^2        \right) dt
      \right] <\infty.
    \end{equation}
In the same way, we obtain the measurability of $F_{\sigma^\alpha}(X)$ and
  \begin{equation}\label{2020-06-20:10}
      \|F_{\sigma^\alpha}(X)\|_{L^2_\mathbb{F}(H)}^2
      \leq
      3L^2\mathbb{E} \left[ \int_0^T
        \left(
          1+2\|X\|_t^2
        \right) dt
      \right] <\infty.
    \end{equation}
    Regarding $S\conv{t}X$, for  $X\in L^2_\mathbb{F}(H)$, it is not difficult to see that it is continuous, $\mathbb{F}$-adapted, and that
    \begin{equation}\label{2020-06-20:08}
      \|S\conv{t}X\|_{\mathbf{S}_2}^2\leq e^{2\eta (T-t)}(T-t)\|X\|^2_{L^2_\mathbb{F}(H)}.
    \end{equation}
    Finally, regarding $S\sconv{t}\Phi$, for $\Phi\in L^2_\mathbb{F}(\mathcal{L}_2(K;H))$, by
    \cite[Theorem~1.111]{FabbriGozziSwiech}
     we know that the  $\mathbb{F}$-adapted process $ \left\{ \mathbf{1}_{[t,T]}(t')\int_t^{t'}S_{t'-s}\Phi_sds \right\} _{t'\in [0,T]}$ admits a continuous version, that we name $S\sconv{t}\Phi$, and that
\begin{equation}\label{2020-06-20:09}
  \|S\sconv{t}\Phi\|_{\mathbf{S}_2}\leq C_{\eta,T}\|\Phi\|_{L^2_\mathbb{F}(\mathcal{L}_2(K;H))},
\end{equation}
where $C_{\eta,T}$ is a constant depending only on $\eta,T$.

\noindent We now define the map
\begin{equation*}
  \psi\colon
\mathcal{U}\times
  [0,T]\times \mathbf{S}_2(\mathbb{F})\times \mathbf{S}_2(\mathbb{F})\rightarrow \mathbf{S}_2(\mathbb{F})
\end{equation*}
by
\begin{equation*}
  \psi(\alpha,t,\xi,X) \ = \ \mathrm{id}^S_t(\xi)+S\conv{t} F_{b^\alpha}(X)+S\sconv{t}F_{\sigma^\alpha}(X), \quad \forall\,(\alpha,t,\xi,X)\in \mathcal{U}\times[0,T]\times \mathbf{S}_2(\mathbb{F})\times \mathbf{S}_2(\mathbb{F}).
\end{equation*}

\smallskip
\noindent \emph{\textbf{Claim I.} For fixed $\xi,X\in \mathbf{S}_2(\mathbb{F})$, $\psi(\alpha,t,\xi,X)$ is continuous in $t$, uniformly in $\alpha\in \mathcal{U}$.}\\
The continuity of $t \mapsto \mathrm{id}^S_t(\xi)$ follows from Lebesgue's dominated convergence theorem and the fact that, for every fixed $\omega \in \Omega$, $\mathrm{id}^S_{t'}(\xi)(\omega)$ converges uniformly to $\mathrm{id}^S_{t}(\xi)(\omega)$ as $t'\rightarrow t$.\\
Moreover, for $t',t\in [0,T]$, $t'<t$, we have
\begin{equation*}
    S\conv{t'}F_{b^\alpha}(X)-
    S\conv{t}F_{b^\alpha}(X) \ = \
  \mathbf{1}_{[t',t]}(\cdot)\int_{t'}^\cdot S_{\cdot-s}F_{b^\alpha}(X)_sds
+\mathbf{1}_{[t,T]}(\cdot)\int_{t'}^t S_{\cdot-s}F_{b^\alpha}(X)_sds.
\end{equation*}
Then
\begin{equation*}
  \|  S\conv{t'}F_{b^\alpha}(X)-
  S\conv{t}F_{b^\alpha}(X)\|_T \ \leq \
2  e^{\eta T}\int_{t'}^t|F_{b^\alpha}(X)_s|_Hds.
\end{equation*}
This implies, recalling
Assumption~\ref{A_A,b,sigma}\eqref{2020-06-20:01},
\begin{equation*}
  \sup_{\alpha\in \mathcal{U}}\lim_{|t'-t|\rightarrow 0}   \|  S\conv{t'}F_{b^\alpha}(X)-
  S\conv{t}F_{b^\alpha}(X)\|_{\mathbf{S}_2} \ = \ 0.
\end{equation*}
Finally, take again $t',t\in[0,T]$, $t'<t$. Then
\begin{equation*}
    S\sconv{t'}F_{\sigma^\alpha}(X)-
    S\sconv{t}F_{\sigma^\alpha}(X) \ = \
  \mathbf{1}_{[t',t]}(\cdot)\int_{t'}^\cdot S_{\cdot-s}F_{\sigma^\alpha}(X)_sdB_s
+\mathbf{1}_{[t,T]}(\cdot)\int_{t'}^t S_{\cdot-s}F_{\sigma^\alpha}(X)_sdB_s.
\end{equation*}
By
    \cite[Theorem~1.111]{FabbriGozziSwiech},
we have
\begin{equation*}
  \|  S\sconv{t'}F_{\sigma^\alpha}(X)-
  S\sconv{t}F_{\sigma^\alpha}(X)\|_{\mathbf{S}_2}^2 \ \leq \
  C'_{\eta,T}
    \|\mathbf{1}_{[t',t]}F_{\sigma^\alpha}(X)\|_{L^2_\mathbb{F}(H)}^2
\end{equation*}
and then,
after recalling
Assumption~\ref{A_A,b,sigma}\eqref{2020-06-20:01},
\begin{equation*}
  \sup_{\alpha\in \mathcal{U}}  \lim_{|t'-t|\rightarrow 0}   \|  S\sconv{t'}F_{\sigma^\alpha}(X)-
  S\sconv{t}F_{\sigma^\alpha}(X)\|_{\mathbf{S}_2} \ = \ 0.
\end{equation*}

\smallskip
\noindent\emph{\textbf{Claim II.} $\psi(\alpha,t,\xi,X)$ is Lipschitz continuous in $\xi$, uniformly in $t,X,\alpha$.}\\
For $\alpha\in \mathcal{U}$, $t\in [0,T]$, $\xi,\xi'\in \mathbf{S}_2(\mathbb{F})$,
$X\in \mathbf{S}_2(\mathbb{F})$,
we have
\begin{equation*}
 \| \psi(\alpha,t,\xi,X)-
 \psi(\alpha,t,\xi',X)\|_{\mathbf{S}_2}^2 \ = \
 \|\mathrm{id}^S_t(\xi-\xi')\|_{\mathbf{S}_2}^2 \ \leq \ 2(1+e^{\eta T})
 \|\xi-\xi'\|_{\mathbf{S}_2}^2.
\end{equation*}
Now, for $a,b\in[0,T]$, $a<b$, let us consider the restriction of $\psi$ to the time interval $[a,b]$ and stopped at time $b$, namely
\begin{equation*}
  \psi_{a,b}\colon \mathcal{U}\times[a,b]\times \mathbf{S}_2(\mathbb{F})\times \mathbf{S}_2(\mathbb{F})\rightarrow \mathbf{S}_2(\mathbb{F}),\qquad (\alpha,t,\xi,X) \mapsto \psi(\alpha,t,\xi,X)_{b\wedge \cdot}
\end{equation*}

\smallskip
\noindent\emph{\textbf{Claim III.} There exists $\epsilon>0$ such that, if $b-a\leq \epsilon$, then $\psi_{a,b}(\alpha,t,\xi,X)$ is a contraction in $X$, uniformly in $\alpha,t,\xi,a,b$, namely: for some $\gamma\in [0,1)$,
\begin{equation*}
  \|\psi_{a,b}(\alpha,t,\xi,X)_{b\wedge \cdot}-\psi_{a,b}(\alpha,t,\xi,X')_{b\wedge \cdot}\|_{\mathbf{S}_2} \ \leq \ \gamma
  \|X-X'\|_{\mathbf{S}_2},
\end{equation*}
for all $\alpha\in \mathcal{U},t\in [a,b],(\xi,X,X')\in \mathbf{S}_2(\mathbb{F})^3$
and  for all $ a,b\in [0,T],a<b, b-a\leq \epsilon$.}\\
Let $a,b\in [0,T]$, $a<b$, $t\in [a,b]$,
$\alpha\in \mathcal{U}$,
$\xi\in \mathbf{S}_2(\mathbb{F})$, $X,X'\in \mathbf{S}_2(\mathbb{F})$. Notice that
\begin{equation}\label{2020-07-22:00}
  \psi(\alpha,t,\xi,X)-\psi(\alpha,t,\xi,X') \ = \
  S\conv{t} \left( F_{b^\alpha}(X)-F_{b^\alpha}(X') \right)+
   S\sconv{t} \left( F_{\sigma^\alpha}(X)-F_{\sigma^\alpha}(X') \right)
 \end{equation}
 and
\begin{equation}\label{2020-07-22:01}
  \| \psi_{a,b}(\alpha,t,\xi,X)-\psi_{a,b}(\alpha,t,\xi,X')\|_T \ = \
   \| \psi(\alpha,t,\xi,X)-\psi(\alpha,t,\xi,X')\|_b.
 \end{equation}
Moreover, recalling  Assumption~\ref{A_A,b,sigma}-\eqref{2020-06-20:01}, we  have
\begin{equation*}
  \begin{split}
      \|S\conv{t}
  \left( F_{b^\alpha}(X)-F_{b^\alpha}(X') \right)\|_b^2 \ &\leq \
  e^{2\eta T} \left( \int _a^b\left| F_{b^\alpha}(X)_s-F_{b^\alpha}(X')_s\right|ds \right) ^2\\
  &\leq \
  2L^2 e^{2\eta T}(b-a)^2 \left( \|X-X'\|^2_T+\Wc^2_2(\mathbb{P}_X,\mathbb{P}_{X'}) \right) .
\end{split}
\end{equation*}
Then
\begin{equation}\label{2020-06-20:03}
     \mathbb{E}  \left[ \|S\conv{t}
      \left( F_{b^\alpha}(X)-F_{b^\alpha}(X') \right)\|_b ^2\right]
 \ \leq \
  4L^2 e^{2\eta T}(b-a)^2  \|X-X'\|_{\mathbf{S}_2}.
\end{equation}
By \cite[Theorem~1.111]{FabbriGozziSwiech} there exists a constant $C_{\eta,T}$, depending only on $\eta,T$, such that
\begin{equation} \label{2020-06-20:04}
  \begin{split}
    \mathbb{E} \left[
      \|S\sconv{t}
    \left( F_{\sigma^\alpha}(X)-F_{\sigma^\alpha}(X') \right) \|_b^2 \right]  \ &\leq \ C_{\eta,T}^2
\mathbb{E} \left[  \int _a^b \|
  F_{\sigma^\alpha}(X)_s-F_{\sigma^\alpha}(X')_s
  \|_{\mathcal{L}_2(K;H)}^2ds \right] \\
    &\leq \
    4L^2 C^2_{\eta,T}(b-a)
      \|X-X'\|_{\mathbf{S}_2}^2,
  \end{split}
\end{equation}
where for the last inequality we have used
Assumption~\ref{A_A,b,sigma}-\eqref{2020-06-20:01} again. By
\eqref{2020-07-22:00},
\eqref{2020-07-22:01},
\eqref{2020-06-20:03},  and \eqref{2020-06-20:04}, we have, if $b-a<1$,
\begin{equation*}
  \| \psi_{a,b}(\alpha,t,\xi,X)-\psi_{a,b}(\alpha,t,\xi,X')\|_{\mathbf{S}_2} \ \leq \ C_{L,\eta,T}(b-a)^{1/2}
   \|X-X'\|_{\mathbf{S}_2},
 \end{equation*}
 where $C_{L,\eta,T}$ is a constant depending only on $L,\eta,T$.
 We can then choose $\epsilon\in (0,1)$ such that $C_{L,\eta,T}\epsilon^{1/2}<1/2$.
 Then  the map $\psi_{a,b}(\alpha,t,\xi,\cdot)$ is a $1/2$-contraction, uniformly
in $\alpha,t,\xi$ and in $a,b\in [0,T]$, whenever $a<b$, $b-a\leq\epsilon$.

\smallskip
\noindent\emph{\textbf{Claim IV.} For $\epsilon>0$ as in Claim III, and whenever $a,b\in [0,T]$, $a<b$, $b-a\leq\epsilon$,
 there exists a unique mild solution\footnote{We say that $X^{t,\xi,\alpha}$ is a mild solution \emph{on the interval $[a,b]$} if equation
\eqref{cont-MKV-differential} is solved in the mild sense for $s\in[a,b]$.} $X^{t,\xi,\alpha}$ to equation
\eqref{cont-MKV-differential} on the interval $[a,b]$,
for any $\alpha\in \mathcal{U},t\in[a,b],\xi\in \mathbf{S}_2(\mathbb{F})$.}\\
Let $\epsilon$ be as in \emph{Claim III}.
Then, for any $\alpha,t,\xi$, by the Banach contraction principle, there exists a unique fixed  point $X^{t,\xi,\alpha}$ to $\psi_{a,b}(\alpha,t,\xi,\cdot)$. Clearly $X^{t,\xi,\alpha}$ is a mild solution to
\eqref{cont-MKV-differential} on the interval $[a,b]$.

\smallskip
\noindent \emph{\textbf{Claim V.} For $\epsilon$ as in Claim III, and uniformly for $a,b\in [0,T]$, $a<b$, $b-a\leq \epsilon$,
the map
\begin{equation}\label{2020-06-20:05}
\varphi_{a,b}\colon  \mathcal{U}\times [a,b]\times \mathbf{S}_2(\mathbb{F})\rightarrow \mathbf{S}_2(\mathbb{F}),\qquad (\alpha,t,\xi) \mapsto X^{t,\xi,\alpha}
\end{equation}
is continuous in $(t,\xi)$, uniformly in $\alpha$, and Lipschitz-continuous
in $\xi$, uniformly in $\alpha,t$.}\\
Let $\epsilon$ be as in \emph{Claim III}.
We apply Lemma~\ref{2020-10-11:02} with $\mathcal{R}=\mathcal{U}$, $\mathcal{T}=[a,b]$, $M=\mathbf{S}_2(\mathbb{F})$, $Y=\mathbf{S}_2(\mathbb{F})$, $\mathcal{E}=2^\mathcal{R}$. Then, by \emph{Claims I,II,III}, we get
\begin{equation*}
  \lim_{\substack{t\rightarrow t'\\t\in[a,b]}}\sup_{\alpha\in \mathcal{U}}
  \|\varphi_{a,b}(\alpha,t,\xi)-\varphi_{a,b}(\alpha,t',\xi')\|_{\mathbf{S}_2} \ \leq \ 4(1+e^{\eta T})^{1/2}\|\xi-\xi'\|_{\mathbf{S}_2}, \qquad \forall\,t'\in [a,b],\, \xi,\xi'\in \mathbf{S}_2.
\end{equation*}

\smallskip

\noindent \emph{\textbf{Claim VI.} For any $\alpha\in \mathcal{U}, t\in[0,T], \xi\in \mathbf{S}_2(\mathbb{F})$, there exists a unique mild solution $X^{t,\xi,\alpha}$ to \eqref{cont-MKV-differential}, and
the map
\begin{equation}
  \label{2020-06-20:06}
\varphi\colon \mathcal{U}\times  [0,T]\times \mathbf{S}_2(\mathbb{F})\rightarrow \mathbf{S}_2(\mathbb{F}),\ (\alpha,t,\xi) \mapsto X^{t,\xi,\alpha}
\end{equation}
is continuous in $t,\xi$, uniformly in $\alpha$, and Lipschitz continuous in $\xi$, uniformly in $t,\alpha$.}\\
Pick $\epsilon>0$ as in \emph{Claim III}.
Choose $a_0=0<a_1<\ldots<a_n=T$ with $a_{i+1}-a_i\leq \epsilon$.
Define
\begin{equation*}
  \hat\varphi_{a_i}\colon\mathcal{U}\times[a_i,T]\times \mathbf{S}_2(\mathbb{F})\rightarrow  \mathcal{U}\times[a_i,T]\times\mathbf{S}_2(\mathbb{F})
\end{equation*}
by
$\hat\varphi_{a_i}(\alpha,t,\xi)=(\alpha,a_{i+1},\varphi_{a_i,a_{i+1}}(\alpha,t,\xi))$ if $t\in[a_i,a_{i+1}]$
and
$\hat\varphi_{a_i}(\alpha,t,\xi)=(\alpha,t,\xi_{a_{i+1}\wedge\cdot})$ if $t>a_{i+1}$.
If we now define
$  \varphi(\alpha,t,\xi)$ to be the second component of
\begin{equation*}
  \hat\varphi_{a_n}
   \left(
     \hat\varphi_{a_{n-1}}
     \left(\ldots
       \hat\varphi_{a_1}\left(\varphi_{a_0}(\alpha,t,\xi)
       \right)\ldots
     \right)
   \right),
\end{equation*}
we can easily check, thanks to \emph{Claim V}, that $\varphi(\alpha,t,\xi)$
is the unique mild
solution $X^{t,\xi,\alpha}$ to \eqref{cont-MKV-differential}, and
the map
\eqref{2020-06-20:06}
has the desired regularity properties.

\smallskip
\noindent \emph{\textbf{Claim VII.} $X^{t,\xi,\alpha}=X^{t,\xi_{t\wedge \cdot},\alpha}$.}\\
This is due to the fact that, for any $\alpha\in \mathcal{U}$,
$a,b\in[0,T]$, $a<b$, $t\in[a,b]$, $\xi\in \mathbf{S}_2(\mathbb{F}),X\in \mathbf{S}_2(\mathbb{F})$,
\begin{equation*}
  \psi_{a,b}(\alpha,t,\xi,X) \ = \ \psi_{a,b}(\alpha,t,\xi_{t\wedge \cdot},X).
\end{equation*}
Hence, the unique fixed point of  $  \psi_{a,b}(\alpha,t,\xi,\cdot)$ has be the same of $  \psi_{a,b}(\alpha,t,\xi_{t\wedge \cdot},\cdot)$.

\smallskip
\noindent \emph{\textbf{Claim VIII.} There exists a constant $C$ such that}
\begin{equation*}
  \|X^{t,\xi,\alpha}\|_{\mathbf{S}_2} \ \leq \ C(1+\|\xi_{t\wedge \cdot}\|_{\mathbf{S}_2}), \qquad \forall\,t\in [0,T],\,\xi\in \mathbf{S}_2,\,\alpha\in \mathcal{U}.
\end{equation*}
Due to \emph{Claim VI}, we only need to show that
\begin{equation}\label{2020-06-20:11}
  \sup_{\alpha\in \mathcal{U}, t\in[0,T]}\|X^{t,0,\alpha}\|_{\mathbf{S}_2}<\infty.
\end{equation}
We have, by using
\eqref{2020-06-20:07},
\eqref{2020-06-20:10},
\eqref{2020-06-20:08},
\eqref{2020-06-20:09}, with $T$ replaced with $t'\in[0,T]$,
\begin{equation*}
  \begin{split}
    \mathbb{E} \left[\|X^{t,0,\alpha}\|_
      {t'}^2\right]
    \ &
    \leq \ 2 \left( \|S\conv{t}
       \left( \mathbf{1}_{[0,t']}F_{b^\alpha}(X^{t,0,\alpha}) \right) \|_{\mathbf{S}_2}^2+
       \|S\sconv{t}
        \left( \mathbf{1}_{[0,t']}F_{\sigma^\alpha}(X^{t,0,\alpha}) \right) \|_{\mathbf{S}_2}^2
\right) \\
&\leq \
2 \left(
  3L^2e^{2\eta T}T
  +  3L^2C_{\eta,T}^2
 \right)
  \int _0^{t'}(1+2\mathbb{E} \left[\|X^{t,0,\alpha}\|_s^2\right])ds .
  \end{split}
\end{equation*}
An application of Gronwall's inequality yields
\begin{equation*}
    \mathbb{E} \left[\|X^{t,0,\alpha}\|_
      {t'}^2\right] \ \leq \ C, \qquad\forall\,t'\in[0,T],
  \end{equation*}
  for some $C$ independent of $\alpha,t$, which proves \eqref{2020-06-20:11} and then \eqref{EstimateX}, after recalling \emph{Claim VII}.
\end{proof}

\begin{proof}[\textbf{Proof of Proposition~\ref{Prop-diff-MKVPD-An}.}]
For $\alpha\in \mathcal{U}$,  let $F_{b^\alpha}, F_{\sigma^\alpha}$
  be as in \eqref{2020-09-19:00}.
  Then, let $\mathrm{id}^{S^n}_t,S^n\conv{t}\#,S^n\sconv{t}\#$ be defined as in \eqref{2020-09-19:00} by replacing $S$ with $S^n$.
  Denote
  $ \mathbb{\overline N}=\mathbb{N}\cup \{\infty\}$.
Let us now define
\begin{equation*}
  \wt \psi\colon \mathcal{U}\times \mathbb{\overline N}\times[0,T]\times \mathbf{S}_2(\mathbb{F})\times \mathbf{S}_2(\mathbb{F})\rightarrow \mathbf{S}_2(\mathbb{F})
\end{equation*}
by
\begin{equation*}
  \wt \psi (\alpha,n,t,\xi,X) \ = \
  \mathrm{id}^{S^n}_t(\xi)+S^n\conv{t}F_{b^\alpha}(X)+S^n\sconv{t}F_{\sigma^\alpha}(X),
\end{equation*}
for all $\alpha\in \mathcal{U}, n\in \mathbb{\overline N}, t\in[0,T],\xi\in \mathbf{S}_2(\mathbb{F}), X\in\mathbf{S}_2(\mathbb{F})$, where we set $S^\infty\coloneqq S$. Let
  \begin{equation*}
    \wt\psi_{a,b}(\alpha,n,t,\xi,X) \ = \
    \wt\psi(\alpha,n,t,\xi,X)_{b\wedge \cdot}
  \end{equation*}
  whenever $a,b\in [0,T]$, $a<b$, $t\in[a,b]$.
  Due to the uniform boundedness of the Yosida approximation, and by arguing as in the proof of \emph{Claims II,III} of
Proposition~\ref{Prop-diff-MKVPD}, one can show that
  \begin{equation}
    \label{2020-10-11:06}
    \begin{minipage}[c]{0.85\linewidth}
      \emph{
    $\wt \psi(\alpha,n,t,\xi,X)$ is Lipschitz continuous in $\xi$, uniformly in $\alpha,n,t,X$. Moreover, there exists $\epsilon>0$ such that, if $b-a<\epsilon$, then $\wt \psi_{a,b}(\alpha,n,t,\xi,X)$ is a contraction in $X$, uniformly in $\alpha,n,t,\xi,a,b$.}
  \end{minipage}
\end{equation}
Now we show that, for $\alpha\in \mathcal{U}, t'\in [0,T],\xi,X\in \mathbf{S}_2(\mathbb{F})$,
  \begin{equation}\label{2020-10-11:04}
    \lim_{\substack{t\rightarrow t'\\n\rightarrow \infty}}
    \|\wt\psi (\alpha,n,t,\xi,X)-\wt\psi (\alpha,\infty,t',\xi,X)\|_{\mathbf{S}_2} \ = \ 0.
  \end{equation}
  First,
  notice that, for $\mathbb{P}$-a.e.\ $\omega\in \Omega$, the range of $\xi(\omega)$ is compact.
  Since $S^n_tx\rightarrow S_tx$ uniformly for $t\in[0,T]$ and $x\in K$, whenever $K\subset H$ is compact,
  an application of Lebesgue's dominated convergence theorem provides
    \begin{equation*}
      \lim_{n\rightarrow \infty}  \| \mathrm{id}^{S^n}_{t'}(\xi)-      \mathrm{id}^{S}_{t'}(\xi)\|_{\mathbf{S}_2(\mathbb{F})}
     \ = \ 0, \qquad  \forall\,t'\in[0,T],\,\xi\in \mathbf{S}_2(\mathbb{F}).
   \end{equation*}
   Secondly,
   for $\alpha\in \mathcal{U}, X\in \mathbf{S}_2(\mathbb{F})$,
   after defining
   \begin{equation*}
     f_n(r) \ = \ \sup_{t\in[0,T]}|(S^n_t-S_t)b_r(X,\mathbb{P}_X,\alpha_r,\mathbb{P}_{\alpha_r})|^2,
   \end{equation*}
   we have, by Lebesgue's dominated convergence theorem,
   \begin{equation*}
        \begin{split}
       \lim_{n\rightarrow \infty}
       \|S^n\conv{t'}F_{b^\alpha}(X)-
       S\conv{t'}F_{b^\alpha}(X)\|_{\mathbf{S}_2}^2
       =&
              \lim_{n\rightarrow \infty}
              \mathbb{E} \left[ \sup_{t\in[t',T]}\left|
                  \int_{t'}^t
                  \left( S^n_{t-r}-S_{t-r} \right) b_r(X,\mathbb{P}_X,\alpha_r,\mathbb{P}_{\alpha_r})dr \right|^2  \right] \\
              \leq &\lim_{n\rightarrow \infty}
              \mathbb{E}
              \left[ \int_{t'}^T
                f_n(r)dr \right] \ = \ 0.
            \end{split}
          \end{equation*}
   Thirdly,
   by \cite[Proposition 1.112]{FabbriGozziSwiech},
 we have
 \begin{equation*}
   \lim_{n\rightarrow \infty}\|S^n\sconv{t'}F_{\sigma^\alpha}(X)-S\sconv{t'}F_{\sigma^\alpha}(X)\|_{\mathbf{S}_2(\mathbb{F})} \ = \ 0, \qquad \forall\,\alpha\in \mathcal{U},\, t\in[0,T],\, X\in \mathbf{S}_2(\mathbb{F}).
 \end{equation*}
 Putting together the above partial results, we get
\begin{equation*}
      \lim_{n\rightarrow \infty}
\|\wt \psi(\alpha,n,t',\xi,X)
  -  \wt \psi(\alpha,\infty,t',\xi,X)\|_{\mathbf{S}_2} \ = \
  0.
\end{equation*}
 Then, to prove \eqref{2020-10-11:04},
it is enough to show that
\begin{equation*}
      \lim_{t\rightarrow t'}
  \sup_{n\in \mathbb{N}}\|\wt \psi(\alpha,n,t,\xi,X)
  -  \wt \psi(\alpha,n,t',\xi,X)\|_{\mathbf{S}_2} \ = \
  0.
\end{equation*}
But this can be obtained by arguing as in the proof of
 \emph{Claim I} of
Proposition~\ref{Prop-diff-MKVPD}, due to the uniform boundedness $\|S^n_t\|_{\mathcal{L}(H)}\leq e^{\wt \eta t}$, for $n\in \mathbb{N}, t\geq 0$.\\
Now, for any small $\epsilon$
  as in \eqref{2020-10-11:06}, and when $a,b\in[0,T]$, $a<b$, $b-a<\epsilon$, denote by
\begin{equation*}
  \wt \varphi_{a,b}\colon \mathcal{U}\times \mathbb{\overline N}\times [a,b]\times \mathbf{S}_2(\mathbb{F})\rightarrow\mathbf{S}_2(\mathbb{F}),\qquad  (\alpha,n,t,\xi)\mapsto X^{n,t,\xi,\alpha}
\end{equation*}
the fixed-point map associated with $\wt \psi_{a,b}$ (similarly as done for $\psi_{a,b},\varphi_{a,b}$ in the proof of Proposition~\ref{Prop-diff-MKVPD}).
Notice that
$X^{n,t,\xi,\alpha}$
is the unique mild solution
of
\eqref{cont-MKV-differential-An} (resp.\ \eqref{cont-MKV-differential}),
when $n\in \mathbb{N}$ (resp.\ $n=\infty$), on the interval $[a,b]$.
Thanks to
\eqref{2020-10-11:06} and \eqref{2020-10-11:04}, we can apply Lemma~\ref{2020-10-11:02} with $\mathcal{R}=\mathcal{U}$, $\mathcal{T}=\mathbb{\overline N}\times [a,b]$, $M=\mathbf{S}_2(\mathbb{F})$, $Y=\mathbf{S}_2(\mathbb{F})$, $\mathcal{E}=\{\{\alpha\}\}_{\alpha\in \mathcal{U}}$, and obtain
\begin{equation}
  \label{2020-10-11:09}
\sup_{\substack{n\in \mathbb{N}\\t\in[a,b]}}
  \|\wt \varphi_{a,b}(\alpha,n,t,\xi)
  -
  \wt \varphi_{a,b}(\alpha,n,t,\xi')
  \|_{\mathbf{S}_2} \ \leq \ C\|\xi-\xi'\|_{\mathbf{S}_2}, \qquad \forall\,\alpha\in \mathcal{U},\, \xi,\xi'\in\mathbf{S}_2(\mathbb{F}),
\end{equation}
and
\begin{equation}
  \label{2020-10-11:05}
  \lim_{\substack{t\rightarrow t'\\t\in[a,b]\\n\rightarrow \infty}}
  \|\wt \varphi_{a,b}(\alpha,n,t,\xi)
  -
  \wt \varphi_{a,b}(\alpha,\infty,t',\xi)\|_{\mathbf{S}_2} \ = \ 0, \qquad \forall\,\alpha\in \mathcal{U},
\end{equation}
for some constant $C$, uniformly for
$a,b\in[0,T]$, $a<b$, $b-a<\epsilon$, $t'\in[a,b]$. To conclude the proof it is enough to
use
\eqref{2020-10-11:09}
and
\eqref{2020-10-11:05} iteratively, recalling the relation between the
mild solution on the subinterval $[a,b]\subset [0,T]$ and the global mild solution on $[0,T]$ (arguing as in the proof of \emph{Claim VI}
of Proposition~\ref{Prop-diff-MKVPD}, simply replacing $[a_i,T]$ with $\mathbb{\overline N}\times [a_i,T]$).
\end{proof}

\section{Law invariance property of $V$: technical results}
\label{App:LawInvariance}

In the present appendix we prove three technical results that are needed in the proof of the law invariance property (Theorem \ref{T:id-law}). The first technical result corresponds to Theorem 6.10 in \cite{Kallenberg}. For the convenience of the reader, we restate it here using the notation adopted in the paper and in the form needed for the proof of Theorem \ref{T:id-law}.

\begin{Lemma}\label{L:Kall}
Consider a probability space $(\hat\Omega,\hat{\mathcal F},\hat\P)$, a measurable space $(E,\mathscr E)$, a Borel space $(\Ur,\mathscr U)$. Consider also two random variables $\Gamma\colon\hat\Omega\rightarrow E$ and $\alpha\colon\hat\Omega\rightarrow\Ur$. Suppose that there exists a random variable $\hat U\colon\hat\Omega\rightarrow\R$, having uniform distribution on $[0,1]$, such that $\Gamma$ and $\hat U$ are independent. Then, there exists a measurable function $\mathrm a\colon E\times[0,1]\rightarrow\Ur$ satisfying
\[
\big(\Gamma,\mathrm a(\Gamma,\hat U)\big) \overset{\mathscr L_{\hat\Omega}}{=} \big(\Gamma,\alpha\big),
\]
where $\overset{\mathscr L_{\hat\Omega}}{=}$ stands for equality in law (between random objects defined on $(\hat\Omega,\hat{\mathcal F},\hat\P)$).
\end{Lemma}
\begin{proof}[\textbf{Proof.}]
See Theorem 6.10 in \cite{Kallenberg}.
\end{proof}

Before stating next result, we introduce the following notation. For every $t\in[0,T]$, let $\F^{B,t}=(\mathcal F_s^{B,t})_{s\geq0}$ be the $\P$-completion of the filtration generated by $(B_{s\vee t}-B_t)_{s\geq0}$. Let also $Prog(\F^{B,t})$ be the $\sigma$-algebra of $[0,T]\times\Omega$ of all $\F^{B,t}$-progressive sets (recall that a set $C\subset[0,T]\times\Omega$ is called $\F^{B,t}$\emph{-progressive} if the corresponding indicator function $\mathds{1}_C$ is an $\F^{B,t}$-progressively measurable process; notice that the family of all $\F^{B,t}$-progressive sets is a $\sigma$-algebra).

\begin{Lemma}\label{L:alpha=a}
Let $t\in[0,T]$, $\alpha\in\Uc$, $\xi\in\S_2(\F)$, with $\xi$ being $\Bc([0,T])\otimes\Fc_t$-measurable. Suppose that there exists an $\Fc_t$-measurable random variable $U_\xi$, having uniform distribution on $[0,1]$ and being independent of $\xi$. Then, there exists a measurable function
\[
\mathrm a\colon\big([0,T]\times\Omega\times C([0,T];H)\times[0,1],Prog(\F^{B,t})\otimes\mathscr B\otimes\mathcal B([0,1])\big) \longrightarrow (\Ur,\mathscr U)
\]
such that $(\mathrm a_s(\xi,U_\xi))_{s\in[0,T]}\in\Uc$, $\mathrm a_s(\cdot,\cdot)$ is constant for every $s<t$, moreover
\begin{equation}\label{EqLaw_bis}
\Big((\xi_s)_{s\in[0,T]},(\mathrm a_s(\xi,U_\xi))_{s\in[t,T]},(B_s-B_t)_{s\in[t,T]}\Big) \overset{\mathscr L}{=} \Big((\xi_s)_{s\in[0,T]},(\alpha_s)_{s\in[t,T]},(B_s-B_t)_{s\in[t,T]}\Big),
\end{equation}
where $\overset{\mathscr L}{=}$ stands for equality in law (between random objects defined on $(\Omega,\Fc,\P)$).
\end{Lemma}
\begin{proof}[\textbf{Proof.}]
Denote
\[
\hat\Omega = [0,T]\times\Omega, \qquad \hat{\mathcal F} = \mathcal B([0,T])\otimes\mathcal F, \qquad \hat\P = \lambda_T\otimes\P,
\]
with $\lambda_T$ being the uniform distribution on $([0,T],\mathcal B([0,T]))$. Then, consider the canonical extension of $U_\xi$ to $\hat\Omega$, which will be denoted by $\hat U_\xi$ (notice that $\hat U_\xi$ has uniform distribution on $[0,1]$ and is independent of $\xi$). Let also $(\bar E,\bar{\mathscr E})$ be the measurable space defined as: $\bar E=[0,T]\times\Omega$ and $\bar{\mathscr E}=Prog(\F^{B,t})$. Then, let $\mathcal I^{B,t}\colon\hat\Omega\rightarrow\bar E$ be the identity map. Finally, define $\Gamma=(\mathcal I^{B,t},\xi)$. Notice that $\Gamma$ is a random variable taking values in the measurable space $(E,\mathscr E)$, with $E=\bar E\times C([0,T];H)=[0,T]\times\Omega\times C([0,T];H)$ and $\mathscr E=\bar{\mathscr E}\otimes\mathscr B=Prog(\F^{B,t})\otimes\mathscr B$. We also observe that $\Gamma$ and $\hat U_\xi$ are independent. We can then apply Lemma \ref{L:Kall}, from which it follows the existence of a map $\bar{\mathrm a}\colon[0,T]\times\Omega\times C([0,T];H)\times[0,1]\rightarrow\Ur$, measurable with respect to the $\sigma$-algebra $Prog(\F^{B,t})\otimes\mathscr B\otimes\mathcal B([0,1])$, such that
\[
\big(\Gamma,\bar{\mathrm a}(\Gamma,\hat U_\xi)\big) \overset{\mathscr L_{\hat\Omega}}{=} \big(\Gamma,\alpha\big),
\]
where $\mathscr L_{\hat\Omega}$ stands for equality in law between random objects defined on $(\hat\Omega,\hat\Fc,\hat\P)$. Then, we deduce
\[
\Big((\xi_s)_{s\in[0,T]},(\bar{\mathrm a}_s(\xi,U_\xi))_{s\in[0,T]},(B_s-B_t)_{s\in[t,T]}\Big) \overset{\mathscr L}{=} \Big((\xi_s)_{s\in[0,T]},(\alpha_s)_{s\in[0,T]},(B_s-B_t)_{s\in[t,T]}\Big),
\]
where we recall that $\mathscr L$ stands for equality in law between random objects defined on $(\Omega,\Fc,\P)$. Finally, let $\mathrm a\colon[0,T]\times\Omega\times C([0,T];H)\times[0,1]\rightarrow\Ur$ be the map given by
\[
\mathrm a_s(\omega,x,u) = u_0\,\mathds{1}_{[0,t)}(s) + \bar{\mathrm a}_s(\omega,x,u)\,\mathds{1}_{[t,T]}(s), \quad \forall\,(s,\omega,x,u)\in[0,T]\times\Omega\times C([0,T];H)\times\Ur,
\]
where $u_0\in\Ur$ is arbitrarily chosen. We have that $(\mathrm a_s(\xi,U_\xi))_{s\in[0,T]}$ is $\F$-progressively measurable (here we use that $\mathrm a_s(\cdot,\cdot)$ is constant for every $s<t$, $\xi$ is $\Bc([0,T])\otimes\Fc_t$-measurable and $U_\xi$ is $\Fc_t$-measurable). So, in particular, $(\mathrm a_s(\xi,U_\xi))_{s\in[0,T]}\in\Uc$ and equality \eqref{EqLaw_bis} holds.
\end{proof}

\begin{Lemma}\label{L:ExistUnif}
Let $t\in[0,T]$ and $\xi\in\S_2(\F)$, with $\xi$ being $\Bc([0,T])\otimes\Fc_t$-measurable. Suppose that there exists $\{x_1,\ldots,x_m\}\subset C([0,T];H)$, with $x_i\neq x_j$ if $i\neq j$, such that
\[
\P_{\xi} = \sum_{i=1}^m p_i\,\delta_{x_i},
\]
where $\delta_{x_i}$ is the Dirac measure at $x_i$ and $p_i>0$, with $\sum_{i=1}^m p_i=1$. Then, there exists an $\Fc_t$-measurable random variable $U_\xi$ having uniform distribution on $[0,1]$ and being independent of the $\mathcal{F}_t$-measurable map
$$
\tilde \xi\colon \Omega\rightarrow C([0,T];H),\ \omega \mapsto \xi(\omega).
$$
\end{Lemma}
\begin{proof}[\textbf{Proof.}]
We recall from Lemma \ref{L:U_G} that there exists a $\Gc$-measurable (so, in particular, $\Fc_t$-measurable) random variable $U_\Gc$ with uniform distribution on $[0,1]$. If $U_\Gc$ is independent of {\color{mrc}$\tilde\xi$}, then we take $U_\xi=U_\Gc$, otherwise we proceed as follows. Denote
\[
E_i \coloneqq  \big\{\omega\in\Omega\colon \xi(\omega)=x_i\big\}, \qquad \forall\,i=1,\ldots,m.
\]
Since $\xi$ is $\Bc([0,T])\otimes\Fc_t$-measurable, it follows that each $E_i$ belongs to $\Fc_t$.
 Now, for each $i=1,\ldots,m$, define the function $F_i\colon[0,1]\rightarrow[0,1]$ as follows
\[
F_i(r) \coloneqq  \P\big(\{U_\Gc\leq r\}\cap E_i\big), \qquad \forall\,r\in[0,1].
\]
Notice that $F_i$ satisfies the following properties:
\begin{itemize}
\item $F_i$ is a non-decreasing function;
\item $F_i$ is continuous;
\item $F_i(0)=0$ and $F_i(1)=\P(E_i)=p_i>0$.
\end{itemize}
For every $i=1,\ldots,m$ ed $n\in\N$, define
\[
r_{i,\frac{k}{2^n}} \coloneqq  \min\bigg\{r\in[0,1]\colon F_i(r)\,=\,\frac{k}{2^n}p_i\bigg\}, \qquad k=1,\ldots,2^n.
\]
Then, for each $n\in\N$ define the random variable $X_n\colon\Omega\rightarrow\R$ as
\[
X_n(\omega) \coloneqq  \sum_{i=1}^m \sum_{k=1}^{2^{n-1}} \mathds{1}_{\left\{r_{i,\frac{2k-1}{2^n}}<U_\Gc\leq r_{i,\frac{2k}{2^n}}\right\}\cap E_i}(\omega), \qquad \forall\,\omega\in\Omega.
\]
Notice that $\{X_n\}_n$ is a sequence of independent and identically distributed Bernoulli random variables of parameter $1/2$. In addition, $\{X_n\}_n$ and {\color{mrc}$\tilde\xi$} are independent. Then, consider the random variable $U_\xi\colon\Omega\rightarrow\R$ given by
\[
U_\xi \coloneqq  \sum_{n=1}^\infty \frac{1}{2^n} X_n.
\]
We have that $U_\xi$ and $\tilde\xi$ are independent. Finally, it follows for instance from Lemma 3.20 in \cite{Kallenberg} that the random variable $U_\xi$ has uniform distribution on $[0,1]$.
\end{proof}

\section{Pathwise measure derivative: law invariance property}
\label{App:PathDeriv}

We devote this appendix to extend a useful result (firstly proved in \cite[Section 6.1]{carda12} in the case of the set $\R^d$) to the case of the set $C([0,T];H)$. More precisely, our aim is to prove that  $D\hat\Phi(t,\hat\xi)$ only depends on the law of $\hat\xi$. Here we substantially follow the idea of \cite{WuZhang18} and
\cite[Section 2.3]{WuZhang18PPDE}. However, since our setting is more general, we present the full proof.

\subsection{The discrete case}

We first consider the case where the random variable $\hat\xi$ takes a countable number of values.
We assume the following.

\begin{Assumption}{\bf(A$_\xi$)}
\label{hp:xidiscrete}
Consider a sequence $\{\hat x_i\}_{i \in \N}\subset D([0,T];H)$, $\hat x_i\neq\hat x_j$ if  $i\neq j$. We assume that the random variable $\hat\xi$  has law
\[
\P_{\hat\xi} = \sum_{i=1}^N p_i\delta_{\hat x_i}
\]
where $N\in \mathbb{N}\setminus\{0\}$, $\delta_{\hat x_i}$ denotes the Dirac measure at $\hat x_i$ and $p_i> 0$,  with $\sum_{i=1}^N p_i=1$.
\end{Assumption}

\begin{Lemma}\label{L:Discrete}
Let $(t,\hat\xi)\in\mathfrak{\hat H}$, with $\hat\xi$ as in Assumption \textup{\ref{hp:xidiscrete}}.
Let $\hat\varphi\colon\mathscr{\hat H}\to\R$ be such that its lifting
$\hat\Phi$ is pathwise differentiable in space at $(t,\hat\xi)$. Then there exists a measurable function
$\hat g\colon D([0,T],H)\rightarrow H$ such that
$\hat g(\hat\xi)\in L^2(\Omega;H)$ and
\begin{equation}\label{2018-11-06:00}
D\hat\Phi(t,\hat\xi) = \hat g(\hat\xi)\,.
\end{equation}
The function $\hat g$ can be defined by $\hat g(\hat x)\coloneqq 0$  if $\hat x\not\in \{\hat x_i\}_{i=1,\ldots,N}$, and
\begin{align}\label{eq:gdiscrete}
 \<\hat g(\hat x_i),h\>_{H}
 &\coloneqq p_i^{-1}\E\left[\<D\hat\Phi(t,\hat\xi),h\mathds{1}_{\hat\xi^{-1}(\hat x_i)}\>_H\right] \notag  \\
 &= \lim_{\eps\rightarrow 0^+}
  \frac{1}{\eps p_i}
  \bigg[
    \hat\varphi\bigg(t,
      \sum_{\substack{j=1,\ldots,N\\j\ne i}} p_j\delta_{\hat x_j}+p_i\delta_{\hat x_i+\eps h\mathds{1}_{[t,T]}}
    \bigg)
    - \hat\varphi\bigg(t,\sum_{j=1}^N p_j\delta_{\hat x_j}\bigg)\bigg],
\end{align}
for all $h\in H$ and $i=1,\ldots,N$.
\end{Lemma}

\begin{proof}[\textbf{Proof.}]
Fix $i=1,\ldots,N$ and $h \in H, h \ne 0$.
Take any
measurable set $A'\subset A_i\coloneqq \{\hat\xi=\hat x_i\}$
and set $\hat\eta\coloneqq h \mathds{1}_{A'} \in L^2\big(\Omega;H)$.
Notice that, for any $\eps > 0$ we have
$$
\P_{\hat\xi + \eps\hat\eta} = \sum_{j \ne i}p_j \delta_{\hat x_j} + \P(A')\delta_{\hat x_i+\eps h\mathds{1}_{[t,T]}}
+\big(p_i- \P(A')\big) \delta_{\hat x_i}
$$
which depends only on $\P_{\hat\xi}$ and on $\P(A')$.
We get
\begin{align}
  \label{eq:derdiscrete1}
&\E\Big[\<D\hat\Phi(t,\hat\xi),h\mathds{1}_{A'}\>_H\Big] \notag \\
&= \lim_{\eps \to 0^+}
\frac{{\displaystyle\hat\varphi\Big(t, \sum_{j \ne i}p_j \delta_{\hat x_j} + \P(A')\delta_{\hat x_i+\eps h\mathds{1}_{[t,T]}}
+\big(p_i- \P(A')\big) \delta_{\hat x_i} \Big)\!-\!\hat\varphi \Big( t,\sum_{j=1}^Np_j \delta_{\hat x_j} \Big) }}{\eps}.
\end{align}
Notice that the map
  \begin{equation*}
   H\times L^2(\Omega;\mathbb{R})\longrightarrow L^2(\Omega;H),\qquad (h,\zeta) \longmapsto \zeta h
  \end{equation*}
is bilinear and continuous, hence
\begin{equation*}
\Lambda\colon  H\times L^2(\Omega;\mathbb{R})\longrightarrow \mathbb{R}, \qquad (h,\zeta) \longmapsto \E[\langle D\hat\Phi(t,\hat\xi),\zeta h\rangle_H]
\end{equation*}
is a bilinear and continuous form.
By the Riesz representation theorem,
there exists a (unique) linear and continuous map
\begin{equation*}
  T\colon H\longrightarrow L^2(\Omega;\mathbb{R})
\end{equation*}
representing $\Lambda$, namely
\begin{equation}
  \label{2018-11-06:01}
\Lambda(h,\zeta) = \E[\langle D\hat\Phi(t,\hat\xi),\zeta h\rangle_H] = \E \left[ T(h)\zeta \right] \qquad \forall\,(h,\zeta)\in H\times L^2(\Omega;\mathbb{R}).
\end{equation}
By \eqref{2018-11-06:01} and \eqref{eq:derdiscrete1}, it follows that
$\E[\langle D\hat\Phi(t,\hat\xi),h\mathds{1}_{A'}\rangle_H]=\E\left[
T(h) \mathds{1}_{A'} \right] $ depends only on the law of $\hat\xi$ and on $\P(A')$ for every measurable $A'\subset \{\hat\xi = \hat x_i\}$.
Recall from Lemma \ref{L:U_G}
 that the probability space $(\Omega,\Fc,\P)$ supports a random variable with uniform distribution on $[0,1]$. We can  then apply Lemma~2 of \cite{WuZhang18},
getting that
$T(h)$
is $\P$-a.s.\ constant on $\{\hat\xi = \hat x_i\}$.
For $i=1,\ldots,N$, define the map $\psi_i\colon H\rightarrow \mathbb{R}$ by
\begin{equation*}
\psi_i(h)  \coloneqq  \frac{1}{p_i}\,\E[T(h)\mathds{1}_{A_i}], \qquad \forall\,h\in H.
\end{equation*}
Then $\psi_i(h)=T(h)$, $\P$-a.s.\ on\ $A_i$.
Notice that, by the very definition of $\psi_i$ and the linearity of $\Lambda$, we have
\begin{equation*}
 \sum_{i=1}^N
\psi_i(h)
\E[\zeta\mathds{1}_{A_i}] = \sum_{i=1}^N \E[T(h)\zeta\mathds{1}_{A_i}] = \
  \sum_{i=1}^N \Lambda(h,\zeta\mathds{1}_{A_i}) = \
\Lambda(h,\zeta)
 = \E[\langle D\hat\Phi(t,\hat\xi),h\zeta\rangle_H],
\end{equation*}
for all $h\in H$, $\zeta\in L^2(\Omega;\R)$. The linearity and continuity of $T$ entails $\psi_i\in H'$. Therefore, it can be identified to some $\phi_i\in H$.
Define $\hat g\colon D([0,T];H)\rightarrow H$ by
\begin{equation*}
  \hat g(\hat x)  \coloneqq  \sum_{i=1}^N \phi_i\,\delta_{\hat x_i}(\hat x)\qquad \forall\,\hat x\in D([0,T];H).
\end{equation*}
Notice that, $\P$-a.s. on $\Omega$ and for all $h\in H$,
\begin{align}\label{2018-11-06:03}
      \langle \hat g(\hat \xi),h\rangle_{H} = \
\bigg\langle \sum_{i=1}^N \phi_i\,\delta_{\hat x_i}(\hat\xi),
h
\bigg\rangle_{H} &= \bigg\langle \sum_{i=1}^N \phi_i\, \mathds{1}_{A_i},
h
\bigg\rangle_{H} \notag \\
&= \sum_{i=1}^N \psi_i(h)\,\mathds{1}_{A_i} = \sum_{i=1}^N T(h)\,\mathds{1} _{A_i} = T(h).
\end{align}
Now let $Y\coloneqq \sum_{k=1}^M a_k\mathds{1}_{B_k}$ be a simple function, where $M\in \mathbb{N}, a_k\in H, B_k\in\Fc$.
By \eqref{2018-11-06:03} it follows
\begin{align}
  \label{2018-11-06:04}
   \E[\langle D\hat\Phi(t,\hat\xi),Y\rangle_H] = \sum_{k=1}^M \E[\langle D\hat\Phi(t,\hat\xi),a_k\mathds{1}_{B_k}\rangle_H] &=
\sum_{k=1}^M \E \left[ T(a_k)\mathds{1}_{B_k} \right] \\
&= \sum_{k=1}^M \E\left[ \langle
\hat g(\hat\xi)
,a_k
\rangle_{H}\mathds{1}_{B_k} \right] = \E\left[
\langle \hat g(\hat\xi),Y\rangle_{H}
 \right]. \notag
\end{align}
Since \eqref{2018-11-06:04}  holds for any simple function $Y$, we conclude
\begin{equation*}
  D\hat{\Phi}(t,\hat{ \xi})=\hat{ g}(\hat{ \xi})\qquad \mathbb{P}\textrm{-a.s.}
\end{equation*}
which is \eqref{2018-11-06:00}. Finally,
\eqref{eq:gdiscrete}
comes from
\eqref{2018-11-06:00} and
\eqref{eq:derdiscrete1}.
\end{proof}

\subsection{The general case}

\begin{proof}[\textbf{{Proof of Lemma \ref{L:General}.}}]
We proceed by approximation.
For $n\in \mathbb{N}$, let   $\{D^n_i\}_{i\in \mathbb{N}}$ be a partition of
 $D([0,T];H)$ made by Borel sets such that $\operatorname{diam}_{d_{\textup{Sk}}}(D^n_i)<2^{-n}$.
Such a partition exists because of the separability of $(D([0,T];H),d_{\textup{Sk}})$.
 For $i\in \mathbb{N}$, choose $d^n_i\in D^n_i$.
Define, for every $n \in \N$,
\begin{equation*}
\hat\zeta_n  \coloneqq  \sum_{i \in \N}d^n_{i}\,\mathds{1}_{D^n_i}(\hat\xi).
\end{equation*}
Then $d_{\textup{Sk}}(\hat\xi(\omega),\hat\zeta_n(\omega))\leq 2^{-n}$ for all $\omega\in\Omega$ and all $n\in\N$. Therefore
$\hat\zeta_n\in L^2(\Omega;D([0,T],H))$
and
 $\hat\zeta_n\rightarrow\hat\xi$
 uniformly.
 By a diagonal argument, we can choose $N_n\in \mathbb{N}$ such that the sequence
$\{\hat\xi_n\}_{n\in \mathbb{N}}$, defined by
\begin{equation}\label{2018-11-07:01}
  \hat\xi_n  \coloneqq  \sum_{i=1}^{N_n}d^n_i\mathds{1} _{D^n_i}(\hat\xi)
\end{equation}
converges to $\hat\xi$ both $\P$-a.s.\ and in $L^2(\Omega;D([0,T];H))$.
By Lemma \ref{L:Discrete}, there exists  for each $n$ a function $\hat g_n\colon D([0,T];H)\rightarrow H$ such that
\begin{equation}
\label{2018-11-06:08}
 \hat g_n(\hat\xi_n)
\in L^2(\Omega;H)
\  \mbox{and}\   D\hat\Phi(t,\hat\xi_n)=\hat g_n(\hat\xi_n).
\end{equation}
Define
\begin{equation*}
  \tilde{g}_n  \coloneqq  \sum_{i=1}^{N_n} \hat g_n \left( d^n_i  \right)
 \mathds{1}_{D^n_i}.
\end{equation*}
Notice  that $\hat g_n(\hat\xi_n)=\tilde{g}_n(\hat\xi)$.
Then, by
\eqref{2018-11-06:08}
and by
continuity of $D\hat\Phi(t,\cdot)$ in a neighborhood of $\hat\xi$, we
have
$  \tilde{ g}_n(\hat\xi)\rightarrow D\hat{\Phi}(t,\hat{ \xi})$
in $L^2(\Omega;H)$.
Let $\{\tilde{ g}_{n_k}(\hat\xi)\}_k$
be a subsequence
converging $\P$-a.s. to $D\hat{\Phi}(t,\hat{ \xi})$.
Define
\begin{equation*}
  S  \coloneqq  \big\{\hat x\in D([0,T];H)\colon
 \{\tilde{g}_{n_k}(\hat x)\}_k\mbox{\ is convergent}\big\},
\end{equation*}
and $\hat g\colon D([0,T];H)\rightarrow H$ by
$  \hat g\coloneqq \mathds{1}_{S}\lim_k \tilde{ g}_{n_k}$.
Clearly $\hat g$ is measurable.
Notice that
$\mathbb{P}(\hat\xi\in S)=1$, then
\begin{equation}\label{2020-11-13:00}
  \lim_{k\rightarrow \infty}
  \tilde{g}_{n_k}(\hat\xi)
  =
  \hat g(\hat\xi)\quad
  \mathbb{P}\textrm{-a.s.\ in $D([0,T];H)$}
\end{equation}
which provides
$\hat g(\hat\xi) = D\hat{\Phi}(t,\hat{\xi})$.

Finally,
if $\hat\xi'$ is distributed as $\hat\xi$,
then we can perform exactly the same steps by replacing $\hat\xi,\hat\xi_n$ by $\hat\xi',\hat\xi_n'$
and
by choosing the same $N_n$ in \eqref{2018-11-07:01} and then the same $\tilde g_n,\hat g$. Moreover, $\P(\hat\xi'\in S)=1$ as well. Then
\eqref{2020-11-13:00}
holds with the same  $\hat g$ and with $\hat\xi$ replaced by $\hat\xi'$, which entails
$\hat{g}(\hat{\xi}')=D\hat{ \Phi}(t,\hat{\xi}')$.

  The remaining part of the proof goes exactly as in the proof of Corollary 2.3 of \cite{WuZhang18PPDE}.
\end{proof}

\section{Proof of It\^o's formula}\label{ItoProof}

\begin{proof}[\textbf{Proof of Theorem \ref{T:ItoD}}]
  Let $\xi, \hat \varphi$ be as in the statement.
  In what follows, we will tacitly make use of the non-anticipative property of $\hat{\varphi}$ and of its derivatives.
  Let $t\in[0,T), s\in(t,T]$.
  For $n\in \mathbb{N}$, let
  \begin{equation*}
t^n_{-1}\coloneqq t,\qquad    t^n_k \coloneqq t+ \frac{k}{n}(s-t)\quad \forall\ k=1,\ldots,n
  \end{equation*}
  and
   $ X^n = \sum_{k=1}^n X_{t^n_{k-1}} \mathds{1}_{[t^n_{k-1},t^n_k)}$.
  Clearly $\mathbb{E} \left[ \|X^n\|_T \right] <\infty$.
Observe that
\begin{equation}
  \label{2020-12-04:00}
  X^n_{t^n_k\wedge \cdot}= X^n_{t^n_{k-1}\wedge \cdot}+(X^n_{t^n_k}-X^n_{t^n_{k-1}})\mathds{1}_{[t^n_k,T]}\qquad  \forall\ k=1,\ldots,n.
\end{equation}
By continuity of $X$, we clearly have
\begin{equation}\label{2020-12-04:03}
  \lim_{n\rightarrow \infty}
  \sup_{r\in [0,T]}\|X_r(\omega)-X_r^n(\omega)\|_H=0\qquad  \forall \omega\in\Omega.
\end{equation}
Denote
\begin{equation}
  \label{2020-12-04:04}
X^{n,\theta,k}=  X^n_{t^n_{k-1}\wedge \cdot}+\theta(X^n_{t^n_k}-X^n_{t^n_{k-1}})\qquad \forall \theta\in[0,1].
\end{equation}
By continuity of $X$, if $n\rightarrow \infty$ and if $\{k_n\}_{n\in \mathbb{N}}\subset \mathbb{N}$ is a sequence such that $k_n\in\{1,\ldots,n\}$ and $t^n_{k_n}\rightarrow r$ in $[t,s]$, then
\begin{equation}
  \label{2020-12-04:05}
  \lim_{n\rightarrow \infty}
  \sup_{\theta\in [0,1]}\|X_{r\wedge \cdot}(\omega)-X_{t_{k_n}\wedge \cdot}^{n,\theta,k_n}(\omega)\|_T=0\qquad \forall \omega\in\Omega.
\end{equation}

\noindent
Consider the difference
$\hat \varphi(s,X^n)-\hat\varphi (t,X^n)$, written as
\begin{equation}
  \label{2020-11-13:01}
  \begin{split}
      \hat \varphi(s,\hat{\mathbb{P}}_{X^n})-\hat\varphi (t,\hat{ \mathbb{P}}_{X^n})
  =&
  \sum_{k=1}^n
  \left( \hat \varphi(t^n_k,\hat{\mathbb{P}}_{X^n})-\hat\varphi (t^n_k,\hat{\mathbb{P}}_{X^n_{t^n_{k-1}\wedge \cdot}})
     \right) \\
& +
 \sum_{k=1}^n
  \left(
  \hat \varphi(t^n_k,\hat{ \mathbb{P}}_{X^n_{t^n_{k-1}\wedge \cdot}})-\hat\varphi (t^n_{k-1},
  \hat{ \mathbb{P}}_{X^n})
   \right) .
 \end{split}
\end{equation}
Let us firt take in consideration the quantity $  \hat \varphi(t^n_k,\hat{ \mathbb{P}}_{X^n})-\hat\varphi (t^n_k,\hat{ \mathbb{P}}_{X^n_{t^n_{k-1}\wedge \cdot}})$.
By
\eqref{2020-12-04:00} and by the the fact that $\hat{\varphi}\in \boldsymbol C_b^{1,2}(\mathscr{\hat H})$, we have
\begin{equation}\label{2020-12-04:17}
  \begin{split}
    \hat \varphi(t^n_k,\hat{ \mathbb{P}}_{X^n})-\hat\varphi (t^n_k,\hat{ \mathbb{P}}_{X^n_{t^n_{k-1}\wedge \cdot}})=&
    \hat \Phi(t^n_k,X^n)-\hat\Phi (t^n_k,X^n_{t^n_{k-1}\wedge \cdot})\\
    =&\int_0^1 \mathbb{E} \left[ \langle
      D\hat \Phi \left( t^n_k,
X^{n,\theta,k},
      \right) ,X^n_{t^n_k}-X^n_{t^n_{k-1}}
      \rangle \right] d\theta\\
    =&\int_0^1 \mathbb{E} \left[ \langle
      \partial_\mu \hat \varphi \left( t^n_k,
        \hat{ \mathbb{P}}_{X^{n,\theta,k}}
      \right)
      \left(X^{n,\theta,k}
      \right) , X^n_{t^n_k}-X^n_{t^n_{k-1}}
      \rangle \right] d\theta\\
    =&\textrm{(after recalling \eqref{2020-12-04:02})}\\
        =&\int_0^1 \mathbb{E} \left[ \langle
      \partial_\mu \hat \varphi \left( t^n_k,
        \hat{\mathbb{P}}_{X^{n,\theta,k}}
      \right)
      \left(X^{n,\theta,k}
      \right) , \int_{t^n_{k-1}}^{t^n_k}F_rdr
      \rangle \right] d\theta\\
    &+
        \int_0^1 \mathbb{E} \left[ \langle
      \partial_\mu \hat \varphi \left( t^n_k,
        \hat{ \mathbb{P}}_{X^{n,\theta,k}}
      \right)
      \left(X^{n,\theta,k}
      \right) ,
       \int_{t^n_{k-1}}^{t^n_k}G_rdB_r
       \rangle \right] d\theta\\
     =& \mathbf{I}^{n,k}_F + \mathbf{I}^{n,k}_B.
  \end{split}
\end{equation}
Observe that, by the boundedness assumption on $ \partial _\mu\hat{\varphi}$  and the integrability assumption on $F$, we can
commute the integral
$\int_{t^n_{k-1}}^{t^n_k}\cdot dr$ first with
$\mathbb{E}$
and then
(since the map $[0,T] \mapsto L^1(\Omega;H),\ r \mapsto F_r$ is measurable and integrable)
with
$\int_0^1\cdot d\theta$
\footnote{But notice that the integral in $\int_0^1\cdot d\theta$ cannot be commuted with $\mathbb{E}$
  before showing the existence of a jointly measurable representant $(\theta,\omega) \mapsto  Z(\theta,\omega)$
  of
  $\theta \mapsto   \partial_\mu \hat{\varphi} (t^n_k, \hat{ \mathbb{P}}_{X^{n,\theta,k}})(X^{n,\theta,k})(\omega)$.}.
By summing over $k$ the quantity $\mathbf{I}^{n,k}_F$, we then have
\begin{equation}\label{2020-12-04:07}
  \begin{split}
    \sum_{k=1}^n    \mathbf{I}^{n,k}_F
    =&
    \sum_{k=1}^n
    \int_{t^n_{k-1}}^{t^n_k}
      \int_0^1
\mathbb{E} \left[ \langle
  \partial_\mu \hat \varphi \left( t^n_k,
       \hat{ \mathbb{P}}_{X^{n,\theta,k}}
      \right)
      \left(X^{n,\theta,k}
      \right), F_r
      \rangle \right]   d\theta dr\\
    =&
    \int_t^s
    \int_0^1
\mathbb{E} \left[ \langle
           \sum_{k=1}^n
   \partial_\mu \hat \varphi \left( t^n_k,
        \hat{ \mathbb{P}}_{X^{n,\theta,k}}
      \right)
      \left(X^{n,\theta,k}
      \right)\mathds{1}_{[t^n_{k-1},t^n_k)}(r) , F_r
      \rangle \right] d\theta dr.
  \end{split}
\end{equation}
By continuity of $ \partial _\mu \hat{\varphi}$ and by
\eqref{2020-12-04:05}, we have
\begin{equation}
  \label{2020-12-04:06}
  \lim_{n\rightarrow \infty}
  \left|
  \sum_{k=1}^n
  \partial_\mu \hat \varphi \left( t^n_k,
    \hat{ \mathbb{P}}_{X^{n,\theta,k}}
  \right)
  \left(X^{n,\theta,k}(\omega)
  \right)\mathds{1}_{[t^n_{k-1},t^n_k)}(r)
  -
    \partial_\mu \hat \varphi \left( r,
    \hat{ \mathbb{P}}_X
  \right)
  \left(X(\omega)\right)
  \right|_H=0\qquad \forall \omega\in\Omega.
\end{equation}
By Lebesgue's dominated convergence theorem  and by
\eqref{2020-12-04:06},\eqref{2020-12-04:07}, we conclude
\begin{equation}
  \label{2020-12-04:08}
  \lim_{n\rightarrow \infty}
      \sum_{k=1}^n    \mathbf{I}^{n,k}_F
      =
      \int_t^s\mathbb{E}
       \left[ \langle
   \partial_\mu \hat \varphi \left( r,
        \hat{ \mathbb{P}}_X
      \right)
      \left(X
      \right), F_r
      \rangle \right] dr.
\end{equation}

We now address $\mathbf{I}^{n,k}_B$.  In this case, we cannot
immediately commute $\int_{t_{k-1}^n}^{t^n_k}\cdot dB_r$ with
$\mathbb{E}$, because $X^{n,\theta,k}$ is a-priori only
$\mathcal{F}_{t^n_k}$-measurable.  We then add a null term and then
expand the sum.  Indeed, we have
\begin{multline*}
  \int_0^1 \mathbb{E} \left[ \langle \partial_\mu \hat \varphi \left(
      t^n_k, \hat{ \mathbb{P}}_{X^{n,\theta,k}} \right)
    \left(X^{n,\theta,k-1} \right) , \int_{t^n_{k-1}}^{t^n_k}G_rdB_r
    \rangle \right] d\theta =\\
  =\int_0^1 \mathbb{E} \left[
    \int_{t^n_{k-1}}^{t^n_k} \langle \partial_\mu \hat \varphi \left(
      t^n_k, \hat{ \mathbb{P}}_{X^{n,\theta,k}} \right)
    \left(X^{n,\theta,k-1} \right) , G_rdB_r \rangle \right] d\theta =0.
\end{multline*}
 Then, by using the continous second-order pathwise derivative in mesure and space of $\hat{\varphi}$
\begin{equation}\label{2020-12-04:16}
  \begin{split}
    \mathbf{I}^{n,k}_B=&
            \int_0^1 \mathbb{E} \left[ \langle
      \partial_\mu \hat \varphi \left( t^n_k,
        \hat{ \mathbb{P}}_{X^{n,\theta,k}}
      \right)
      \left(X^{n,\theta,k}
      \right) ,
       \int_{t^n_{k-1}}^{t^n_k}G_rdB_r
       \rangle \right] d\theta\\
     =&
                 \int_0^1 \mathbb{E} \left[ \langle
      \partial_\mu \hat \varphi \left( t^n_k,
        \hat{ \mathbb{P}}_{X^{n,\theta,k}}
      \right)
      \left(X^{n,\theta,k}
      \right)-
            \partial_\mu \hat \varphi \left( t^n_k,
        \hat{ \mathbb{P}}_{X^{n,\theta,k}}
      \right)
      \left(X^{n,\theta,k-1}
      \right),
       \int_{t^n_{k-1}}^{t^n_k}G_rdB_r
       \rangle \right] d\theta\\
     =&
     \int_0^1
     \mathbb{E}
     \left[
       \langle
              \int_0^1
       \partial _x \partial _\mu \hat{ \varphi}
       (t^n_k,\hat{\mathbb{P}}_{X^{n,\theta,k}})(X^{n,\epsilon\theta,k})
       \left(
         X^n_{t^n_k}-X^n_{t^n_{k-1}}
          \right)\theta d\epsilon,\int_{t^n_{k-1}}^{t^n_k} G_r dB_r
          \rangle \right] d\theta\\
        =&
             \int_0^1
     \mathbb{E}
     \left[
       \langle
              \int_0^1
       \partial _x \partial _\mu \hat{ \varphi}
       (t^n_k,\hat{\mathbb{P}}_{X^{n,\theta,k}})(X^{n,\epsilon\theta,k})
       \left( \int_{t^n_{k-1}}^{t^n_k}
         F_rdr \right)
\theta d\epsilon,\int_{t^n_{k-1}}^{t^n_k} G_r dB_r
\rangle \right] d\theta\\
&+
             \int_0^1
     \mathbb{E}
     \left[
       \langle
              \int_0^1
       \partial _x \partial _\mu \hat{ \varphi}
       (t^n_k,\hat{\mathbb{P}}_{X^{n,\theta,k}})(X^{n,\epsilon\theta,k})
       \left( \int_{t^n_{k-1}}^{t^n_k}
         G_rdB_r \right)
\theta d\epsilon,\int_{t^n_{k-1}}^{t^n_k} G_r dB_r
\rangle \right] d\theta\\
     =& \mathbf{I}^{n,k}_{BF} + \mathbf{I}^{n,k}_{BB}.
  \end{split}
\end{equation}
Notice that
\begin{equation*}
  \begin{split}
 \sum_{k=1}^n
    \mathbb{E}&
    \left[
      \left|
\langle  \partial _x \partial _\mu \hat{ \varphi}
       (t^n_k,\hat{\mathbb{P}}_{X^{n,\theta,k}})(X^{n,\epsilon\theta,k})
       \left( \int_{t^n_{k-1}}^{t^n_k}
         F_rdr \right)
        ,\int_{t^n_{k-1}}^{t^n_k} G_r dB_r
       \rangle\right| \right]
   \leq \\
   &\qquad \leq
   M\sum_{k=1}^n
   \left(
     \mathbb{E}
     \left[
         \left|\int_{t^n_{k-1}}^{t^n_k}
         F_r dr\right|_H^2
       \right]
            \mathbb{E}
     \left[
           \left|\int_{t^n_{k-1}}^{t^n_k}
              G_r dB_r
            \right|_H^2
          \right]
        \right) ^{1/2}\\
        &\qquad\leq
        M\sum _{k=1}^n
           \left(\frac{s-t}{n}
     \mathbb{E}
     \left[
         \int_{t^n_{k-1}}^{t^n_k}
         \left|F_r\right|_H^2 dr
       \right]
            \mathbb{E}
     \left[
       \int_{t^n_{k-1}}^{t^n_k}
              |G_r|_{\mathcal{L}_2(K;H)}^2 dr
          \right]
        \right) ^{1/2}\\
        &\qquad\leq
         M  \left(  \frac{s-t}{n} \right) ^{1/2} \sum _{k=1}^n
           \left(
     \mathbb{E}
     \left[
         \int_{t^n_{k-1}}^{t^n_k}
         \left|F_r\right|_H^2 dr
       \right]
            \mathbb{E}
     \left[
       \int_{t^n_{k-1}}^{t^n_k}
              |G_r|_{\mathcal{L}_2(K;H)}^2 dr
          \right]
        \right) ^{1/2}\\
                &\qquad\leq
         M  \left( \frac{s-t}{n}  \right) ^{1/2}
           \left(
     \mathbb{E}
     \left[
         \int_t^s
         \left|F_r\right|_H^2 dr
       \right]
            \mathbb{E}
     \left[
       \int_t^s
              |G_r|_{\mathcal{L}_2(K;H)}^2 dr
          \right]
        \right) ^{1/2},
  \end{split}
\end{equation*}
which goes to $0$ as $n\rightarrow\infty$.
By Lebesgue's dominated convergence theorem, it then follows
\begin{equation}
  \label{2020-12-04:09}
  \lim_{n\rightarrow \infty}
  \sum_{k=1}^n
  \mathbf{I}^{n,k}_{BF}=0.
\end{equation}
Now, in order to compute $\mathbf{I}^{n,k}_{BB}$, consider first that, for $m\in \mathbb{N}$ and $n\geq m$,
\begin{equation*}
  \begin{split}
\sum_{k=1}^n    \mathbb{E}&
    \left[
      \left|
        \langle
         \left(
   \partial _x \partial _\mu \hat{ \varphi}
   (t^n_k,\hat{\mathbb{P}}_{X^{n,\theta,k}})(X^{n,\epsilon\theta,k})-
      \partial _x \partial _\mu \hat{ \varphi}
       (t^n_k,\hat{\mathbb{P}}_{X^{n,\theta,k}})(X^{n,\epsilon\theta,k-1}) \right)
       \left( \int_{t^n_{k-1}}^{t^n_k}
         G_rdB_r \right)
       ,\int_{t^n_{k-1}}^{t^n_k} G_r dB_r
       \rangle \right| \right]\\
   &\leq
\sum_{k=1}^n   \mathbb{E}
    \left[
      \left|
   \partial _x \partial _\mu \hat{ \varphi}
   (t^n_k,\hat{\mathbb{P}}_{X^{n,\theta,k}})(X^{n,\epsilon\theta,k})-
      \partial _x \partial _\mu \hat{ \varphi}
       (t^n_k,\hat{\mathbb{P}}_{X^{n,\theta,k}})(X^{n,\epsilon\theta,k-1}) \right|_{L(H)}
       \left| \int_{t^n_{k-1}}^{t^n_k}
         G_rdB_r \right|_H^2
     \right]   \\
        &\leq
   \mathbb{E}
    \left[
   \sup_{\substack{n\in \mathbb{N},\ n\geq m,\\k=1,\ldots,n}}   \left|
   \partial _x \partial _\mu \hat{ \varphi}
   (t^n_k,\hat{\mathbb{P}}_{X^{n,\theta,k}})(X^{n,\epsilon\theta,k})-
      \partial _x \partial _\mu \hat{ \varphi}
       (t^n_k,\hat{\mathbb{P}}_{X^{n,\theta,k}})(X^{n,\epsilon\theta,k-1}) \right|_{L(H)}
       \left| \int_t^s
         G_rdB_r \right|_H^2
        \right]   .
  \end{split}
\end{equation*}
By  continuity of $ \partial _x \partial_\mu \hat{ \varphi}$ and by
\eqref{2020-12-04:05} it follows
\begin{equation*}
  \lim_{m\rightarrow\infty}
     \sup_{\substack{n\in \mathbb{N},\ n\geq m,\\k=1,\ldots,n}}   \left|
   \partial _x \partial _\mu \hat{ \varphi}
   (t^n_k,\hat{\mathbb{P}}_{X^{n,\theta,k}})(X^{n,\epsilon\theta,k}(\omega))-
      \partial _x \partial _\mu \hat{ \varphi}
       (t^n_k,\hat{\mathbb{P}}_{X^{n,\theta,k}})(X^{n,\epsilon\theta,k-1}(\omega)) \right|=0,
   \end{equation*}
   for every $\omega\in\Omega$, and then, by Lebesgue's dominated convergence theorem,
   \begin{multline}\label{2020-12-04:10}
     \lim_{n\rightarrow\infty} \sum_{k=1}^n \mathbb{E} \left[ \left|
         \langle \left( \partial _x \partial _\mu \hat{ \varphi}
           (t^n_k,\hat{\mathbb{P}}_{X^{n,\theta,k}})(X^{n,\epsilon\theta,k})-
           \partial _x \partial _\mu \hat{ \varphi}
           (t^n_k,\hat{\mathbb{P}}_{X^{n,\theta,k}})(X^{n,\epsilon\theta,k-1})
         \right) \left( \int_{t^n_{k-1}}^{t^n_k} G_rdB_r \right)
         ,\right.\right.\\
     \left.\left.  \int_{t^n_{k-1}}^{t^n_k} G_r dB_r \rangle \right|
     \right]=0.
   \end{multline}
By \eqref{2020-12-04:10} and  by Lebesgue's dominated convergence theorem, we can then write
\begin{equation}\label{2020-12-04:11}
  \begin{split}
   & \lim_{n\rightarrow \infty}\sum_{k=1}^n
    \mathbf{I}^{n,k}_{BB}
    =\\
     &=   \lim_{n\rightarrow \infty}\sum_{k=1}^n
                \int_0^1
     \mathbb{E}
     \left[
       \langle
              \int_0^1
       \partial _x \partial _\mu \hat{ \varphi}
       (t^n_k,\hat{\mathbb{P}}_{X^{n,\theta,k}})(X^{n,\epsilon\theta,k-1})
       \left( \int_{t^n_{k-1}}^{t^n_k}
         G_rdB_r \right)
\theta d\epsilon,\int_{t^n_{k-1}}^{t^n_k} G_r dB_r
\rangle \right] d\theta\\
&
=
\lim_{n\rightarrow \infty}\sum_{k=1}^n
                \int_0^1
              \int_0^1 \theta
     \mathbb{E}
     \left[\int_{t^n_{k-1}}^{t^n_k}
          \Tr \left(
        G_r G_r^* \partial _x\partial_\mu
       \hat\varphi(t^n_k,\hat\P_{X^{n,\theta,k}})(X^{n,\epsilon \theta,k-1})\right)dr
   \right] d\epsilon d\theta\\
   &
=
\lim_{n\rightarrow \infty}
                \int_0^1
              \int_0^1 \theta
     \mathbb{E}
     \left[\int_t^s \sum_{k=1}^n
          \Tr \left(
        G_r G_r^* \partial _x\partial_\mu
       \hat\varphi(t^n_k,\hat\P_{X^{n,\theta,k}})(X^{n,\epsilon \theta,k-1})\right)\mathds{1}_{[t^n_{k-1},t^n_k)}(r)dr
 \right] d\epsilon d\theta.
  \end{split}
\end{equation}
Moreover, by continuity of $  \partial _x\partial _\mu \hat{\varphi}$ and by
\eqref{2020-12-04:05}, we have
\begin{equation*}
  \lim_{n\rightarrow \infty}
  \left|
  \sum_{k=1}^n
 \partial _x  \partial_\mu \hat \varphi \left( t^n_k,
    \hat{ \mathbb{P}}_{X^{n,\theta,k}}
  \right)
  \left(X^{n,\epsilon\theta,k}(\omega)
  \right)\mathds{1}_{[t^n_{k-1},t^n_k)}(r)
  -
     \partial _x\partial_\mu \hat \varphi \left( r,
    \hat{ \mathbb{P}}_X
  \right)
  \left(X(\omega)\right)
  \right|_H=0\qquad \forall \omega\in\Omega.
\end{equation*}
This, together with \eqref{2020-12-04:11}, provides
\begin{equation}
  \label{2020-12-04:12}
    \lim_{n\rightarrow \infty}\sum_{k=1}^n
    \mathbf{I}^{n,k}_{BB}
    =
    \frac{1}{2}\int_t^s\mathbb{E}
 \left[
   \Tr \left(
        G_r G_r^* \partial _x\partial_\mu
       \hat\varphi(r,\hat\P_{X_{\cdot\wedge r}})(X_{\cdot\wedge r})
   \right)
 \right]dr.
\end{equation}

We now consider   the quantity $\sum_{k=1}^n\hat{\varphi}(t^n_k,\hat{ \mathbb{P}}_{X^n_{t^n_{k-1}\wedge \cdot}})
-\hat{\varphi}(t^n_{k-1},\hat{\mathbb{P}}_{X^n})$ appearing
in \eqref{2020-11-13:01}.
By using the
continuity of the pathwise time derivative $ \partial _t\hat{ \varphi}$, we can write
\begin{equation}\label{2020-12-04:13}
  \begin{split}
    \sum_{k=1}^n
 \left(       \hat{\varphi}(t^n_k,\hat{ \mathbb{P}}_{X^n_{t^n_{k-1}\wedge \cdot}})
   -\hat{\varphi}(t^n_{k-1},\hat{\mathbb{P}}_{X^n_{t^n_{k-1}\wedge \cdot}}) \right) =&
\sum_{k=1}^n   \int_{t^n_{k-1}}^{t^n_k}
 \partial _t   \hat{\varphi}(r,\hat{\mathbb{P}}_{X^n_{t^n_{k-1}\wedge \cdot}})dr \\
 =&
 \int_t^s
 \sum_{k=1}^n\partial _t   \hat{\varphi}(r,\hat{\mathbb{P}}_{X^n_{t^n_{k-1}\wedge \cdot}})\mathds{1}_{[t^n_{k-1},t^n_k)}(r)dr
\end{split}
\end{equation}
By continuity of $  \partial_t\hat{\varphi}$ and by
\eqref{2020-12-04:05}, we have
\begin{equation}\label{2020-12-04:14}
  \lim_{n\rightarrow \infty}
  \left|
    \sum_{k=1}^n\partial _t   \hat{\varphi}(r,\hat{\mathbb{P}}_{X^n_{t^n_{k-1}\wedge \cdot}})\mathds{1}_{[t^n_{k-1},t^n_k)}(r)
    -
    \partial _t   \hat{\varphi}
    (r,\hat{\mathbb{P}}_X)
    \right|=0.
  \end{equation}
  By
\eqref{2020-12-04:13},\eqref{2020-12-04:14}, and by Lebesgue's dominated convergence theorem, we obtain
\begin{equation}
  \label{2020-12-04:15}
\lim_{n\rightarrow \infty}    \sum_{k=1}^n
 \left(       \hat{\varphi}(t^n_k,\hat{ \mathbb{P}}_{X^n_{t^n_{k-1}\wedge \cdot}})
   -\hat{\varphi}(t^n_{k-1},\hat{\mathbb{P}}_{X^n_{t^n_{k-1}\wedge \cdot}}) \right)
 =\int_t^s  \partial _t\hat{ \varphi}(r,\hat{\mathbb{P}}_X)dr.
\end{equation}
Putting together
\eqref{2020-11-13:01},
\eqref{2020-12-04:17},
\eqref{2020-12-04:08},
\eqref{2020-12-04:16},
\eqref{2020-12-04:09},
\eqref{2020-12-04:12},
\eqref{2020-12-04:15},
we finally obtain, by recalling also \eqref{2020-12-04:05},
\begin{equation*}
  \begin{split}
      \hat{\varphi}(s,\hat{\mathbb{P}}_X)-
  \hat{\varphi}(t,\hat{\mathbb{P}}_X)=&
  \lim_{n\rightarrow \infty}
 \left(     \hat{\varphi}(s,\hat{\mathbb{P}}_{X^n})-
   \hat{\varphi}(t,\hat{\mathbb{P}}_{X^n}) \right) \\
 =&
 \int_t^s \partial_t \hat\varphi(r,\hat\P_{X_{\cdot\wedge r}})\,dr + \int_t^s
\mathbb{E}
 \left[   \langle F_r,\partial_\mu \hat\varphi(r,\hat\P_{X_{\cdot\wedge r}})(X_{\cdot\wedge r})\rangle_H
 \right]dr \\
& + \frac{1}{2}\int_t^s\mathbb{E}
 \left[
   \Tr \left(
        G_r G_r^* \partial _x\partial_\mu
       \hat\varphi(r,\hat\P_{X_{\cdot\wedge r}})(X_{\cdot\wedge r})
   \right)
 \right]dr,
\end{split}
\end{equation*}
which concludes the proof.
\end{proof}

\section{Consistency property of pathwise derivatives}
\label{App:Consistency}

\begin{proof}[\textbf{Proof of Lemma \ref{L:Consistency}.}]
The consistency of pathwise time derivatives is a direct consequence of their definition (Definition \ref{Def:HorizontalDer}).

We now prove the consistency of the other two derivatives using It\^o's formula (Theorem \ref{T:ItoD}). Fix $t\in[0,T]$ and let $\xi\in\S_2(\F)$. Let also $F\colon[0,T]\times\Omega\rightarrow H$ (resp.\ $G\colon[0,T]\times\Omega\rightarrow\mathcal{L}_2(K;H)$) be a integrable (resp.\ square-integrable) and $\F$-progressively measurable process, so, in particular,
\[
\int_0^T \E[|F_s|_H]\,ds < \infty, \qquad\qquad \int_0^T \E\big[\textup{Tr}(G_sG_s^*)\big]\,ds < \infty.
\]
Consider the process $X=(X_s)_{s\in[0,T]}$ given by
\begin{equation*}
X_s = \xi_{s\wedge t} + \int_t^{s\vee t} F_r\,dr + \int_t^{s\vee t} G_r\,dB_r, \qquad \forall\,s\in[0,T].
\end{equation*}
Then, by Theorem \ref{T:ItoD} {and Remark \ref{R:Sym}} we have the It\^o formulae, for every $s\in[t,T]$,
\begin{align*}
  \hat\varphi_1(s,\hat\P_{X_{\cdot\wedge s}}) &= \hat\varphi_1(t,\hat\P_{\xi_{\cdot\wedge t}}) + \int_t^s \partial_t \hat\varphi_1(r,\hat\P_{X_{\cdot\wedge r}})\,dr + \int_t^s
\mathbb{E}
 \left[   \langle F_r,\partial_\mu \hat\varphi_1(r,\hat\P_{X_{\cdot\wedge r}})(X_{\cdot\wedge r})\rangle_H
 \right]dr \\
&\quad + \frac{1}{2}\int_t^s\mathbb{E}
 \left[
   \Tr \left(
        G_r G_r^* \partial _x^{{\textup{sym}}}\partial_\mu
       \hat\varphi_1(r,\hat\P_{X_{\cdot\wedge r}})(X_{\cdot\wedge r})
   \right)
 \right]dr
\end{align*}
and
\begin{align*}
  \hat\varphi_2(s,\hat\P_{X_{\cdot\wedge s}}) &= \hat\varphi_2(t,\hat\P_{\xi_{\cdot\wedge t}}) + \int_t^s \partial_t \hat\varphi_2(r,\hat\P_{X_{\cdot\wedge r}})\,dr + \int_t^s
\mathbb{E}
 \left[   \langle F_r,\partial_\mu \hat\varphi_2(r,\hat\P_{X_{\cdot\wedge r}})(X_{\cdot\wedge r})\rangle_H
 \right]dr \\
&\quad + \frac{1}{2}\int_t^s\mathbb{E}
 \left[
   \Tr \left(
        G_r G_r^* \partial _x^{{\textup{sym}}}\partial_\mu
       \hat\varphi_2(r,\hat\P_{X_{\cdot\wedge r}})(X_{\cdot\wedge r})
   \right)
 \right]dr.
\end{align*}
Since $\hat\varphi_1(r,\hat\P_{X_{\cdot\wedge r}})=\hat\varphi_2(r,\hat\P_{X_{\cdot\wedge r}})$ and $\partial_t\hat\varphi_1(r,\hat\P_{X_{\cdot\wedge r}})=\partial_t\hat\varphi_2(r,\hat\P_{X_{\cdot\wedge r}})$, for every $r\in[0,T]$, we find
\begin{align}\label{FGG*=0}
&\int_t^s
\mathbb{E}
 \left[   \big\langle F_r,\big(\partial_\mu \hat\varphi_1(r,\hat\P_{X_{\cdot\wedge r}})(X_{\cdot\wedge r}) - \partial_\mu \hat\varphi_2(r,\hat\P_{X_{\cdot\wedge r}})(X_{\cdot\wedge r})\big)\big\rangle_H
 \right]dr \\
&+ \frac{1}{2}\int_t^s\mathbb{E}
 \left[
   \Tr \left(
        G_r G_r^* \big(\partial _x^{{\textup{sym}}}\partial_\mu
       \hat\varphi_1(r,\P_{X_{\cdot\wedge r}})(X_{\cdot\wedge r}) - \partial _x^{{\textup{sym}}}\partial_\mu
       \hat\varphi_2(r,\P_{X_{\cdot\wedge r}})(X_{\cdot\wedge r})\big)
   \right)
 \right]dr = 0, \notag
\end{align}
for every $s\in[t,T]$.

\vspace{1mm}

\noindent\emph{Consistency of $\partial_\mu \hat\varphi_1$ and $\partial_\mu \hat\varphi_2$.} Let $Z\colon\Omega\rightarrow H$ be an $\Fc_t$-measurable random variable in $L^2(\Omega;H)$, and define
\[
F_s \coloneqq  Z\,\mathds 1_{[t,T]}(s), \qquad G_s \coloneqq  0, \qquad \forall\,s\in[0,T].
\]
Then, from \eqref{FGG*=0} we obtain
\begin{equation}\label{Consistency1}
\int_t^s
\mathbb{E}
 \left[   \big\langle Z,\big(\partial_\mu \hat\varphi_1(r,\hat\P_{X_{\cdot\wedge r}})(X_{\cdot\wedge r}) - \partial_\mu \hat\varphi_2(r,\hat\P_{X_{\cdot\wedge r}})(X_{\cdot\wedge r})\big)\big\rangle_H
 \right]dr = 0, \qquad \forall\,s\in[t,T].
\end{equation}
From the continuity of the map $r\mapsto\partial_\mu
       \hat\varphi_1(r,\hat\P_{X_{\cdot\wedge r}})(X_{\cdot\wedge r}) - \partial_\mu
       \hat\varphi_2(r,\hat\P_{X_{\cdot\wedge r}})(X_{\cdot\wedge r})$, we get in particular that the integrand of \eqref{Consistency1} is equal to zero at $r=t$, namely
\[
\mathbb{E}
 \left[   \big\langle Z,\big(\partial_\mu \hat\varphi_1(t,\hat\P_{\xi_{\cdot\wedge t}})(\xi_{\cdot\wedge t}) - \partial_\mu \hat\varphi_2(t,\hat\P_{\xi_{\cdot\wedge t}})(\xi_{\cdot\wedge t})\big)\big\rangle_H
 \right] = 0.
\]
From the arbitrariness of the $\Fc_t$-measurable random variable $Z\in L^2(\Omega;H)$, this yields
\[
\partial_\mu
       \hat\varphi_1(t,\hat\P_{\xi_{\cdot\wedge t}})(\xi_{\cdot\wedge t}) = \partial_\mu
       \hat\varphi_2(t,\hat\P_{\xi_{\cdot\wedge t}})(\xi_{\cdot\wedge t}), \qquad \P\text{-a.s.}
\]
Recalling Remark \ref{R:NA-Consistency}, we see that this proves the consistency of $\partial_\mu \hat\varphi_1$ and $\partial_\mu \hat\varphi_2$.

\vspace{1mm}

\noindent\emph{Consistency of $\partial_x^{{\textup{sym}}}\partial_\mu \hat\varphi_1$ and $\partial_x^{{\textup{sym}}}\partial_\mu \hat\varphi_2$.} Let $\Lambda\colon\Omega\rightarrow\Lc_2(K;H)$ be an $\Fc_t$-measurable random variable in $L^2(\Omega;\Lc_2(K;H))$, and define
\[
F_s \coloneqq  0, \qquad G_s \coloneqq  \Lambda\,\mathds 1_{[t,T]}(s), \qquad \forall\,s\in[0,T].
\]
Then, from \eqref{FGG*=0} we obtain
\[
\int_t^s\mathbb{E}
 \left[
   \Tr \left(
        \Lambda\Lambda^* \big(\partial _x^{{\textup{sym}}}\partial_\mu
       \hat\varphi_1(r,\hat\P_{X_{\cdot\wedge r}})(X_{\cdot\wedge r}) - \partial _x^{{\textup{sym}}}\partial_\mu
       \hat\varphi_2(r,\hat\P_{X_{\cdot\wedge r}})(X_{\cdot\wedge r})\big)
   \right)
 \right]dr = 0, \quad \forall\,s\in[t,T].
\]
From the continuity of the map $r\mapsto\partial_x^{{\textup{sym}}}\partial_\mu
       \hat\varphi_1(r,\hat\P_{X_{\cdot\wedge r}})(X_{\cdot\wedge r}) - \partial_x^{{\textup{sym}}}\partial_\mu
       \hat\varphi_2(r,\hat\P_{X_{\cdot\wedge r}})(X_{\cdot\wedge r})$, we get
\[
\mathbb{E}
 \left[
   \Tr \left(
        \Lambda\Lambda^* \big(\partial_x^{{\textup{sym}}}\partial_\mu
       \hat\varphi_1(t,\hat\P_{\xi_{\cdot\wedge t}})(\xi_{\cdot\wedge t}) - \partial_x^{{\textup{sym}}}\partial_\mu
       \hat\varphi_2(t,\hat\P_{\xi_{\cdot\wedge t}})(\xi_{\cdot\wedge t})\big)
   \right)
 \right] = 0.
\]
{From the arbitrariness of the $\Fc_t$-measurable random variable $\Lambda\in L^2(\Omega;\Lc_2(K;H))$, we conclude that}
\[
\partial_x^{{\textup{sym}}}\partial_\mu
       \hat\varphi_1(t,\hat\P_{\xi_{\cdot\wedge t}})(\xi_{\cdot\wedge t}) = \partial_x^{{\textup{sym}}}\partial_\mu
       \hat\varphi_2(t,\hat\P_{\xi_{\cdot\wedge t}})(\xi_{\cdot\wedge t}), \qquad \P\text{-a.s.}
\]
which, together with Remark \ref{R:NA-Consistency}, gives the claimed consistency of $\partial_x^{{\textup{sym}}}\partial_\mu\hat\varphi_1$ and $\partial_x^{{\textup{sym}}}\partial_\mu\hat\varphi_2$.
\end{proof}

\section{Hamilton-Jacobi-Bellman equation: technical results}
\label{App:HJB}

In the present appendix we prove three technical results which are used in Section \ref{S:HJB} to derive alternative forms of the Hamilton-Jacobi-Bellman equation \eqref{HJB}.

\begin{Lemma}\label{L:xi_U_indep}
Given $\mu\in\Pc_2(C([0,T];H))$ there exists $\xi\in\S_2(\Gc)$, with $\P_\xi=\mu$, and a $\Gc$-measurable random variable $U_\xi\colon\Omega\rightarrow\R$, with uniform distribution on $[0,1]$, such that the following holds:
\begin{itemize}
\item $\xi$ and $U_\xi$ are independent.
\end{itemize}
\end{Lemma}
\begin{proof}[\textbf{Proof.}]
Fix $\mu\in\Pc_2(C([0,T];H))$ and consider the probability space $([0,1],\Bc([0,1]),\lambda)$, where $\lambda$ denotes the Lebesgue measure on the unit interval. Denote
\[
\bar\Omega = [0,1]\times C([0,T];H), \qquad \bar{\mathcal F} = \mathcal B([0,1])\otimes\mathscr B, \qquad \bar\P = \lambda\otimes\mu.
\]
Fix an orthonormal basis $\{e_n\}_{n\in\N}$ of $H$. Then, let $\mathcal J\colon[0,1]\times C([0,T];H)\rightarrow C([0,T];H)$ be the map defined as
\[
\mathcal J(r,x) \coloneqq  re_1 + \sum_{n=2}^\infty \langle x,e_{n-1}\rangle_H\,e_n, \qquad \forall\,(r,x)\in[0,1]\times C([0,T];H).
\]
Let $\tilde\mu$ denote the image measure (or push-forward) of $\bar\P=\lambda\otimes\mu$ by $\mathcal J$. Notice that $\tilde\mu\in\Pc_2(C([0,T];H))$. Then, from property \textup{\ref{A_G}}-ii) it follows that there exists a continuous and $\Bc([0,T])\otimes\Gc$-measurable process $\tilde\xi\colon[0,T]\times\Omega\rightarrow H$ with law equal to $\tilde\mu$.

Now, define the maps $P\colon C([0,T];H)\rightarrow\R$ and $Q\colon C([0,T];H)\rightarrow C([0,T];H)$ as
\[
P(x) \coloneqq  \langle x_0,e_1\rangle_H, \qquad Q(x) \coloneqq  \sum_{n=1}^\infty \langle x,e_{n+1}\rangle_H\,e_n, \qquad\qquad \forall\,x\in C([0,T];H),
\]
where we recall that $x_0$ is the value of the path $x$ at time $t=0$. Then, denote
\[
U_\xi \coloneqq  P(\tilde\xi), \qquad\qquad \xi  \coloneqq  Q(\tilde\xi).
\]
It is then easy to see that $U_\xi$ and $\xi$ satisfy the claimed properties.
\end{proof}

\begin{Lemma}\label{L:Formula_F}
Fix $t\in[0,T]$, $\xi\in\S_2(\Gc)$ and let $F\colon C([0,T];H)\times\Pc_2(C([0,T];H))\times\Ur\times\Pc(\Ur)\rightarrow\R$ be a measurable function. Suppose that $\E[|F(\xi,\P_\xi,{\mathfrak a},\P_{{\mathfrak a}})|]<+\infty$, $\forall\,{\mathfrak a}\in\Mc_t$. Suppose also that
\begin{align}\label{eq:indepF2}
&\hbox{there exists a $\Gc$-measurable random variable $U_\xi\colon\Omega\rightarrow\R$}
\\
\nonumber
&\hbox{having uniform distribution on $[0,1]$ and being independent of $\xi$.}
\end{align}
Then, it holds that
\begin{equation}\label{Formula_F}
\sup_{{\mathfrak a}\in\Mc_t} \E\big[F\big(\xi,\P_\xi,{\mathfrak a},\P_{\mathfrak a}\big)\big] = \sup_{{\mathfrak a}\in\Mc_\Gc} \E\big[F\big(\xi,\P_\xi,{\mathfrak a},\P_{\mathfrak a}\big)\big] = \sup_{\mathrm{\check a}\in\check\Mc} \E\big[F\big(\xi,\P_\xi,\mathrm{\check a}(\xi,U_\xi),\P_{\mathrm{\check a}(\xi,U_\xi)}\big)\big].
\end{equation}
Moreover, if $F$ is continuous in the first two variables, uniformly with respect to the last two, then the first equality holds true without assuming \eqref{eq:indepF2}.
\end{Lemma}
\begin{proof}[\textbf{Proof.}]
Since $\mathrm{\check a}(\xi,U_\xi)\in\Mc_\Gc$ and $\Mc_\Gc\subset\Mc_t$, we immediately get
\[
\sup_{{\mathfrak a}\in\Mc_t} \E\big[F\big(\xi,\P_\xi,{\mathfrak a},\P_{\mathfrak a}\big)\big] \geq \sup_{{\mathfrak a}\in\Mc_\Gc} \E\big[F\big(\xi,\P_\xi,{\mathfrak a},\P_{\mathfrak a}\big)\big] \geq \sup_{\mathrm{\check a}\in\check\Mc} \E\big[F\big(\xi,\P_\xi,\mathrm{\check a}(\xi,U_\xi),\P_{\mathrm{\check a}(\xi,U_\xi)}\big)\big].
\]
Then, in order to get \eqref{Formula_F}, it remains to prove the inequality
\begin{equation}\label{Inequality_F}
\sup_{\mathrm{\check a}\in\check\Mc} \E\big[F\big(\xi,\P_\xi,\mathrm{\check a}(\xi,U_\xi),\P_{\mathrm{\check a}(\xi,U_\xi)}\big)\big] \geq \sup_{{\mathfrak a}\in\Mc_t} \E\big[F\big(\xi,\P_\xi,{\mathfrak a},\P_{\mathfrak a}\big)\big].
\end{equation}
To this end, denote
\[
\hat\Omega = [0,T]\times\Omega, \qquad \hat{\mathcal F} = \mathcal B([0,T])\otimes\mathcal F, \qquad \hat\P = \lambda_T\otimes\P,
\]
with $\lambda_T$ being the uniform distribution on $([0,T],\mathcal B([0,T]))$. Given ${\mathfrak a}\in\Mc_t$, consider the canonical extensions of ${\mathfrak a}$ and $U_\xi$ to $\hat\Omega$, which will be denoted respectively by $\hat{\mathfrak a}$ and $\hat U_\xi$ (notice that $\hat U_\xi$ has uniform distribution on $[0,1]$ and is independent of $\xi$). We can now apply Lemma \ref{L:Kall}, with $(E,\mathscr E)$, $\Gamma$, $\hat U$ being respectively $(C([0,T];H),\mathscr B)$, $\xi$, $\hat U_\xi$. Then, it follows the existence of a measurable map $\mathrm{\check a}\colon C([0,T];H)\times[0,1]\rightarrow\Ur$ such that
\[
\big(\xi,\mathrm{\check a}(\xi,\hat U_\xi)\big) \overset{\mathscr L_{\hat\Omega}}{=} \big(\xi,\hat{\mathfrak a}\big),
\]
where $\mathscr L_{\hat\Omega}$ stands for equality in law between random objects defined on $(\hat\Omega,\hat\Fc,\hat\P)$. Then, we deduce
\[
\Big((\xi_s)_{s\in[0,T]},\mathrm{\check a}(\xi,U_\xi)\Big) \overset{\mathscr L}{=} \Big((\xi_s)_{s\in[0,T]},{\mathfrak a}\Big),
\]
where $\mathscr L$ stands for equality in law between random objects defined on $(\Omega,\Fc,\P)$. This implies the validity of inequality \eqref{Inequality_F} and proves \eqref{Formula_F}.

Finally, suppose that $F$ is continuous in the first two variables, uniformly with respect to the last two. Then, in order to get \eqref{Formula_F}, it is enough to proceed as in Step 2 of the proof of Theorem \ref{T:id-law}. More precisely, if $\xi$ is discrete, then by Lemma \ref{L:ExistUnif} there exists an $\Fc_t$-measurable random variable $U_\xi$ having uniform distribution on $[0,1]$ and being independent of $\xi$. In the general case, we rely on the continuity of $F$ and we approximate $\xi$ by a sequence of discrete random variables.
\end{proof}

\begin{Lemma}\label{L:Formula_F_bis}
Fix $t\in[0,T]$, $\xi\in\S_2(\Gc)$ and let $F\colon C([0,T];H)\times\Ur\rightarrow\R$ be a measurable function. Suppose that $\E[|F(\xi,{\mathfrak a})|]<+\infty$, $\forall\,{\mathfrak a}\in\Mc_t$, and also that $\sup_{u\in\Ur}F(x,u)<+\infty$, $\forall\,x\in C([0,T];H)$. Then, it holds that
\begin{equation}\label{Formula_F_bis}
\sup_{{\mathfrak a}\in\Mc_t} \E\big[F(\xi,{\mathfrak a})\big] = \sup_{\mathrm a\in\Mc} \E\big[F(\xi,\mathrm a(\xi))\big]
\end{equation}
and
\begin{equation}\label{Formula_F_ter}
\sup_{{\mathfrak a}\in\Mc_t} \E\big[F(\xi,{\mathfrak a})\big] = \E\Big[\esssup_{u\in\Ur}F(\xi,u)\Big].
\end{equation}
\end{Lemma}
\begin{Remark}
Suppose that $\xi$ and $U_\xi$ are as in Lemma \ref{L:xi_U_indep}, namely $\xi\in\S_2(\Gc)$, $\P_\xi=\mu$, and $\xi$ is such that there exists a $\Gc$-measurable random variable $U_\xi$ having uniform distribution on $[0,1]$ and being independent of $\xi$. Then, formula \eqref{Formula_F_bis} holds without assuming that $\sup_{u\in\Ur}F(x,u)<+\infty$, $\forall\,x\in C([0,T];H)$. As a matter of fact, in this case, thanks to formula \eqref{Formula_F}, it is enough to prove that
\[
\sup_{\mathrm{\check a}\in\check\Mc} \E\big[F(\xi,\mathrm{\check a}(\xi,U_\xi))\big] = \sup_{\mathrm a\in\Mc} \E\big[F(\xi,\mathrm a(\xi))\big].
\]
Clearly, it holds that $\sup_{\mathrm{\check a}\in\check\Mc} \E[F(\xi,\mathrm{\check a}(\xi,U_\xi))]\geq\sup_{\mathrm a\in\Mc} \E[F(\xi,\mathrm a(\xi))]$. On the other hand, for every fixed $\mathrm{\check a}\in\check\Mc$ we have
\[
\E\big[F(\xi,\mathrm{\check a}(\xi,U_\xi))\big] = \int_0^1\E\big[F(\xi,\mathrm{\check a}(\xi,r))\big]\,dr \leq \sup_{\mathrm a\in\Mc} \E\big[F(\xi,\mathrm a(\xi))\big].
\]
From the arbitrariness of $\mathrm{\check a}$ we get the other inequality and we conclude that \eqref{Formula_F_bis} holds.
\end{Remark}
\begin{proof}[\textbf{Proof of Lemma \ref{L:Formula_F_bis}.}]
Since $\mathrm a(\xi)\in\Mc_t$, we immediately get
\[
\sup_{\mathrm a\in\Mc} \E\big[F(\xi,\mathrm a(\xi))\big] \leq \sup_{{\mathfrak a}\in\Mc_t} \E\big[F(\xi,{\mathfrak a})\big].
\]
In addition, for every fixed $\mathfrak a\in\Mc_t$ we have
\[
\E\big[F\big(\xi,{\mathfrak a}\big)\big] \leq \E\Big[\esssup_{u\in\Ur}F(\xi,u)\Big].
\]
From the arbitrariness of $\mathfrak a$, we find $\sup_{{\mathfrak a}\in\Mc_t} \E[F(\xi,{\mathfrak a})]\leq\E[\esssup_{u\in\Ur}F(\xi,u)]$. Then, in order to prove the validity of both \eqref{Formula_F_bis} and \eqref{Formula_F_ter}, it remains to prove that
\begin{equation}\label{E[esssupF]<=supE[F]}
\E\Big[\esssup_{u\in\Ur}F(\xi,u)\Big] \leq \sup_{\mathrm a\in\Mc} \E\big[F(\xi,\mathrm a(\xi))\big].
\end{equation}
Let $\mu=\P_\xi$ denote the law of $\xi$. Suppose for a moment that for every $\eps>0$ there exists $\mathrm a_\mu^\eps\in\Mc$ and a $\mu$-null set $N_\eps\in\mathscr B$ such that
\begin{equation}\label{F_a_mu^eps}
F(x,u) \leq F(x,\mathrm a_\mu^\eps(x)) + \eps, \qquad \forall\,x\in C([0,T];H)\backslash N_\eps,\,\forall\,u\in\Ur.
\end{equation}
Then, in particular, there exists a $\P$-null set $\bar N_\eps\in\Fc$ such that
\[
F(\xi(\omega),u) \leq F(\xi(\omega),\mathrm a_\mu^\eps(\xi(\omega))) + \eps, \qquad \forall\,\omega\in\Omega\backslash\bar N_\eps,\,\forall\,u\in\Ur.
\]
As a consequence, using the definition of essential supremum for the family of real-valued random variables $\{F(\xi,u)\}_{u\in\Ur}$, we find
\[
\esssup_{u\in\Ur} F(\xi,u) \leq F(\xi,\mathrm a_\mu^\eps(\xi)) + \eps, \qquad \P\text{-a.s.}
\]
This yields
\[
\E\Big[\esssup_{u\in\Ur}F(\xi,u)\Big] \leq \E\big[F(\xi,\mathrm a_\mu^\eps(\xi))\big] + \eps \leq \sup_{\mathrm a\in\Mc} \E\big[F(\xi,\mathrm a(\xi))\big] + \eps.
\]
From the arbitrariness of $\eps$ we get inequality \eqref{E[esssupF]<=supE[F]} and we conclude that \eqref{Formula_F_bis} holds.

It remains to prove \eqref{F_a_mu^eps}. To this end, we use that $\Ur$ is a Borel space and we implement Proposition 7.50 in \cite{BertsekasShreve} (in particular, $X$, $Y$, $D$, $f$ in the statement of this latter proposition are given respectively by $C([0,T];H)$, $\Ur$, $C([0,T];H)\times\Ur$, $F$). By Proposition 7.50 in \cite{BertsekasShreve} it follows that, for every $\eps>0$, there exists an analytically measurable function (we refer to Definition 7.20 in \cite{BertsekasShreve} for the definition of \emph{analytically measurable} function) $\mathrm a^\eps\colon C([0,T];H)\rightarrow\Ur$ such that
\[
F(x,\mathrm a^\eps(x)) \geq \
\begin{cases}
\sup_{u\in\Ur} F(x,u) - \eps, \qquad &\text{if } \sup_{u\in\Ur} F(x,u)\,<\,+\infty, \\
1/\eps, &\text{if } \sup_{u\in\Ur} F(x,u)\,=\,+\infty,
\end{cases}
\]
for all $x\in C([0,T];H)$. Since it holds that $\sup_{u\in\Ur}F(x,u)<+\infty$, $\forall\,x\in C([0,T];H)$, we can rewrite the above inequality simply as
\begin{equation}\label{F_a^eps}
\sup_{u\in\Ur} F(x,u) \leq F(x,\mathrm a^\eps(x)) + \eps, \qquad \forall\,x\in C([0,T];H).
\end{equation}
Using Lemma 7.27 in \cite{BertsekasShreve} and the fact that $\Ur$ is Borel-isomorphic to a Borel subset of $[0,1]$, we see that there exists a Borel-measurable function $\mathrm a_\mu^\eps\colon C([0,T];H)\rightarrow\Ur$ such that $\mathrm a^\eps(x)=\mathrm a_\mu^\eps(x)$ for $\mu$-a.e. $x\in C([0,T];H)$. Hence, from \eqref{F_a^eps} we obtain \eqref{F_a_mu^eps}.
\end{proof}

\small

\bibliographystyle{plain}
\bibliography{BibCGKPRbis}

\end{document}